\documentclass[a4paper]{amsart}
\pdfoutput=1
\usepackage[utf8]{inputenc}
\usepackage[T1]{fontenc}
\usepackage{lmodern}
\usepackage{amssymb,amsxtra}
\usepackage{nicefrac,mathtools}
\usepackage[lite]{amsrefs}

\renewcommand{\PrintDOI}[1]{\href{http://dx.doi.org/\detokenize{#1}}{doi: \detokenize{#1}}}

\BibSpec{book}{%
  +{}  {\PrintPrimary}                {transition}
  +{,} { \textit}                     {title}
  +{.} { }                            {part}
  +{:} { \textit}                     {subtitle}
  +{,} { \PrintEdition}               {edition}
  +{}  { \PrintEditorsB}              {editor}
  +{,} { \PrintTranslatorsC}          {translator}
  +{,} { \PrintContributions}         {contribution}
  +{,} { }                            {series}
  +{,} { \voltext}                    {volume}
  +{,} { }                            {publisher}
  +{,} { }                            {organization}
  +{,} { }                            {address}
  +{,} { \PrintDateB}                 {date}
  +{,} { }                            {status}
  +{}  { \parenthesize}               {language}
  +{}  { \PrintTranslation}           {translation}
  +{;} { \PrintReprint}               {reprint}
  +{.} { }                            {note}
  +{.} {}                             {transition}
  +{} { \PrintDOI}                   {doi}
  +{} { available at \url}            {eprint}
  +{}  {\SentenceSpace \PrintReviews} {review}
}

\usepackage{enumitem} 
\setlist[enumerate,1]{label=\textup{(\arabic*)}}

\usepackage{tikz}
\usetikzlibrary{matrix,calc,arrows,decorations.pathmorphing}
\tikzset{node distance=2cm, auto}

\tikzset{cd/.style=matrix of math nodes,row sep=2em,column sep=2em, text height=1.5ex, text depth=0.5ex}
\tikzset{cdar/.style=->,auto}
\tikzset{mid/.style={anchor=mid}} 
\tikzset{narrowfill/.style={inner sep=1pt, fill=white}}

\usepackage{microtype}
\usepackage[pdftitle={Reduced C*-algebras of Fell bundles over inverse semigroups},
  pdfauthor={Alcides Buss, Ruy Exel, Ralf Meyer},
  pdfsubject={Mathematics}
]{hyperref}

\numberwithin{equation}{section}
\theoremstyle{plain}
\newtheorem{theorem}[equation]{Theorem}
\newtheorem{lemma}[equation]{Lemma}
\newtheorem{proposition}[equation]{Proposition}
\newtheorem{corollary}[equation]{Corollary}
\theoremstyle{definition}
\newtheorem{definition}[equation]{Definition}
\theoremstyle{remark}
\newtheorem{remark}[equation]{Remark}
\newtheorem{example}[equation]{Example}
\newtheorem{notation}[equation]{Notation}

\DeclareMathOperator{\Ind}{Ind}
\DeclareMathOperator{\Aut}{Aut}
\DeclareMathOperator{\cspn}{\overline{span}}

\DeclareMathOperator{\Prim}{Prim}

\newcommand*{\nb}{\nobreakdash}
\newcommand*{\Star}{\(^*\)\nobreakdash-}

\newcommand*{\C}{\mathbb C}
\newcommand*{\Z}{\mathbb Z}

\newcommand*{\N}{\mathbb N}

\newcommand*{\Bound}{\mathbb B}
\newcommand*{\Comp}{\mathbb K}
\newcommand*{\Mat}{\mathbb M}

\newcommand*{\red}{\mathrm r}
\newcommand*{\Cred}{\mathrm{C^*_r}} 
\newcommand*{\alg}{\mathrm{alg}} 

\newcommand*{\Sect}{\mathfrak S}

\newcommand*{\cstar}{\texorpdfstring{\(C^*\)\nobreakdash-\hspace{0pt}}{*-}}
\newcommand*{\Cst}{\textup C^*}
\newcommand*{\Wst}{\textup W^*}
\newcommand*{\Mult}{\mathcal M}

\newcommand*{\Cont}{\textup C}
\newcommand*{\Contb}{\textup C_\textup b}
\newcommand*{\Contc}{\Cont_\textup c} 
\newcommand*{\Borel}{\textup B}

\newcommand*{\prid}{\mathfrak{p}} 

\newcommand*{\bid}[1]{#1''}
\newcommand*{\barotimes}{\mathbin{\bar\otimes}}
\newcommand*{\barrtimes}{\mathbin{\bar\rtimes}}
\newcommand{\idealin}{\mathrel{\triangleleft}} 

\newcommand*{\Id}{\textup{Id}}
\newcommand*{\Ad}{\textup{Ad}}

\newcommand*{\Hils}{\mathcal H}
\newcommand*{\Hilm}[1][H]{\mathcal #1}

\newcommand*{\B}{\mathcal B}

\newcommand*{\defeq}{\mathrel{\vcentcolon=}}
\newcommand*{\congto}{\xrightarrow\sim}

\newcommand*{\norm}[1]{\lVert#1\rVert}
\newcommand*{\abs}[1]{\lvert#1\rvert}
\newcommand*{\braket}[2]{\langle#1{\mid}#2\rangle}
\newcommand*{\BRAKET}[2]{\langle\!\langle#1{\mid}#2\rangle\!\rangle}

\newcommand*{\s}{s}
\newcommand*{\rg}{r}

\DeclareMathOperator*{\stlim}{s-lim}

\newcommand*{\into}{\rightarrowtail}
\newcommand*{\onto}{\twoheadrightarrow}

\begin{document}
\title[Reduced C*-algebras of Fell bundles over inverse semigroups]{Reduced C*-algebras of Fell bundles\\over inverse semigroups}

\author{Alcides Buss}
\email{alcides@mtm.ufsc.br}

\author{Ruy Exel}
\email{exel@mtm.ufsc.br}

\address{Departamento de Matem\'atica\\
 Universidade Federal de Santa Catarina\\
 88.040-900 Florian\'opolis-SC\\
 Brazil}

\author{Ralf Meyer}
\email{rmeyer2@uni-goettingen.de}

\address{Mathematisches Institut\\
 Georg-August-Universit\"at G\"ottingen\\
 Bunsenstra\ss e 3--5\\
 37073 G\"ottingen\\
 Germany}

\begin{abstract}
  We construct a weak conditional expectation from the section
  C*\nb-algebra of a Fell bundle over a unital inverse semigroup to
  its unit fibre.  We use this to define the reduced C*-algebra of the
  Fell bundle.  We study when the reduced C*-algebra for an inverse
  semigroup action on a groupoid by partial equivalences coincides
  with the reduced groupoid C*-algebra of the transformation groupoid,
  giving both positive results and counterexamples.
\end{abstract}
\thanks{Supported by CNPq/CsF (Brazil).}
\maketitle

\section{Introduction}
\label{sec:introduction}

Let~\(S\)
be a unital inverse semigroup.  It may act on a space~\(X\)
by partial homeomorphisms, that is, homeomorphisms \(U\congto V\)
for open subsets \(U,V\)
in~\(X\).
This induces an \(S\)\nb-action
on the \(\Cst\)\nb-algebra~\(\Cont_0(X)\)
by partial isomorphisms, that is, isomorphisms between ideals.  We may
describe the \(S\)\nb-action
on~\(X\)
through a transformation groupoid \(X\rtimes S\),
which is an étale, locally compact groupoid, possibly non-Hausdorff.
The full inverse semigroup crossed product \(\Cont_0(X)\rtimes S\)
is canonically isomorphic to the full groupoid \(\Cst\)\nb-algebra
\(\Cst(X\rtimes S)\)
(see~\cite{Exel:Inverse_combinatorial}*{Theorem~8.5}
or~\cite{Buss-Meyer:Actions_groupoids}*{Corollary~5.6}). There is an
analogous isomorphism \(\Cont_0(X)\rtimes_\red
S\cong \Cst_\red(X\rtimes S)\) for reduced
crossed products, which follows from
\cite{BussExel:Fell.Bundle.and.Twisted.Groupoids}*{Theorem~4.11} or
from one of our main results (see
Corollary~\ref{cor:exact=>iterated-crossed-product}).  Here we
are going to consider more general versions of these reduced crossed
product decompositions, allowing a groupoid instead of the space~\(X\).

The notion of an action of~\(S\)
on a locally compact groupoid~\(G\)
by partial equivalences is defined
in~\cite{Buss-Meyer:Actions_groupoids}.  Such an action also has a
transformation groupoid~\(G\rtimes S\),
and it induces an \(S\)\nb-action
on~\(\Cst(G)\)
by ``partial Morita--Rieffel equivalences.''  A partial
Morita--Rieffel equivalence is the same as a Hilbert bimodule (not
necessarily full), so we speak of actions by Hilbert bimodules from
now on.

Actions of~\(S\)
on \(\Cst\)\nb-algebras
by Hilbert bimodules are equivalent to saturated Fell bundles~\((A_t)_{t\in S}\)
over~\(S\).
Here the unit fibre~\(A\defeq A_1\)
is the \(\Cst\)\nb-algebra
on which the action takes place.  The other fibres~\(A_t\)
are Hilbert bimodules over~\(A\)
which, together with the multiplication maps
\(A_t\otimes_{A} A_u \to A_{tu}\),
describe the action of~\(S\).
The full section \(\Cst\)\nb-algebra \(C^*((A_t)_{t\in S})\)
of the Fell bundle plays the role of the full crossed product for the
action and is also denoted by \(A\rtimes S\).

Let~\(S\)
act on a locally compact groupoid~\(G\) by partial equivalences as in~\cite{Buss-Meyer:Actions_groupoids}.
The full section \(\Cst\)\nb-algebra
of the Fell bundle over~\(S\)
that describes the induced action on~\(\Cst(G)\)
is identified in~\cite{Buss-Meyer:Actions_groupoids} with
\(\Cst(G\rtimes S)\),
the groupoid \(\Cst\)\nb-algebra
of the transformation groupoid.  Briefly,
\[
\Cst(G)\rtimes S\cong \Cst(G\rtimes S).
\]
Is there a version of this for reduced \(\Cst\)\nb-algebras?

The reduced \(\Cst\)\nb-algebra \(A\rtimes_\red S\defeq
C^*_\red((A_t)_{t\in S})\)
of a Fell bundle over an inverse semigroup is defined
in~\cite{Exel:noncomm.cartan}, and should be the analogue of the
reduced crossed product for an action of a group or groupoid.  The
idea is to induce representations of~\(A\)
to representations of \(A\rtimes S\)
and use only these induced representations to define the
\(\Cst\)\nb-norm
for \(A\rtimes_\red S\).
If~\(S\) is a group, the induction functor comes from a conditional
expectation \(E\colon A\rtimes S\to A\).
Similarly, such a conditional expectation describes induction of
representations for actions of a
\emph{Hausdorff}, locally compact groupoid.  For a
non-Hausdorff groupoid, however, the conditional
expectation takes values in a larger algebra, where
we adjoin certain central projections in the
enveloping \(\Wst\)\nb-algebra~\(\bid A\).
For the reduced groupoid \(\Cst\)\nb-algebra
\(\Cred(G)\),
this is worked out in~\cite{Khoshkam-Skandalis:Regular}.  The
representation of \(\Cred(G) = \Cont_0(G^0)\rtimes_\red G\)
obtained by inducing a faithful representation of~\(\Cont_0(G^0)\)
need not be faithful any more, unlike in the Hausdorff case.

Inverse semigroup actions behave in many ways like actions of étale
groupoids that are possibly non-Hausdorff, so similar problems occur.
In~\cite{Exel:noncomm.cartan}, induction is only defined for
irreducible representations of the coefficient algebra~\(A\);
these induced representations are used to define the reduced crossed
product~\(A\rtimes_\red S\).
When \(A=\Cred(G)\)
for a locally compact groupoid, it is much more convenient to work
with the family of regular representations of~\(G\)
on \(L^2(G^x,\lambda^x)\).
Thus, to compare the reduced crossed product
\(\Cred(G)\rtimes_\red S\)
with \(\Cred(G\rtimes S)\)
for an inverse semigroup action on a groupoid, we want to extend the
induction process in~\cite{Exel:noncomm.cartan} to arbitrary
representations of~\(A\).

We do this by constructing a \emph{weak} conditional expectation~\(E\)
from~\(A\rtimes S\)
to the bidual \(\bid{A}\supseteq A\).
This produces a \(\Cst\)\nb-correspondence
from~\(A\rtimes S\)
to~\(\bid{A}\).
Any representation of~\(A\)
extends uniquely to a normal representation of~\(\bid{A}\),
which we may tensor with the \(\Cst\)\nb-correspondence
to get a representation of~\(A\rtimes S\).
This is the \emph{induction functor} from the category of
representations of~\(A\)
to that of~\(A\rtimes S\).
We let \(A\rtimes_\red S\)
be the quotient of~\(A\rtimes S\)
that is defined by the \(\Cst\)\nb-seminorm
coming from the family of all induced representations; equivalently,
we may induce the universal representation of~\(A\),
which gives a faithful representation of~\(\bid{A}\).

Inducing the regular representation~\(\Lambda_x\)
of~\(\Cst(G)\)
on~\(L^2(G^x)\)
to \(\Cst(G)\rtimes S\cong \Cst(G\rtimes S)\)
gives the regular representation of~\(G\rtimes S\)
on~\(L^2((G\rtimes S)^x)\).
Hence \(\Cred(G\rtimes S)\)
is the image of \(\Cst(G\rtimes S)\)
under the induced representation of \(\bigoplus_{x\in G^0} \Lambda^x\).
This always gives a representation of \(\Cred(G)\rtimes_\red S\),
so there is a quotient map
\(\Cred(G)\rtimes_\red S\to \Cred(G\rtimes S)\).
We give examples where this representation of \(\Cred(G)\rtimes_\red
S\) is not faithful, that is,
\(\Cred(G)\rtimes_\red S\neq \Cred(G\rtimes S)\).

We prove that a representation~\(\pi\) of~\(A\) induces a faithful
representation of~\(A\rtimes_\red S\) if the canonical extension
of~\(\pi\) to the \(\Cst\)\nb-subalgebra of~\(\bid{A}\) generated by
the image of the weak conditional expectation \(E\colon A\rtimes
S\to \bid{A}\) remains faithful.  As a consequence, our new
definition of~\(A\rtimes_\red S\) using \emph{all} induced
representations is equivalent to the original definition
in~\cite{Exel:noncomm.cartan}.  And \(\Cred(G)\rtimes_\red S\cong
\Cred(G\rtimes S)\) if~\(G\) is closed in~\(G\rtimes S\) or if~\(G\)
is ``inner exact''; this exactness property has been studied
recently in~\cite{AnantharamanDelaroch:Weak_containment}.  We also
prove \(C^*_\red(G,\B)\rtimes_\red S\cong C^*_\red(G\rtimes S,\B)\)
for a Fell bundle~\(\B\) over~\(G\rtimes S\) under similar
conditions.

\section{Inverse semigroup actions on C*-algebras}
\label{sec:act_Cstar-Wstar}

Let~\(S\) be a unital inverse semigroup.  That is, \(S\) is a monoid
and for each \(t\in S\) there is a unique \(t^*\in S\) with \(t t^*
t = t\) and \(t^* t t^* = t^*\).  The map \(t\mapsto t^*\) is
involutive and satisfies \((t u)^* = u^* t^*\).  An element~\(e\)
of~\(S\) is \emph{idempotent} if \(e^2 = e\).  The following results
on inverse semigroups are proved, for instance,
in~\cite{Lawson:InverseSemigroups}.  Idempotent elements satisfy
\(e=e^*\) and commute with each other.  Any inverse semigroup is
partially ordered by \(t\le u\) if there is an idempotent element
\(e\in S\) with \(t = u e\) or, equivalently, if there is an
idempotent element \(e\in S\) with \(t = e u\).  This happens for
some \(e \in S\) if and only if \(t = u t^* t\), if and only if \(t
= t t^* u\).

A \emph{Hilbert \(A,B\)-bimodule} is an \(A,B\)-bimodule~\(\Hilm\)
with a left, \(A\)\nb-valued inner product \(\BRAKET{\xi}{\eta}\)
and a right, \(B\)\nb-valued inner product \(\braket{\xi}{\eta}\)
for \(\xi,\eta\in\Hilm\), so that~\(\Hilm\) is a right Hilbert
\(B\)\nb-module and a left Hilbert \(A\)\nb-module, and the two
inner products are linked by \(\zeta\cdot \braket{\xi}{\eta} =
\BRAKET{\zeta}{\xi}\cdot\eta\) for all \(\zeta,\xi,\eta\in\Hilm\).
Hilbert bimodules are interpreted
in~\cite{Buss-Meyer:Actions_groupoids} as partial Morita--Rieffel
equivalences.  We write~\(\Hilm^*\) for the Hilbert \(B,A\)-module
associated to a Hilbert \(A,B\)-module~\(\Hilm\) by exchanging the
left and right structure.  If \(\xi\in\Hilm\), we denote the
corresponding element of~\(\Hilm^*\) by \(\xi^*\in\Hilm^*\).  Let
\(\rg(\Hilm)\) and \(\s(\Hilm)\) be the ideals in~\(A\) generated by
the left and right inner products of vectors in~\(\Hilm\),
respectively.  Thus~\(\Hilm\) is a
\(\rg(\Hilm),\s(\Hilm)\)-imprimitivity bimodule.

\begin{definition}[\cite{Buss-Meyer:Actions_groupoids}]
  \label{def:S_action_Cstar}
  An \emph{action} of~\(S\)
  on a \(\Cst\)\nb-algebra~\(A\)
  \emph{by Hilbert bimodules} consists of Hilbert
  \(A\)\nb-bimodules~\(\Hilm_t\)
  for \(t\in S\)
  and Hilbert bimodule isomorphisms
  \(\mu_{t,u}\colon \Hilm_t\otimes_A \Hilm_u\congto \Hilm_{tu}\)
  for \(t,u\in S\), such that
  \begin{enumerate}[label=\textup{(A\arabic*)}]
  \item \label{enum:AHB3} for all \(t,u,v\in S\),
    the following diagram commutes (associativity):
    \[
    \begin{tikzpicture}[baseline=(current bounding box.west)]
      \node (1) at (0,1) {\((\Hilm_t\otimes_A \Hilm_u) \otimes_A \Hilm_v\)};
      \node (1a) at (0,0) {\(\Hilm_t\otimes_A (\Hilm_u \otimes_A \Hilm_v)\)};
      \node (2) at (5,1) {\(\Hilm_{tu} \otimes_A \Hilm_v\)};
      \node (3) at (5,0) {\(\Hilm_t\otimes_A \Hilm_{uv}\)};
      \node (4) at (7,.5) {\(\Hilm_{tuv}\)};
      \draw[<->] (1) -- node[swap] {ass} (1a);
      \draw[cdar] (1) -- node {\(\mu_{t,u}\otimes_A \Id_{\Hilm_v}\)} (2);
      \draw[cdar] (1a) -- node[swap] {\(\Id_{\Hilm_t}\otimes_A\mu_{u,v}\)} (3);
      \draw[cdar] (3.east) -- node[swap] {\(\mu_{t,uv}\)} (4);
      \draw[cdar] (2.east) -- node {\(\mu_{tu,v}\)} (4);
    \end{tikzpicture}
    \]
  \item \label{enum:AHB1} \(\Hilm_1\)
    is the identity Hilbert \(A,A\)-bimodule~\(A\);
  \item \label{enum:AHB2} \(\mu_{t,1}\colon \Hilm_t\otimes_A
    A\congto \Hilm_t\) and \(\mu_{1,t}\colon A\otimes_A
    \Hilm_t\congto \Hilm_t\) for \(t\in S\), are the maps defined by
    \(\mu_{1,t}(a\otimes\xi)=a\cdot\xi\) and \(\mu_{t,1}(\xi\otimes
    a) = \xi\cdot a\) for \(a\in A\), \(\xi\in\Hilm_t\).
  \end{enumerate}
  If~\(S\)
  has no unit, then we define an \(S\)\nb-action
  by the same data, subject only to condition~\ref{enum:AHB3}.
\end{definition}

\begin{remark}
  \label{rem:extend_action_unit}
  Let~\(S\)
  be an inverse semigroup, possibly without unit.  Let~\(S^+\)
  be the inverse semigroup obtained by adding a new unit element~\(1\)
  to~\(S\).
  An action of~\(S\)
  satisfying only~\ref{enum:AHB3} extends uniquely to~\(S^+\)
  by choosing \(\Hilm_1\defeq A\)
  and letting \(\mu_{1,t}\)
  and~\(\mu_{t,1}\)
  be the multiplication isomorphisms.  This automatically
  satisfies~\ref{enum:AHB3} if \(1\in \{t,u,v\}\),
  so it gives an action of~\(S^+\)
  satisfying \ref{enum:AHB3}--\ref{enum:AHB2}.
\end{remark}

\begin{example}
  \label{exa:partial_actions}
  An \emph{action} of~\(S\) on a \(\Cst\)\nb-algebra~\(A\) \emph{by
    partial isomorphisms} is given by ideals \(I_e\idealin A\) for
  idempotent \(e\in S\) and \Star{}isomorphisms \(\alpha_t\colon
  I_{t^* t} \congto I_{t t^*}\) for \(t\in S\) such that \(\alpha_{t
    u}(a) = \alpha_t\circ\alpha_u(a)\) for all \(t,u \in S\); this
  includes the requirement that \(\alpha_{t u}(a)\) is defined if and
  only if \(\alpha_t\circ\alpha_u(a)\) is defined, that is, \(I_{t
    u}\) is the set of all \(a\in I_{u^* u}\) with \(\alpha_u(a)\in
  I_{t^* t}\).  This gives an action by Hilbert bimodules as in
  Definition~\ref{def:S_action_Cstar}.  Namely, let \(\Hilm_t \defeq
  I_{t^* t}\) with the bimodule structure \(a\cdot \xi\cdot b \defeq
  \alpha_t^{-1}(a)\cdot \xi\cdot b\) and the inner products
  \(\BRAKET{\xi_1}{\xi_2} = \alpha_t(\xi_1 \xi_2^*)\) and
  \(\braket{\xi_1}{\xi_2} = \xi_1^* \xi_2\).  This is indeed a Hilbert
  bimodule with \(\rg(\Hilm)= I_{t t^*}\) and \(\s(\Hilm)= I_{t^*
    t}\).  There are well defined Hilbert bimodule isomorphisms
  \[
  \mu_{t,u}\colon \Hilm_t \otimes_A \Hilm_u \congto \Hilm_{t u},
  \qquad
  \xi\otimes\eta\mapsto \alpha_u^{-1}\bigl(\xi\cdot \alpha_u(\eta)\bigr),
  \]
  because
  \begin{align*}
    \braket{\alpha_u^{-1}(\xi_1 \cdot \alpha_u\eta_1)}
    {\alpha_u^{-1}(\xi_2 \cdot \alpha_u\eta_2)}_{\Hilm_{t u}}
    &= \alpha_u^{-1}(\alpha_u(\eta_1^*) \xi_1^*\xi_2 \alpha_u(\eta_2))
    \\&= \braket{\eta_1}{\braket{\xi_1}{\xi_2}_{\Hilm_t}\cdot\eta_2}_{\Hilm_u},\\
    \BRAKET{\alpha_u^{-1}(\xi_1 \cdot \alpha_u\eta_1)}
    {\alpha_u^{-1}(\xi_2 \cdot \alpha_u\eta_2)}_{\Hilm_{t u}}
    &=
    \alpha_{t u}\bigl(\alpha_u^{-1}(\xi_1 \alpha_u(\eta_1))
    \alpha_u^{-1}(\xi_2 \alpha_u(\eta_2))^*\bigr)
    \\&= \alpha_t(\xi_1 \alpha_u(\eta_1 \eta_2^*) \xi_2^*)
    = \BRAKET{\xi_1\cdot\BRAKET{\eta_1}{\eta_2}_{\Hilm_t}}{\xi_2}_{\Hilm_u}.
  \end{align*}
  This is an action by Hilbert bimodules.  Hence these actions
  generalise actions by isomorphisms.
  Example~\ref{exa:graph_example} shows an action by Hilbert modules
  not of this form.
\end{example}

Actions of~\(S\) on \(\Cst\)\nb-algebras by Hilbert bimodules are
shown in~\cite{Buss-Meyer:Actions_groupoids} to be equivalent to saturated
Fell bundles over~\(S\) as defined in~\cite{Exel:noncomm.cartan}.
Exel's definition of a (saturated) Fell bundle in~\cite{Exel:noncomm.cartan}
starts with a collection of Banach spaces \((\Hilm_t)_{t\in S}\)
and requires the existence of multiplications (with linearly dense range)
\(\Hilm_t\times \Hilm_u\to \Hilm_{tu}\) and involutions \(\Hilm_t\to \Hilm_{t^*}\) for
all \(t\in S\)  and, in addition, inclusion maps \(j_{u,t}\colon
\Hilm_t\hookrightarrow\Hilm_u\) for all \(t,u\in S\) with \(t \le
u\) satisfying a bunch of conditions. The multiplications and involutions can be used to view each \(\Hilm_t\) as a Hilbert bimodule over \(A=\Hilm_1\), and then this data gives an action of $S$ on $A$ in our sense. In the setting of the present paper the existence and properties of the inclusion maps \(j_{u,t}\) in ~\cite{Exel:noncomm.cartan} follow automatically from
\cite{Buss-Meyer:Actions_groupoids}*{Theorem 4.8}.  We recall
briefly how to construct the inclusion maps \(j_{u,t}\) and the involutions, now viewed as isomorphisms of Hilbert bimodules~\(J_t\colon \Hilm_t^*\congto \Hilm_{t^*}\), from the data in Definition~\ref{def:S_action_Cstar}.

If \(t\le u\), there is an idempotent element \(e\in S\) with \(t =
u e\).  Since~\(e\) is idempotent, there is a unique isomorphism
between the Hilbert bimodule~\(\Hilm_e\) and an ideal in~\(A\) so
that the multiplication map \(\mu_{e,e}\colon \Hilm_e\otimes_A
\Hilm_e \to \Hilm_e\) becomes the usual multiplication in~\(A\) (see
\cite{Buss-Meyer:Actions_groupoids}*{Proposition 4.6}).  The
inclusion map~\(j_{u,t}\) is the composite map
\[
\Hilm_t \xleftarrow[\cong]{\mu_{u,e}} \Hilm_u \otimes_A \Hilm_e
\hookrightarrow \Hilm_u \otimes_A A
\cong \Hilm_u,
\]
where the last map is the multiplication map in the right
\(A\)\nb-module~\(\Hilm_u\).

For each \(\xi\in\Hilm_t\) there is a unique element \(J_t(\xi^*)\in
\Hilm_{t^*}\) with \(\mu_{t^*,t}(J_t(\xi^*)\otimes
\eta)=\braket{\xi}{\eta}\) for all \(\eta\in \Hilm_t\); this defines
the involutions~\(J_t\), see
\cite{Buss-Meyer:Actions_groupoids}*{Theorem 4.8}.

We shall need a
stronger result about the ``intersection'' of \(\Hilm_t\)
and~\(\Hilm_u\)
for \(t,u\in S\).  We have
\[
\s(\Hilm_t) = \s(\Hilm_{t^*t}) = \rg(\Hilm_{t^*t}) = \rg(\Hilm_{t^*}).
\]
If \(v\le t\),
then the inclusion map~\(j_{t,v}\)
is a Hilbert bimodule \emph{isomorphism}
\[
\rg(\Hilm_v) \cdot \Hilm_t = \Hilm_t\cdot \s(\Hilm_v) \cong \Hilm_v
\]
because \(\Hilm' = \Hilm\cdot \s(\Hilm') = \rg(\Hilm')\cdot \Hilm\)
holds whenever~\(\Hilm'\) is a Hilbert bimodule contained in another
Hilbert bimodule~\(\Hilm\) by
\cite{Buss-Meyer:Actions_groupoids}*{Proposition 4.3}.

Hence we get Hilbert bimodule isomorphisms
\begin{equation}
  \label{eq:Def-thetas}
  \theta^v_{u,t}\colon
  \Hilm_t\cdot \s(\Hilm_v) \xleftarrow[\cong]{j_{t,v}} \Hilm_v
  \xrightarrow[\cong]{j_{u,v}} \Hilm_u\cdot \s(\Hilm_v)
\end{equation}
for all \(v,t,u\in S\) with \(v\le t,u\).
Let \(I_{t,u}\idealin A\)
be the (closed) ideal generated by~\(\s(\Hilm_v)\)
for all \(v\le t,u\).
This is contained in \(\s(\Hilm_t)\cap\s(\Hilm_u)\),
and the inclusion may be strict.

\begin{lemma}
  \label{lem:intersect_Hilm_tu}
  There is a unique Hilbert bimodule isomorphism
  \[
  \theta_{u,t}\colon \Hilm_t\cdot I_{t,u} \congto \Hilm_u\cdot I_{t,u}
  \]
  that restricts to~\(\theta_{u,t}^v\) on \(\Hilm_t\cdot
  \s(\Hilm_v)\) for all \(v\le t,u\).  These maps satisfy
  \(\theta_{u,t}^{-1} = \theta_{t,u}\) for all \(t,u\in S\) and
  \(\theta_{w,u}(\xi) = \theta_{w,t}\theta_{t,u}(\xi)\) for all
  \(t,u,w\in S\) and \(\xi\in \Hilm_u\cdot (I_{t,u}\cap I_{w,u})\).
\end{lemma}

\begin{proof}
  Linear combinations \(\sum_{v\le t,u} a_v\)
  with \(a_v\in \s(\Hilm_v)\)
  for all~\(v\)
  and only finitely many non-zero~\(a_v\)
  are dense in~\(I_{t,u}\).  We want to define
  \begin{equation}
    \label{eq:theta_tu}
    \theta_{u,t} \Bigl(\xi\cdot \sum_{v\le t,u} a_v\Bigr) \defeq
    \sum_{v\le t,u} \theta^v_{u,t} (\xi \cdot a_v)
  \end{equation}
  for \(\xi\in\Hilm_t\),
  \(a_v\in \s(\Hilm_v)\).
  To check that this is well-defined, we first show that inner
  products are preserved.  The left \(\rg(\Hilm_v)\)-module
  \(\Hilm_t\cdot \s(\Hilm_v)\)
  is nondegenerate because
  \(\Hilm_t\cdot \s(\Hilm_v) = \Hilm_v = \rg(\Hilm_v)\cdot \Hilm_v\).
  Hence we may write \(\xi \cdot a_v = j_{t,v}(a'_v\cdot \xi'_v)\)
  for certain \(a'_v\in \rg(\Hilm_v)\),
  \(\xi'_v\in \Hilm_v\),
  by the Cohen--Hewitt Factorisation Theorem.  For another linear
  combination \(\sum \eta \cdot b_w\)
  with \(\eta\in \Hilm_t\),
  \(b_w\in\s(\Hilm_w)\), and finitely many \(w\le t,u\), we compute
  \begin{multline*}
    \biggl< \sum_{v\le u,t} \theta^v_{u,t} (\xi\cdot a_v) \bigg|
      \sum_{w\le t,u} \theta^w_{u,t} (\eta\cdot b_w) \biggr>_{\Hilm_u}
    = \sum_{v,w\le t,u} \braket{\theta^v_{u,t}j_{t,v}(a_v'\cdot \xi'_v)}
    {\theta^w_{u,t} (\eta\cdot b_w)}_{\Hilm_u}
    \\= \sum_{v,w\le t,u} \braket{j_{u,v}(\xi'_v)}
    {(a_v')^*\cdot \theta^w_{u,t} (\eta\cdot b_w)}_{\Hilm_u}
    = \sum_{v,w\le t,u} \bigl< \xi'_v \big|
    j_{t,v}^{-1}\bigl((a_v')^*\cdot \eta\cdot b_w\bigr)\bigr>_{\Hilm_v}
    \\= \sum_{v,w\le t,u} \braket{j_{t,v}(a'_v\cdot \xi'_v)}
    {\eta\cdot b_w}_{\Hilm_t}
    = \biggl< \sum_{v\le t,u} \xi\cdot a_v \bigg|
      \sum_{w\le t,u} \eta\cdot b_w \biggr>_{\Hilm_t}.
  \end{multline*}
  Since an element of a Hilbert module is determined uniquely by its
  inner products with other elements of the same Hilbert module, the
  right hand side in~\eqref{eq:theta_tu} does not
  depend on the chosen decomposition of
  \(\xi\cdot \sum_{v\le t,u} a_v\in \Hilm_u\cdot I_{t,u}\).
  Hence~\eqref{eq:theta_tu} well-defines an isometric map from a
  dense subspace of \(\Hilm_t\cdot I_{t,u}\)
  to \(\Hilm_u\cdot I_{t,u}\).
  This map extends uniquely to an isometric map
  \(\theta_{u,t}\colon \Hilm_t\cdot I_{t,u} \to \Hilm_u\cdot
  I_{t,u}\).

  The same construction with \(t\) and~\(u\) exchanged gives the map
  \(\theta_{t,u}\colon \Hilm_u\cdot I_{t,u} \to \Hilm_t\cdot
  I_{t,u}\).  This map is inverse to~\(\theta_{u,t}\), so
  that~\(\theta_{u,t}\) is an isomorphism.  Since
  each~\(\theta^v_{u,t}\) is \(A\)\nb-bilinear, so
  is~\(\theta_{u,t}\).  Let \(t,u,w\in S\) and \(\xi\in \Hilm_u\cdot
  (I_{t,u}\cap I_{w,u})\).  The ideal \(I_{t,u}\cap I_{w,u} =
  I_{t,u}\cdot I_{w,u}\) is the sum of the ideals \(\s(\Hilm_e)\cdot
  \s(\Hilm_f) = \s(\Hilm_{e f})\) for idempotent \(e,f\in S\) with
  \(e\le t^* u\), \(f\le w^* u\).  Equivalently, it is the sum
  of~\(\s(\Hilm_x)\) with \(x\le t,u,w\).  If \(x\le t,u,w\) and
  \(\xi\in \Hilm_u\cdot \s(\Hilm_x)\), then
  \[
  \theta_{w,t}\theta_{t,u}(\xi)
  = j_{x,w} j_{x,t}^{-1} j_{x,t} j_{u,t}^{-1}(\xi)
  = j_{x,w} j_{u,t}^{-1}(\xi)
  = \theta_{w,u}(\xi).
  \]
  This implies \(\theta_{w,u}(\xi) = \theta_{w,t}\theta_{t,u}(\xi)\)
  for linear combinations of such~\(\xi\) and hence for all \(\xi\in
  \Hilm_u\cdot (I_{t,u}\cap I_{w,u})\).
\end{proof}

\begin{definition}
  \label{def:representation}
  Let \((\Hilm_t,\mu_{t,u})\)
  be an action of~\(S\)
  on a \(\Cst\)\nb-algebra~\(A\)
  by Hilbert bimodules.  A \emph{representation} of this action by
  multipliers of a \(\Cst\)\nb-algebra~\(D\)
  is a family of linear maps \(\pi_t\colon \Hilm_t\to \Mult(D)\) for
  \(t\in S\) such that
  \begin{enumerate}[label=\textup{(R\arabic*)}]
  \item \label{enum:Rep1}
    \(\pi_{tu}(\mu_{t,u}(\xi\otimes\eta)) = \pi_t(\xi)\pi_u(\eta)\)
    for all \(t,u\in S\), \(\xi\in\Hilm_t\), \(\eta\in\Hilm_u\);
  \item \label{enum:Rep2}
    \(\pi_t(\xi_1)^*\pi_t(\xi_2) = \pi_1(\braket{\xi_1}{\xi_2})\)
    for all \(t\in S\), \(\xi_1,\xi_2\in\Hilm_t\);
  \item \label{enum:Rep3}
    \(\pi_t(\xi_1)\pi_t(\xi_2)^* = \pi_1(\BRAKET{\xi_1}{\xi_2})\)
    for all \(t\in S\), \(\xi_1,\xi_2\in\Hilm_t\);
  \end{enumerate}
  here \(\BRAKET{\xi_1}{\xi_2}\)
  and \(\braket{\xi_1}{\xi_2}\)
  denote the left and right inner products of
  \(\xi_1,\xi_2\in \Hilm_t\).

  The representation is \emph{nondegenerate} if~\(\pi_1(A)D\)
  has dense linear span in~\(D\),
  that is, if~\(\pi_1\) is a nondegenerate representation of~\(A\).

  The \emph{\textup{(}full\textup{)} crossed product}~\(A\rtimes S\)
  of the action~\((\Hilm_t,\mu_{t,u})\)
  is the universal \(\Cst\)\nb-algebra
  for these representations, that is, there is a natural bijection
  between (nondegenerate) \Star{}homomorphisms
  \(A\rtimes S\to\Mult(D)\)
  and (nondegenerate) representations of~\((\Hilm_t,\mu_{t,u})\)
  in~\(\Mult(D)\).
\end{definition}

Like the full crossed product \(A\rtimes S\), the \emph{full section
  \(\Cst\)\nb-algebra} of a Fell bundle is defined by a universal
property with respect to representations of the Fell bundle.  A
representation of a Fell bundle is very close to a representation of
the corresponding action~\((\Hilm_t,\mu_{t,u})\).  The difference is
that representations of the Fell bundle must also be compatible with
the maps \(j_{u,t}\) and~\(J_{t^*}\),
which are part of the data of a Fell bundle over an inverse
semigroup.  However, this extra data is essentially redundant by
\cite{Buss-Meyer:Actions_groupoids}*{Theorem 4.8}.  We are going to
show that any representation of an action is compatible with the
maps \(j_{u,t}\) and~\(J_{t^*}\) in the appropriate sense.  Hence
the full section \(\Cst\)\nb-algebra of a Fell bundle is the same as
the full crossed product of the corresponding action.

By~\ref{enum:AHB2}, condition~\ref{enum:Rep1} for \(t=1\)
and \(u=1\) says that the maps~\(\pi_t\) are \(A\)\nb-bilinear.

\begin{lemma}
  \label{lem:ConsequencesRepresentationDef}
  Let \(\pi_t\colon \Hilm_t \to \Mult(D)\) for \(t\in S\) satisfy
  \ref{enum:Rep1} and~\ref{enum:Rep2}.  Then~\ref{enum:Rep3} is
  equivalent to the following condition:
  \begin{enumerate}[label=\textup{(R3')}]
  \item \label{enum:Rep4}
    \(\pi_t(\Hilm_t)D=\pi_1(\rg(\Hilm_t))D=\pi_1(\rg(\Hilm_{tt^*}))D\)
    for each \(t\in S\).
  \end{enumerate}
\end{lemma}

\begin{proof}
  First assume~\ref{enum:Rep3}.  Recall that
  \(\rg(\Hilm_t)=\rg(\Hilm_{tt^*})\) for all \(t\in S\).  Since
  \(\Hilm_t=\rg(\Hilm_t)\cdot \Hilm_t\), we have \(\pi_t(\Hilm_t)D =
  \pi_1(\rg(\Hilm_t))\pi_t(\Hilm_t)D \subseteq
  \pi_1(\rg(\Hilm_t))D\).  The reverse inclusion follows
  from~\ref{enum:Rep3}: \(\pi_t(\Hilm_t)D\supseteq
  \cspn\pi_t(\Hilm_t)\pi_t(\Hilm_t)^*D=\pi_1(\rg(\Hilm_t))D\).
  Conversely, assume~\ref{enum:Rep4}.  If
  \(\xi_1,\xi_2,\xi_3\in\Hilm_t\), then
  \[
  \pi_t(\xi_1) \pi_t(\xi_2)^* \pi_t(\xi_3)
  = \pi_t(\xi_1 \braket{\xi_2}{\xi_3})
  = \pi_1(\BRAKET{\xi_1}{\xi_2}) \pi_t(\xi_3).
  \]
  Hence the operators in~\ref{enum:Rep3} are equal
  on~\(\pi_t(\Hilm_t)D\).  Then they are also equal on
  \(\pi_1(\BRAKET{\Hilm_t}{\Hilm_t})\cdot D\) by~\ref{enum:Rep4}.
  Now~\ref{enum:Rep3} follows because
  \[
  (\pi_t(\xi_1) \pi_t(\xi_2)^* - \pi_1(\BRAKET{\xi_1}{\xi_2}))\cdot
  (\pi_t(\xi_1) \pi_t(\xi_2)^* - \pi_1(\BRAKET{\xi_1}{\xi_2}))^*
  = 0.\qedhere
  \]
\end{proof}

\begin{proposition}
  \label{pro:crossed_product_section_algebra}
  Any representation \((\pi_t)_{t\in S}\)
  of \((\Hilm_t,\mu_{t,u})\) is compatible with the maps
  \(j_{u,t}\colon \Hilm_t\to\Hilm_u\)
  and \(J_t\colon \Hilm_t^*\to \Hilm_{t^*}\)
  in the sense that \(\pi_u(j_{u,t}(\xi)) = \pi_t(\xi)\)
  and \(\pi_{t^*}(J_t(\xi^*)) = \pi_t(\xi)^*\)
  for all \(t,u\in S\) with \(t\le u\) and \(\xi\in\Hilm_t\).

  Even more, \(\pi_u\circ\theta_{u,t}(\xi)=\pi_t(\xi)\) for
  all \(t,u \in S\), \(\xi\in\Hilm_t\cdot I_{t,u}\).

  The full crossed product \(A\rtimes S\)
  is the same as the full section \(\Cst\)\nb-algebra
  of the Fell bundle associated to the action.
\end{proposition}

\begin{proof}
  The full section \(\Cst\)\nb-algebra of a Fell bundle and the full
  crossed product \(A\rtimes S\) are both defined through a
  universal property for certain representations.  To prove their
  equality, we show that any representation
  of~\((\Hilm_t,\mu_{t,u})\) is compatible with the maps~\(j_{u,t}\)
  and~\(J_{t^*}\).

  We may assume that~\(\Hilm_e\) for an idempotent element \(e\in
  S\) is an ideal of~\(A\) with the standard Hilbert bimodule
  structure and that \(\mu_{t,e}\) and~\(\mu_{e,t}\) are the maps
  \(\xi\otimes a\mapsto \xi\cdot a\) and \(a\otimes\xi\mapsto
  a\cdot\xi\) for all \(t\in S\).  This normalisation follows from
  \cite{Buss-Meyer:Actions_groupoids}*{Proposition 4.6} as in the
  proof of \cite{Buss-Meyer:Actions_groupoids}*{Proposition 3.7}.
  The results in~\cite{Buss-Meyer:Actions_groupoids} for actions of
  inverse semigroups on groupoids carry over to actions on
  \(\Cst\)\nb-algebras because both setups share some basic
  properties, which suffice for the proofs to go through, compare
  the proof of \cite{Buss-Meyer:Actions_groupoids}*{Theorem 4.8},
  which merely refers to the earlier proofs of
  \cite{Buss-Meyer:Actions_groupoids}*{Propositions 3.7 and 3.9}.

  Let \(t\le u\)
  and \(e=t^*t\in E(S)\), where \(E(S) \defeq \{e\in S\mid e^* e =
  e\}\).
  The embedding \(j_{u,t}\colon\Hilm_t\hookrightarrow \Hilm_u\)
  is defined by \(j_{u,t}(\mu_{u,e}(\xi\otimes a))=\xi\cdot a\)
  for \(\xi\in \Hilm_u\)
  and \(a\in \Hilm_e\);
  this is well-defined because the multiplication map
  \(\mu_{u,e}\colon \Hilm_u\otimes_A\Hilm_e\congto \Hilm_t\)
  is an isomorphism.  The conditions for a representation imply
  \[
  \pi_u(j_{u,t}(\mu_{u,e}(\xi\otimes a))
  = \pi_u(\xi)\cdot \pi_1(a),\qquad
  \pi_t(\mu_{u,e}(\xi\otimes a))
  = \pi_u(\xi)\cdot \pi_e(a).
  \]
  So we are done if we show that \(\pi_e=\pi_1|_{\Hilm_e}\)
  for idempotent~\(e\).
  For this we take \(a,b\in \Hilm_e\)
  and use~\ref{enum:Rep1} to compute:
  \[
  \pi_1(a)\pi_e(b)=\pi_e(ab)=\pi_e(a)\pi_e(b).
  \]
  Hence \(\pi_1(a)\xi=\pi_e(a)\xi\)
  for all \(\xi\in \pi_e(\Hilm_e)D=\pi_1(\Hilm_e)D\)
  by Lemma~\ref{lem:ConsequencesRepresentationDef}.  Since the image
  of \(\pi_1(a^*)-\pi_e(a^*)\)
  is contained in \(\pi_e(\Hilm_e)D=\pi_1(\Hilm_e)D\),
  this implies \((\pi_1(a)-\pi_e(a))(\pi_1(a^*)-\pi_e(a^*))=0\),
  hence \(\pi_1(a)=\pi_e(a)\) for all \(a\in\Hilm_e\).

  The compatibility of~\(\pi\)
  with the maps~\(j_{u,t}\)
  implies that \(\pi_u(\theta_{u,t}^v(\eta))=\pi_t(\eta)\)
  for all \(v,t,u\in S\)
  with \(v\le t,u\)
  and \(\eta\in \Hilm_t\cdot \s(\Hilm_v)\).
  Since this holds for all~\(v\),
  we get \(\pi_u\circ\theta_{u,t}(\xi)=\pi_t(\xi)\)
  by the construction of~\(\theta_{u,t}\).

  By definition, \(J_t(\xi^*)\in \Hilm_{t^*}\)
  for \(\xi\in\Hilm_t\)
  is the unique element with
  \[
  \mu_{t^*,t}(J_t(\xi^*)\otimes \eta)=\braket{\xi}{\eta}
  \]
  for all \(\eta\in \Hilm_t\).  Hence
  \begin{multline*}
    \pi_t(\xi)^*\pi_t(\eta)
    = \pi_1(\braket{\xi}{\eta})=\pi_{t^*t}(\braket{\xi}{\eta})
    \\= \pi_{t^*t}(\mu_{t^*,t}(J_t(\xi^*)\otimes\eta))
    = \pi_{t^*}(J_t(\xi^*))\pi_t(\eta).
  \end{multline*}
  Hence \(\pi_t(\xi)^*x=\pi_{t^*}(J_t(\xi^*))x\)
  for all \(x\in \pi_t(\Hilm_t)D=\pi_1(\Hilm_{tt^*})D\)
  by Lemma~\ref{lem:ConsequencesRepresentationDef}.  Since
  \(\Hilm_{t^*} = \Hilm_{t^*}\cdot \Hilm_{tt^*}\),
  we have
  \(\pi_{t^*}(\Hilm_{t^*})^*D =
  \pi_{tt^*}(\Hilm_{tt^*})^*\pi_{t^*}(\Hilm_{t^*})^*D \subseteq
  \pi_1(\Hilm_{tt^*}) D\).
  Hence the image of \(\pi_t(\xi) - \pi_{t^*}(J_t(\xi^*))^*\)
  is contained in \(\pi_1(\Hilm_{tt^*})D\).
  Thus
  \((\pi_t(\xi)^* - \pi_{t^*}(J_t(\xi^*))(\pi_t(\xi) -
  \pi_{t^*}(J_t(\xi^*))^*) = 0\),
  so that \(\pi_t(\xi)^* = \pi_{t^*}(J_t(\xi^*))\).
\end{proof}

We now describe~\(A\rtimes S\) as the \(\Cst\)\nb-completion of a
certain dense \Star{}\hspace{0pt}subalgebra.  Let \(A\rtimes_\alg
S\) be the quotient of \(\bigoplus_{t\in S} \Hilm_t\) by the linear
span of \(\theta_{u,t}(\xi)\delta_u-\xi\delta_t\) for all \(t,u\in
S\) and \(\xi\in\Hilm_t\cdot I_{t,u}\).  The multiplication
maps~\(\mu_{t,u}\) and the involutions~\(J_t\) described above turn
this into a \Star{}algebra.  Our definition of \(A\rtimes_\alg S\)
is slightly different from the one in~\cite{Exel:noncomm.cartan},
where only the linear span of \(j_{u,t}(\xi)\delta_u-\xi\delta_t\)
for \(t,u\in S\) with \(t\le u\) is divided out.  This linear span
contains \(\theta_{u,t}(\xi)\delta_u-\xi\delta_t\) for \(t,u\in S\)
and \(\xi\in\Hilm_t\cdot \Hilm_e\) for all \(e\in E(S)\)
with \(te = ue\).  For fixed \(t,u\), the closure of this subspace
in \(\Hilm_u \oplus \Hilm_t\) contains
\(\theta_{u,t}(\xi)\delta_u-\xi\delta_t\) even if only
\(\xi\in\Hilm_t\cdot I_{t,u}\).  Therefore, a linear map or a
seminorm on \(\bigoplus_{t\in S} \Hilm_t\) that is norm-bounded on
each summand and vanishes on \(j_{u,t}(\xi)\delta_u-\xi\delta_t\)
for \(t,u\in S\) with \(t\le u\) and \(\xi\in\Hilm_t\) still
vanishes on \(\theta_{u,t}(\xi)\delta_u-\xi\delta_t\) for \(t,u\in
S\) and \(\xi\in\Hilm_t\cdot I_{t,u}\).

We will show below that \(A\rtimes_\alg S\) as defined above embeds
into the \(\Cst\)\nb-algebra \({A\rtimes_\red S}\), which justifies
our small change in the definition.

A representation of the action~\((\Hilm_t)_{t\in S}\) is equivalent
to a representation of the \Star{}algebra \(A\rtimes_\alg S\) by
Proposition~\ref{pro:crossed_product_section_algebra}.
Hence~\(A\rtimes S\) is the \(\Cst\)\nb-completion
of~\(A\rtimes_\alg S\).  Why does the maximal \(\Cst\)\nb-seminorm
on~\(A\rtimes_\alg S\) exist?  If \(\xi\in\Hilm_t\), then
\(\norm{\xi} \le \norm{\xi^* \xi}^{1/2}\) for any
\(\Cst\)\nb-seminorm on~\(A\rtimes_\alg S\).  Since \(\xi^*\xi\in
A\), which is already a \(\Cst\)\nb-algebra, the spectral radius
of~\(\xi^* \xi\) gives a finite upper bound on~\(\norm{\xi}^2\) for
any \(\xi\in\Hilm_t\).  This implies a finite upper bound
for~\(\norm{\xi}\) for any \(\xi\in A\rtimes_\alg S\).

In order to define a reduced analogue of~\(A\rtimes S\),
we need a way to induce representations of~\(A\)
to representations of~\(A\rtimes_\alg S\).
Then \(A\rtimes_\red S\)
is defined as the completion in the \(\Cst\)\nb-seminorm
on~\(A\rtimes_\alg S\)
defined by these ``regular'' representations.

Exel describes an induction process for \emph{pure} states
in~\cite{Exel:noncomm.cartan}.
We want, instead, an induction process for \emph{all} states or,
equivalently, for all representations of~\(A\).
By \cite{Rieffel:Morita}*{Theorem 6.9}, such an induction functor is
equivalent to a self-dual (right) Hilbert module over the
bidual~\(\bid{A}\)
of~\(A\)
with a normal left action of~\(\bid{(A\rtimes S)}\);
the normal left action of~\(\bid{(A\rtimes S)}\)
is equivalent to a nondegenerate representation of~\(A\rtimes S\).
We shall construct this \(\Cst\)\nb-correspondence
from a weak expectation, that is, a normal expectation
\(E\colon \bid{(A\rtimes S)}\to \bid{A}\);
the resulting Hilbert module is the completion of
\(\bid{(A\rtimes S)}\)
for the inner product \(\braket{x_1}{x_2} \defeq E(x_1^*\cdot x_2)\).
To construct this weak expectation, we will extend the
action~\((\Hilm_t,\mu_{t,u})\)
on~\(A\)
to an action on the enveloping \(\Wst\)\nb-algebra~\(\bid{A}\).

\section{Actions on W*-algebras}
\label{sec:act_Wstar}

A Hilbert module~\(\Hilm\)
over a \(\Cst\)\nb-algebra~\(A\)
is \emph{self-dual} if every bounded \(A\)\nb-module
map \(\Hilm\to A\)
is of the form \(\xi\mapsto \braket{\eta}{\xi}\)
for some \(\eta\in\Hilm\) (see~\cite{Paschke:Inner_product}).
Any bounded \(A\)\nb-module
map between self-dual Hilbert modules is adjointable.  For Hilbert
modules over \(\Wst\)\nb-algebras, self-duality is equivalent to
compatibility of the action with the weak topologies,
see~\cite{Skeide:Self-duality}.

\begin{definition}
  \label{def:S-action_Wstar}
  Let~\(S\) be a unital inverse semigroup.
  An action of~\(S\) on a \(\Wst\)\nb-algebra~\(M\)
  consists of \emph{self-dual} Hilbert
  \(M\)\nb-bimodules~\((\Hilm_t)_{t\in S}\)
  and Hilbert \(M\)\nb-bimodule isomorphisms
  \[
  \mu_{t,u}\colon \Hilm_t\barotimes_M \Hilm_u\congto \Hilm_{tu}
  \]
  for \(t,u\in S\)
  satisfying analogues of \ref{enum:AHB3}--\ref{enum:AHB2},
  where~\(\barotimes\)
  denotes the weak closure of the tensor product of the two bimodules.
\end{definition}

\begin{example}
  \label{exa:bidual_action}
  We shall be mostly interested in the following situation.
  Let~\((\Hilm_t)_{t\in S}\) be an action of~\(S\) by Hilbert bimodules on
  a \(\Cst\)\nb-algebra~\(A\), as in the previous section.  The
  bidual~\(\bid{\Hilm_t}\)
  of~\(\Hilm_t\)
  is a self-dual Hilbert \(\bid{A}\)\nb-bimodule
  (see~\cite{Schweizer:Hilbert_Cstar}).  The isomorphism~\(\mu_{t,u}\)
  induces a Hilbert \(\bid{A}\)\nb-bimodule isomorphism
  \[
  \bid{\mu_{t,u}}\colon \bid{\Hilm_t}\barotimes_{\bid{A}}\bid{\Hilm_u}
  \to \bid{\Hilm_{tu}}.
  \]
  A quick way to see this is to form a linking \(\Cst\)\nb-algebra
  and its \(\Wst\)\nb-hull:
  \[
  L_{t,u} \defeq
  \begin{pmatrix}
    A& \Hilm_u^*&\Hilm_{tu}^*\\
    \Hilm_u&A&\Hilm_{t}^*\\
    \Hilm_{tu}&\Hilm_t&A
  \end{pmatrix},\qquad \bid{L_{t,u}} =
  \begin{pmatrix}
    \bid{A}& \bid{(\Hilm_u^*)}& \bid{(\Hilm_{tu}^*)}\\
    \bid{\Hilm_u}&\bid{A}&\bid{(\Hilm_{t}^*)}\\
    \bid{\Hilm_{tu}}&\bid{\Hilm_{t}}&\bid{A}
  \end{pmatrix}.
  \]
  The multiplication in~\(L_{t,u}\)
  restricts to~\(\mu_{t,u}\)
  on the summands \(\Hilm_t,\Hilm_u\),
  so the multiplication in~\(\bid{L_{t,u}}\)
  restricts to~\(\bid{\mu_{t,u}}\)
  on \(\bid{\Hilm_t},\bid{\Hilm_u}\).
  The data \((\bid{\Hilm_t},\bid{\mu_{t,u}})\)
  defined above gives an action of~\(S\)
  on the \(\Wst\)\nb-algebra~\(\bid{A}\).
\end{example}

We return to the general case.  Let~\(M\)
be a \(\Wst\)\nb-algebra
and let \((\Hilm_t,\mu_{t,u})\) be an \(S\)\nb-action
on~\(M\) by self-dual Hilbert bimodules.

For an idempotent \(e\in E(S)\),
\(\Hilm_e\idealin M\)
is a weakly closed, two-sided ideal.  So it is of the form
\(P_e\cdot M = M\cdot P_e\)
for a central projection~\(P_e\in M\).
For \(t,u\in S\),
let \(\bar{I}_{t,u}\)
be the \emph{weak} closure of \(\sum_{v\le t,u} \s(\Hilm_v)\).
Equivalently, \(\bar I_{t,u}\)
is the ideal generated by the supremum of the central
projections~\(P_{v^* v}\)
for \(v\le t,u\).
In the situation of Example~\ref{exa:bidual_action},
\(\bar{I}_{t,u}\subseteq \bid{A}\) is
the bidual of the ideal~\(I_{t,u}\) in~\(A\).

\begin{lemma}
  \label{lem:intersect_Hilm_tu_self-dual}
  For \(t,u\in S\),
  there is a unique \(M\)\nb-bimodule
  map and partial isometry \(\Theta_{u,t}\colon \Hilm_t\to \Hilm_u\)
  that extends the maps
  \[
  \theta_{u,t}^v\colon
  \Hilm_t\supseteq \Hilm_t\cdot \s(\Hilm_v)
  \xrightarrow[\cong]{\mu_{t,v^* v}} \Hilm_v
  \xrightarrow[\cong]{\mu_{u,v^* v}^{-1}}
  \Hilm_u\cdot \s(\Hilm_v) \subseteq \Hilm_u
  \]
  for all \(v\le t,u\)
  and satisfies
  \(\Theta_{u,t}\Theta_{u,t}^*(\Hilm_t) = \Hilm_t\cdot \bar{I}_{t,u}\)
  and
  \(\Theta_{u,t}^*\Theta_{u,t}(\Hilm_u) = \Hilm_u\cdot
  \bar{I}_{t,u}\).
  Furthermore, \(\Theta_{u,t}^*=\Theta_{t,u}\),
  \(\Theta_{u,t}\circ j_{t,v} = \Theta_{u,v}\) for all \(u,t,v\in S\)
  with \(v\le t\), and
  \(\Theta_{u,t}\circ \Theta_{t,v}(\xi) = \Theta_{u,v}(\xi)\) for
  all \(u,t,v\in S\) and \(\xi \in \Hilm_v\cdot \bar{I}_{t,v}\).
\end{lemma}

\begin{proof}
  If \(v\le t\),
  then
  \(\Hilm_v \cong \Hilm_t \barotimes_M \Hilm_{v^* v} \cong \Hilm_t \cdot
  \s(\Hilm_v) = \Hilm_t \cdot P_{v^* v} = P_{v v^*} \cdot \Hilm_t\).
  The projections~\(P_{v v^*}\)
  for \(v\le t,u\)
  are central and hence commute with each other.  Therefore, if
  \(v_1,\dotsc,v_n\le t,u\),
  then there is a unique map that extends~\(\theta_{u,t}^{v_i}\)
  for \(i=1,\dotsc,n\)
  and has image \(\bigvee P_{v_i v_i^*}\cdot \Hilm_t\).
  One way to write it is
  \begin{multline*}
    \theta_{u,t}^{v_1,\dotsc,v_n}(\xi) \defeq
    \theta_{u,t}^{v_1}(P_{v_1 v_1^*}\xi)
    + \theta_{u,t}^{v_2}(P_{v_2 v_2^*}(1-P_{v_1 v_1^*})\xi)
    \\+ \theta_{u,t}^{v_3}(P_{v_3 v_3^*} (1-P_{v_2 v_2^*})
    (1-P_{v_1 v_1^*})\xi) + \dotsb
    \\+ \theta_{u,t}^{v_n}(P_{v_n v_n^*} (1-P_{v_{n-1} v_{n-1}^*}) \dotsm
    (1-P_{v_2 v_2^*}) (1-P_{v_1 v_1^*}) \xi).
  \end{multline*}
  We have defined partial isometries~\(\theta_{u,t}^F\)
  for each finite set \(F\subseteq S\)
  with \(v\le t,u\)
  for all \(v\in F\).
  If \(F\le F'\),
  then the partial isometry~\(\theta_{u,t}^{F'}\)
  agrees with~\(\theta_{u,t}^F\)
  on the image of its source projection, and merely extends it to a
  larger submodule.  Such a net of operators has a weak limit, and
  this limit has the properties required of~\(\Theta_{u,t}\)
  because \(\bigvee_{v\le t,u} P_{v v^*}\)
  generates~\(\bar{I}_{t,u}\) and is the weak limit of the range
  projections \(\bigvee_{v\in F} P_{v v^*}\) for~\(F\) as above.
  Any bounded operator \(\Hilm_t\to \Hilm_u\)
  is weakly continuous.  If its source projection is
  \(\bigvee_{v\le t,u} P_{v v^*}\),
  it is determined uniquely by its restriction to the images of
  \(\bigvee_{v\in F} P_{v v^*}\)
  for all finite sets~\(F\)
  as above.  Hence there is only one operator with the properties
  required of~\(\Theta_{t,u}\).

  The operator~\(\Theta_{t,u}^*\)
  satisfies the conditions that characterise~\(\Theta_{u,t}\),
  so \(\Theta_{t,u}^* = \Theta_{u,t}\).
  If \(v\le t\),
  the operator~\(j_{t,v}\)
  is an isometry whose range projection~\(P_{v v^*}\)
  commutes with the source projection of the partial
  isometry~\(\Theta_{u,t}\).
  Hence \(\Theta_{u,t}\circ j_{t,v}\)
  is again a partial isometry.  This satisfies the conditions that
  characterise~\(\Theta_{u,v}\)
  because the set of \(w\le v,u\)
  is exactly the set of all \(w' = w v^* v = v v^* w\)
  for \(w'\le t,u\).  So \(\Theta_{u,t}\circ j_{t,v} =
  \Theta_{u,v}\) if \(t,u,v\in S\) and \(v\le t\).  If \(\xi\in
  j_{t,w}(\Hilm_w) = \Hilm_t\cdot \s(\Hilm_w)\) for some \(w\in S\)
  with \(w\le t, v\), then \(\xi= j_{t,w}(\xi')\) for \(\xi'\in
  \Hilm_w\), and so
  \[
  \Theta_{u,t} \Theta_{t,v}(\xi)
  = \Theta_{u,t} \Theta_{t,v} j_{v,w}(\xi')
  = \Theta_{u,t} j_{t,w}(\xi')
  = \Theta_{u,w} (\xi')
  = \Theta_{u,t} j_{t,w}(\xi')
  = \Theta_{u,t} (\xi)
  \]
  because \(\Theta_{u',t'}\circ j_{t',v'} = \Theta_{u',v'}\) if
  \(v'\le t'\).  The set of \(\xi\in\Hilm_t\) with \(\Theta_{u,t}
  \Theta_{t,v}(\xi) = \Theta_{u,t} (\xi)\) is a weakly closed
  subspace.  Since \(\Hilm_t\cdot \bar{I}_{t,v}\) is the weakly
  closed linear span of \(j_{t,w}(\Hilm_w)\) for \(w\le t,v\), the
  equation \(\Theta_{u,t} \Theta_{t,v}(\xi) = \Theta_{u,t} (\xi)\)
  holds for all \(\xi\in \Hilm_t\cdot \bar{I}_{t,v}\).
\end{proof}

\begin{remark}
  In general, \(\Theta_{u,t}\circ\Theta_{t,v} \neq \Theta_{u,v}\).
  For instance, \(\Theta_{u,t}\circ\Theta_{t,u}\) is the identity
  map on \(\Hilm_u\cdot I_{t,u}\) and zero on its orthogonal
  complement in~\(\Hilm_u\), whereas~\(\Theta_{u,u}\) is the identity
  on~\(\Hilm_u\).
\end{remark}

We now define~\(M\rtimes_\alg S\) as the quotient of
\(\bigoplus_{t\in S} \Hilm_t\) by the linear span of
\(\Theta_{u,t}(\xi)\delta_u-\xi\delta_t\) for all \(t,u\in S\) and
\(\xi\in\Hilm_t\cdot \bar{I}_{t,u}\).  By construction
of~\(\bar{I}_{t,u}\) and~\(\Theta_{t,u}\), the linear span of the
elements of \(\Hilm_t \oplus \Hilm_u\) of the form
\(j_{u,v}(\xi)\delta_u-j_{t,v}(\xi)\delta_t\) for \(\xi\in\Hilm_v\)
and \(v\le t,u\) is
weakly dense in the space of all
\(\Theta_{u,t}(\xi)\delta_u-\xi\delta_t\) with \(\xi\in\Hilm_t\cdot
\bar{I}_{t,u}\).  Therefore, a linear map or a seminorm on
\(\bigoplus_{t\in S} \Hilm_t\) descends to~\(M\rtimes_\alg S\) if it
is weakly continuous on each summand and vanishes on
\(j_{u,v}(\xi)\delta_u-\xi\delta_v\) for \(u,v\in S\) with \(v\le
u\) and \(\xi\in\Hilm_v\).

The map
\begin{equation}
  \label{eq:cond.exp}
  E\colon \bigoplus_{t\in S} \Hilm_t \to M,\qquad
  \sum_{t\in S} \xi_t\delta_t \mapsto \sum_{t\in S} \Theta_{1,t}(\xi_t),
\end{equation}
is normal on each summand and vanishes on
\(j_{t,v}(\eta)\delta_t-\eta\delta_v\) for \(t\le v\),
\(\eta\in\Hilm_v\) because
Lemma~\ref{lem:intersect_Hilm_tu_self-dual} gives
\(\Theta_{1,t}\circ j_{t,v} = \Theta_{1,v}\) if \(v\le t\).
Hence~\eqref{eq:cond.exp} defines a map \(M\rtimes_\alg S\to M\),
which we also denote by~\(E\).  This is an \(M\)\nb-bimodule map
with \(E|_M=\Id_M\).

\begin{proposition}
  \label{pro:E_positive}
  The map \(E\colon M\rtimes_\alg S\to M\) is a faithful conditional
  expectation.  That is, \textup{(1)}
  \(E(\xi)^*=E(\xi^*)\), \textup{(2)} \(E(\xi^*\xi)\ge 0\) in \(M\)
  for all \(\xi\in M\rtimes_\alg S\), and \textup{(3)}
  \(E(\xi^*\xi)= 0\) only if
  \(\xi=0\).  Hence \((\xi,\eta)\mapsto \braket{\xi}{\eta} \defeq
  E(\xi^*\eta)\) defines an \(M\)\nb-valued inner product
  on~\(M\rtimes_\alg S\).
\end{proposition}

\begin{proof}
  Condition~(1) holds in general once it holds for
  \(\xi=\xi\delta_t\) with \(\xi\in \Hilm_t\), \(t\in S\).  In this
  case, (1)~means that \(\Theta_{1,t}(\xi)^* =
  \Theta_{1,t^*}(\xi^*)\).  The map \(\Hilm_t\ni \xi\mapsto
  \Theta_{1,t^*}(\xi^*)^*\in \Hilm_1\) is also \(M\)\nb-bilinear
  and a partial isometry, and it has the properties in
  Lemma~\ref{lem:intersect_Hilm_tu_self-dual} that
  characterise~\(\Theta_{1,t}\).  Hence \(\Theta_{1,t}(\xi) =
  \Theta_{1,t^*}(\xi^*)^*\) for all \(\xi\in\Hilm_t\), as desired.

  To prove (2) and~(3), fix \(\xi\in M\rtimes_\alg S\).  Write \(\xi=
  \sum_{i=1}^n \xi_i\delta_{t_i}\) for \(t_i\in S\),
  \(\xi_i\in\Hilm_{t_i}\), \(i=1,\dotsc,n\).  First we choose a
  normal form of this representative of~\(\xi\).  (This normal form
  becomes unique if we fix some total order~\(\prec\) on the
  set~\(S\) and assume that \(t_1\prec t_2\prec \dotsb\prec t_n\).)
  First, we split \(\xi_i =
  \Theta_{t_1,t_i}^*\Theta_{t_1,t_i}(\xi_i) +
  (1-\Theta_{t_1,t_i}^*\Theta_{t_1,t_i})(\xi_i)\) for \(i\ge 2\).
  Thus
  \[
  \xi_i \delta_{t_i}
  \equiv \Theta_{t_1,t_i}(\xi_i) \delta_{t_1}
  + (1-\Theta_{t_1,t_i}^*\Theta_{t_1,t_i})(\xi_i) \delta_{t_i}
  \qquad \text{in }M\rtimes_\alg S
  \]
  for \(i\ge 2\).  Replacing~\(\xi_i\delta_{t_i}\) by the right hand
  side gives a new representative \(\xi = \sum_{i=1}^n \xi'_i
  \delta_{t_i}\) with the extra property \(\Theta_{t_1,
    t_i}(\xi'_i)=0\) for \(i=2,\dotsc,n\).  Next, we split \(\xi'_i
  = \Theta_{t_2,t_i}^*\Theta_{t_2,t_i}(\xi'_i) +
  (1-\Theta_{t_2,t_i}^*\Theta_{t_2,t_i})(\xi'_i)\) for \(i\ge 3\)
  and repeat the normalisation step above.  This gives a new
  representative \(\xi = \sum_{i=1}^n \xi''_i \delta_{t_i}\) with
  the extra property \(\Theta_{t_2, t_i}(\xi''_i)=0\) for
  \(i=3,\dotsc,n\); we still have \(\Theta_{t_1, t_i}(\xi''_i)=0\)
  for \(i=2,\dotsc,n\).  Continuing this way, we eventually arrive
  at a new representative \(\xi = \sum_{i=1}^n \bar\xi_i
  \delta_{t_i}\) with \(\Theta_{t_i, t_j}(\bar\xi_j)=0\) for all
  \(1\le i<j\le n\).

  By definition, \(E((\zeta\delta_t)^*\cdot \eta\delta_u) =
  \Theta_{1,t^*u}(\zeta^*\cdot \eta)\) for all \(t,u\in S\),
  \(\zeta\in\Hilm_t\), \(\eta\in\Hilm_u\).  We claim that this is
  equal to \(\zeta^*\cdot \Theta_{t,u}(\eta)\).  If \(e\le 1,t^*u\),
  then \(te\le t,u\); conversely, if \(v\le t,u\), then \(t^*v \le
  1,t^*u\).  The maps \(e\mapsto te\) and \(v\mapsto t^*v\) are
  bijective between the relevant subsets of~\(S\) because
  \(t^*te=e\) if \(e\le t^*u\) and \(tt^*v = v\) if \(v\le t\).
  Hence \(\bar I_{1,t^*u}=\bar I_{t,u}\); the
  compatibility of the embeddings~\(j_{u,t}\) with the
  multiplication maps~\(\mu\) implies \(\Theta_{1,t^*u}(\zeta^*\eta) =
  \zeta^*\Theta_{t,u}(\eta)\) for all \(\zeta\in\Hilm_t\),
  \(\eta\in\Hilm_u\).  Now the claim follows.  So
  \[
  E((\zeta\delta_t)^*\cdot \eta\delta_u)
  = \zeta^*\cdot \Theta_{t,u}(\eta)\qquad
  \text{for all }t,u\in S,\ \zeta\in\Hilm_t,\ \eta\in\Hilm_u.
  \]

  Our normalisation condition \(\Theta_{t_i, t_j}(\bar\xi_j)=0\)
  for all \(1\le i<j\le n\) gives
  \[
  E((\bar\xi_i\delta_{t_i})^* \cdot \bar\xi_j\delta_{t_j})
  = \bar\xi_i^* \Theta_{t_i,t_j}(\bar\xi_j)
  = 0
  \]
  if \(i<j\).  Since \(E(\eta^*)=E(\eta)^*\), this also vanishes for
  \(i>j\), so
  \[
  E(\xi^*\xi) = \sum_{i=1}^n \bar\xi_i^*\bar\xi_i
  = \sum_{i=1}^n \braket{\bar\xi_i}{\bar\xi_i}_{\Hilm_{t_i}}.
  \]
  Thus \(E(\xi^*\xi)\ge0\) in~\(M\), and \(E(\xi^*\xi)=0\) in~\(M\)
  only if \(\braket{\bar\xi_i}{\bar\xi_i}_{\Hilm_{t_i}}=0\) for
  all~\(i\), that is, only if \(\xi=0\).
\end{proof}

Proposition~\ref{pro:E_positive} allows us to
complete~\(M\rtimes_\alg S\) to a Hilbert module~\(\ell^2(S,M)\)
over~\(M\).  The left multiplication action of~\(M\rtimes_\alg S\)
on itself extends to a unital left action~\(\lambda\) of
\(M\rtimes_\alg S\) on~\(\ell^2(S,M)\) by adjointable operators.
This representation is faithful because~\(E\) is faithful
on~\(M\rtimes_\alg S\): if \(\lambda(\xi)=0\) for \(\xi\in
M \rtimes_\alg S\), then \(\lambda(\xi)(1) = \xi=0\)
in~\(\ell^2(S,M)\); this gives \(\braket{\xi}{\xi} = E(\xi^* \xi)
=0\) and hence \(\xi=0\).  The module~\(\ell^2(S,M)\) need not be
self-dual again, so we replace it by its self-dual
completion~\(\bar{\ell}^2(S,M)\).  We still
represent~\(M\rtimes_\alg S\) on~\(\bar{\ell}^2(S,M)\) by left
multiplication.  Since~\(\bar{\ell}^2(S,M)\) is self-dual, any
bounded \(M\)\nb-linear operator on it is adjointable, and these
operators form a \(\Wst\)\nb-algebra \(\Bound(\bar{\ell}^2(S,M))\).

\begin{definition}
  \label{def:Wstar_crossed_product}
  The \emph{\(\Wst\)\nb-algebra crossed product}
  \(M\barrtimes S\) for an action~\((\Hilm_t,\mu_{t,u})\)
  of~\(S\) on a \(\Wst\)\nb-algebra~\(M\) by self-dual Hilbert
  bimodules is defined as the weak closure of
  \(\lambda(M\rtimes_\alg S)\) in the \(\Wst\)\nb-algebra
  \(\Bound(\bar{\ell}^2(S,M))\).
\end{definition}

By construction, \(\lambda\) gives an injective \Star{}homomorphism
\(M\rtimes_\alg S \hookrightarrow M\barrtimes S\).  The
map~\(E\) above extends to a faithful conditional expectation
\(M\barrtimes S\to M\), namely, \(E(T) = \iota^* \circ
T\circ \iota\), where \(\iota\colon M\to \bar\ell^2(S,M)\) is the
inclusion of the summand \(M=\Hilm_1\).

\section{The reduced crossed product and induction}
\label{sec:reduced_Cstar-crossed}

Now we return to the \(\Cst\)\nb-algebraic
case.  Let~\((\Hilm_t,\mu_{t,u})\)
be an action of~\(S\)
on a \(\Cst\)\nb-algebra~\(A\)
by Hilbert bimodules.  Then \((\bid{\Hilm_t}, \bid{\mu_{t,u}})\)
is an action of~\(S\)
on~\(\bid{A}\)
by self-dual Hilbert bimodules.  The representation of
\(\bid{A}\rtimes_\alg S\)
on~\(\bar{\ell}^2(S,\bid{A})\)
restricts to a \Star{}homomorphism on~\(A\rtimes_\alg S\).
This extends to a \Star{}homomorphism
\(A\rtimes S\to \Bound(\bar{\ell}^2(S,\bid{A}))\).

\begin{definition}
  \label{def:reduced_Cstar-crossed}
  The \emph{reduced crossed product}~\(A\rtimes_\red S\) is the
  image of~\(A\rtimes S\) in \(\Bound(\bar\ell^2(S,\bid{A}))\), the
  \(\Wst\)\nb-algebra of adjointable operators on
  \(\bar\ell^2(S,\bid{A})\).
\end{definition}

By construction, \(A\rtimes_\red S\) is contained in the
\(\Wst\)\nb-algebra crossed product \(\bid A\barrtimes S\).

\begin{remark}
  \label{rem:barell2_versus_ell2}
  Every adjointable operator on~\(\ell^2(S,\bid A)\) extends
  uniquely to the self-dual completion~\(\bar{\ell}^2(S,\bid A)\),
  and this gives a unital embedding of \(\Bound(\ell^2(S,\bid A))\)
  into \(\Bound(\bar{\ell}^2(S,\bid A))\).  Since \(A\rtimes_\alg
  S\) and hence also the image of \(A\rtimes S\) map the Hilbert
  submodule~\(\ell^2(S,\bid A)\) into itself by adjointable
  operators, \(A\rtimes_\red S\) is contained in the image of
  \(\Bound(\ell^2(S,\bid A))\) in \(\Bound(\bar{\ell}^2(S,\bid
  A))\).  Hence it makes no difference whether we use
  \(\ell^2(S,\bid A)\) or \(\bar{\ell}^2(S,\bid A)\) to
  define~\(A\rtimes_\red S\).
\end{remark}

\begin{proposition}
  \label{pro:embed_alg_red}
  The canonical \Star{}homomorphisms from~\(A\rtimes_\alg S\) to
  \(\bid{A}\rtimes_\alg S\), \(A\rtimes_\red S\), and \(A\rtimes S\)
  are injective.
\end{proposition}

\begin{proof}
  By definition, \(A\rtimes_\alg S\) is a quotient of the direct sum
  \(\bigoplus_{t\in S} \Hilm_t\).  Let \(F\subseteq S\) be a finite
  subset.  We are going to prove that the kernel of the map
  \(\bigoplus_{t\in F} \Hilm_t \to \bid{A}\rtimes_\alg S\) is the
  linear span~\(W_F\) of \(\theta_{t,u}(x)\delta_t - x\delta_u\) for
  \(t,u\in F\), \(x\in \Hilm_u\cdot I_{t,u}\).  Therefore, the
  canonical \Star{}homomorphism from~\(A\rtimes_\alg S\) to
  \(\bid{A} \rtimes_\alg S\) is injective.  Then the map
  from~\(A\rtimes_\alg S\) to~\(A\rtimes_\red S\) is injective
  because \(A\rtimes_\red S \hookrightarrow \bid{A} \barrtimes S
  \hookleftarrow \bid{A}\rtimes_\alg S\).  And then the map
  from~\(A\rtimes_\alg S\) to~\(A\rtimes S\) is injective because
  the injective map to~\(A\rtimes_\red S\) factors through it.  Thus
  everything follows from the above description of the kernel of the
  map \(\bigoplus_{t\in F} \Hilm_t \to \bid{A}\rtimes_\alg S\).

  Let~\(F'\) be another finite subset of~\(S\) with \(F\subseteq F'\),
  and assume that \(\xi \defeq \sum_{t,u\in F'}
  \theta_{t,u}(x_{t,u})\delta_t - x_{t,u}\delta_u\) with \(x_{t,u}\in
  \Hilm_u\cdot I_{t,u}\) for all \(t,u\in F'\) belongs to
  \(\bigoplus_{t\in F} \Hilm_t\delta_t\).  We claim that \(\xi\in
  W_F\), that is, we may rewrite \(\xi= \sum_{t,u\in F}
  \theta_{t,u}(y_{t,u})\delta_t - y_{t,u}\delta_u\) with \(y_{t,u}\in
  \Hilm_u\cdot I_{t,u}\) for all \(t,u\in F\) with \(t,u\in F\)
  instead of \(t,u\in F'\).  There is nothing to prove if \(F=F'\), so
  choose some \(v\in F'\setminus F\).  It suffices to rewrite~\(\xi\)
  as a sum over \(t,u\in F'\setminus\{v\}\): if we can always do this,
  then we may go on and remove the other elements of \(F'\setminus
  F\), until we bring~\(\xi\) into the desired form.  Thus we may
  assume without loss of generality that \(F'= F\cup\{v\}\).

  Since \(\theta_{t,u}= \theta_{u,t}^{-1}\), we may assume that the
  summands in~\(\xi\) containing~\(v\) all have~\(v\) as the second
  entry.  Any summand with \(t=u=v\) is~\(0\) because
  \(\theta_{v,v}=\Id_{\Hilm_v}\), so we may remove such a summand from
  our representation of~\(\xi\).  Let~\(F''\) be the set of \(t\in
  F'\) for which~\(\xi\) contains a summand of the form
  \(\theta_{t,v}(x_{t,v})\delta_t - x_{t,v}\delta_v\).  Since
  \(v\notin F''\), we have \(F''\subseteq F\).  There is nothing to do
  if~\(F''\) is empty.  So we assume that~\(F''\) is non-empty and pick
  some \(w\in F''\).  We are going to rewrite~\(\xi\) so that only
  summands for~\((t,v)\) with \(t\in F''\setminus\{w\}\) appear.  If
  we can do this, we may repeat this step and remove all points
  from~\(F''\), until we arrive at a sum that does not involve~\(v\)
  any more.  Thus it suffices to prove that we may reduce~\(F''\)
  to~\(F''\setminus\{w\}\).

  The \(\delta_v\)\nb-component of~\(\xi\) is the sum \(\sum_{t\in
    F''} x_{t,v} \delta_v\).  This must vanish because \(v\notin F\).
  Thus \(x_{w,v} = - \sum_{t\in F''\setminus \{w\}} x_{t,v}\).  This
  belongs to \(\sum_{t\in F''\setminus \{w\}} \Hilm_v\cdot I_{t,v}\)
  and to \(\Hilm_v \cdot I_{w,v}\).  This intersection is
  \(\Hilm_v\cdot I\) with \(I = I_{w,v} \cap \sum_{t\in F''\setminus
    \{w\}} I_{t,v} = \sum_{t\in F''\setminus \{w\}} (I_{w,v} \cap
  I_{t,v})\) because the map \(I\mapsto \Hilm_v\cdot I\) is a
  lattice isomorphism from the lattice of ideals \(I\idealin
  \s(\Hilm_v)\) onto the lattice of Hilbert subbimodules
  in~\(\Hilm_v\).
  Thus we may rewrite \(x_{w,v} = \sum_{t\in F''\setminus
    \{w\}} x_{w,v,t}\) with \(x_{w,v,t} \in I_{w,v} \cap I_{t,v}\).
  Then \(\theta_{w,v}(x_{w,v,t}) =
  \theta_{w,t}\theta_{t,v}(x_{w,v,t})\)
  by Lemma~\ref{lem:intersect_Hilm_tu}, so that
  \[
  \theta_{w,v}(x_{w,v,t})\delta_w - x_{w,v,t}\delta_v
  = \theta_{w,t}\theta_{t,v}(x_{w,v,t})\delta_w - \theta_{t,v}(x_{w,v,t})\delta_t
  + \theta_{t,v}(x_{w,v,t})\delta_t - x_{w,v,t}\delta_v.
  \]
  When we substitute this in~\(\xi\) for all \(t\in
  F''\setminus\{w\}\), we replace the summand
  \(\theta_{w,v}(x_{w,v,t})\delta_w - x_{w,v,t}\delta_v\)
  for~\((w,v)\) by summands for \((w,t)\) and~\((t,v)\) for \(t\in
  F''\setminus\{w\}\).  Since \(t,w\in F\), this achieves the
  reduction step that we still need.  This finishes the proof
  that~\(W_F\) is the set of all finite linear combinations of
  \(\theta_{t,u}(x_{t,u})\delta_t - x_{t,u}\delta_u\) with
  \(x_{t,u}\in \Hilm_u\cdot I_{t,u}\) and \(t,u\in S\) that belong to
  \(\bigoplus_{t\in F} \Hilm_t\delta_t\).

  Next we prove that \(W_F\subseteq \bigoplus_{t\in F}
  \Hilm_t\delta_t\) is closed in the norm topology for each finite
  subset \(F\subseteq S\).

  We prove this by induction on the size of~\(F\).  If~\(F\) is empty,
  the assertion is trivial.  So let \(\abs{F}\ge 1\) and pick \(t\in
  F\).  Let \(F'\defeq F\setminus \{t\}\) and assume that~\(W_{F'}\)
  is norm closed.  Let \(I = \sum_{u\in F'} I_{t,u}\).  This is a
  closed ideal in~\(A\), as a sum of finitely many closed ideals.
  Hence \(\Hilm_t \cdot I\subseteq \Hilm_t\) is closed and contains
  \(\Hilm_t \cdot I_{t,u}\) for all \(u\in F'\).  As in the proof of
  the claim above, the closure of~\(W_F\) can only contain
  \(\sum_{u\in F} x_u \delta_u \in \bigoplus_{u\in F} \Hilm_u\) if
  \(x_t\in \Hilm_t\cdot I\).  In that case, we may write \(x_t =
  x_t^0\cdot a\) with \(x_t^0\in \Hils_t\), \(a\in I\), and
  \(\norm{x_t} = \norm{x_t^0}\), \(\norm{a}< 1+\varepsilon\) for any
  \(\varepsilon>0\).  Moreover, \(a=\sum_{u\in F'} a_u\) with \(a_u\in
  I_{t,u}\) and \(\norm{a_u}< 1+\varepsilon\) for all \(u\in F'\).
  Then
  \[
  \sum_{u\in F} x_u \delta_u - \sum_{u\in F'}
  (\theta_{u,t}(x_t^0\cdot a_u)\delta_u - x_t^0\cdot a_u\delta_t)
  \in \bigoplus_{u\in F'} \Hilm_u.
  \]
  This term no longer involves the summand~\(\Hilm_t\).  By our first
  claim above, this sum belongs to~\(W_{F'}\).

  The argument above shows that \(W_F/W_{F'} \cong \Hilm_t\cdot I\),
  where the quotient norm is equivalent to the norm from~\(\Hilm_t\).
  Since~\(W_{F'}\) is closed by induction assumption, \(W_F\) is an extension of
  the Banach space \(\Hilm_t\cdot I\) by the Banach space~\(W_{F'}\).
  This implies that~\(W_F\) is complete in the subspace topology
  from \(\bigoplus_{u\in F} \Hilm_u\), so that~\(W_F\) must be closed
  in the norm topology.

  The range and source submodules of the partial isometries
  \(\Theta_{t,u}\colon \Hilm_u\to\Hilm_t\) that appear in the
  definition of~\(\bid{A}\rtimes_\alg S\) are the weak closures of
  \(\Hilm_t\cdot I_{t,u}\) and \(\Hilm_u\cdot I_{t,u}\), respectively.
  Therefore, the kernel of the map \(\bigoplus_{t\in F} \Hilm_t \to
  \bid{A} \rtimes_\alg S\) is contained in the weak closure
  of~\(W_F\).  By the Hahn--Banach Theorem, a subspace of
  the Banach space \(\bigoplus_{t\in F} \Hilm_t\) is weakly closed if
  and only if it is norm closed: the norm and the weak topology have
  the same closed convex subsets.  We have shown that~\(W_F\) is norm
  closed, hence weakly closed.  This finishes the proof.
\end{proof}

By definition, \(\ell^2(S,\bid{A})\)
is a \(\Cst\)\nb-correspondence
from~\(A\rtimes_\red S\)
to~\(\bid{A}\).
As explained in \cite{Rieffel:Morita}*{p.~65}, this
\(\Cst\)\nb-correspondence
gives a functor from the \(\Wst\)\nb-category
of Hilbert space representations of~\(A\)
to that of~\(A\rtimes_\red S\):

\begin{definition}
  \label{def:induction}
  The \emph{induction} functor~\(\Ind\)
  from representations of~\(A\)
  to representations of~\(A\rtimes_\red S\)
  maps a representation \(\pi\colon A\to \Bound(\Hils)\)
  to the representation of \(A\rtimes_\red S\)
  on \(\ell^2(S,\bid{A})\otimes_{\bid{\pi}} \Hils\),
  where~\(\bid{\pi}\)
  is the unique weakly continuous extension of~\(\pi\) to~\(\bid{A}\).
\end{definition}

The \(\Cst\)\nb-norm
on~\(A\rtimes_\red S\)
is the supremum of the norms in~\(\Ind \pi\)
for all representations~\(\pi\)
of~\(A\).
We may also take a single faithful representation of~\(\bid{A}\).
For instance, the direct sum of the GNS-representations for all states
of~\(A\)
gives a faithful representation of~\(\bid{A}\).
Hence the norm defining~\(A\rtimes_\red S\)
is equal to the supremum of the norm in~\(\Ind \pi\),
where~\(\pi\)
now runs through the set of GNS-representations for all states
on~\(A\).

But is it enough to take only the irreducible representations as
in~\cite{Exel:noncomm.cartan}?  Can we even take any faithful
representation of~\(A\)?
Notice that the direct sum of all irreducible representations of~\(A\)
is always faithful on~\(A\),
but its extension to~\(\bid{A}\)
need not be faithful (see
\cite{Pedersen:Cstar_automorphisms}*{4.3.11}).

To answer the above questions (the first positively, the second
negatively), we study the range of the map \(E\colon A\rtimes_\red
S\to \bid{A}\).  By definition, \(E\) is the restriction
to~\(A\rtimes_\red S\) of the conditional expectation \(\bid
A\barrtimes S\to \bid A\) constructed in the last section,
see~\ref{eq:cond.exp}.  We shall sometimes view~\(E\) as a map from
the full crossed product~\(A\rtimes S\) to~\(\bid A\), and call it
the \emph{weak conditional expectation} associated to the action.
The following lemma describes~\(E\) in the \(\Cst\)\nb-algebraic
setting:

\begin{lemma}
  \label{lem:description-E}
  Let \(t\in S\) and let~\(\theta_{1,t}\) denote the canonical
  isomorphism \(\Hils_t\cdot I_{1,t}\congto I_{1,t}\) as in
  Lemma~\textup{\ref{lem:intersect_Hilm_tu}}.  View the multiplier
  algebra of~\(I_{1,t}\) as a subalgebra of~\(\bid A\) in the usual
  way.  The map \(E\colon A\rtimes_\red S\to \bid{A}\) maps
  \(\Hils_t\subseteq A\rtimes_\red S\) into \(\Mult(I_{1,t})
  \subseteq \bid{A}\).  More precisely, \(E(\xi\delta_t)\) for
  \(\xi\in \Hils_t\) is the multiplier of~\(I_{1,t}\) given by
  \(E(\xi\delta_t)x=\theta_{1,t}(\xi\cdot x)\) for all \(x\in
  I_{1,t}\).  If~\((u_{t,i})\) is an approximate unit
  for~\(I_{1,t}\), then
  \begin{equation}
    \label{eq:formula-cond.exp}
    E(\xi\delta_t)=\stlim_{i}\theta_{1,t}(\xi\cdot u_{t,i}),
  \end{equation}
  where \(\stlim\) denotes the limit in the strict topology
  on~\(\Mult(I_{1,t})\).
\end{lemma}

\begin{proof}
  By definition, \(E(\xi) = \Theta_{1,t}(\xi[I_{1,t}])\),
  where~\(\Theta_{1,t}\) is the extension (as in
  Lemma~\ref{lem:intersect_Hilm_tu_self-dual}) of the isomorphism
  \(\theta_{1,t}\colon \Hilm_t\cdot I_{1,t}\to \Hilm_1\cdot
  I_{1,t}=I_{1,t}\) described in Lemma~\ref{lem:intersect_Hilm_tu},
  and~\([I_{1,t}]\) is the support projection of the
  ideal~\(I_{1,t}\).  The ideal~\(\bar{I}_{1,t}\) in
  Lemma~\ref{lem:intersect_Hilm_tu} is the bidual or, equivalently,
  the weak closure in~\(\bid A\), of the ideal
  \(I_{1,t}\idealin A\).  By construction, \(\Theta_{1,t}\) is
  an isomorphism \(\bid{\Hilm_t}\cdot [I_{1,t}] \cong \bid{A}\cdot
  [I_{1,t}] \subseteq \bid{A}\).

  The isomorphism \(\theta_{1,t}\colon \Hilm_t\cdot I_{1,t}\to
  I_{1,t}\) induces an isomorphism of multiplier modules:
  \[
  \Mult(\theta_{1,t})\colon \Mult(\Hilm_t\cdot I_{1,t})
  \defeq \Bound(I_{1,t},\Hilm_t\cdot I_{1,t})
  \xrightarrow[\cong]{(\theta_{1,t})_*}
  \Bound(I_{1,t},I_{1,t}) = \Mult(I_{1,t}).
  \]
  There is a canonical map \(\Hilm_t\to \Mult(\Hilm_t\cdot
  I_{1,t})\), sending \(\xi\in \Hilm_t\) to the multiplier
  \(\hat\xi\in \Mult(\Hilm_t\cdot I_{1,t})\) given by
  \(\hat\xi(x)\defeq \xi\cdot x\) for \(x\in I_{1,t}\).  Thus
  \(\Mult(\theta_{1,t})(\hat{\xi})a = \theta_{1,t}(\xi a)\) for all
  \(a\in I_{1,t}\).  The map \(\Mult(\theta_{1,t})\) is the unique
  strictly continuous extension of~\(\theta_{1,t}\),
  and~\(\Theta_{1,t}\) is the unique weakly continuous extension
  of~\(\theta_{1,t}\).  The obvious embedding \(\Mult(I_{1,t})
  \hookrightarrow \bid{I_{1,t}} \hookrightarrow \bid{A}\) is
  continuous from the strict to the weak topology.
  Hence~\(\Theta_{1,t}\) extends \(\Mult(\theta_{1,t})\), so that
  \(E(\xi) = \Theta_{1,t}(\xi) = \Mult(\theta_{1,t})(\hat\xi)\in
  \Mult(I_{1,t})\) for all \(\xi\in\Hilm_t\).  This proves the first
  assertion of the lemma.  Formula~\eqref{eq:formula-cond.exp}
  follows because \(E(\xi\delta_t)\in\Mult(I_{1,t})\) implies
  \(E(\xi\delta_t) = \stlim_i E(\xi\delta_t)u_{t,i} = \stlim_i
  \theta_{1,t}(\xi\cdot u_{t,i})\).
\end{proof}

\begin{remark}
  \label{rem:graded}
  The weak conditional expectation \(E\colon A\rtimes_\red S\to \bid
  A\) is faithful and is the identity on~\(A\).  Hence the canonical
  maps \(\Hils_t\to A\rtimes_\red S\), \(\xi\mapsto \xi\delta_t\),
  are isometric (this is also proved in~\cite{Exel:noncomm.cartan}),
  and the same holds for~\(A\rtimes S\).
  These maps form representations of the action~\((\Hils_t)_{t\in
    S}\), and turn both \(A\rtimes_\red S\) and~\(A\rtimes S\) into
  \(S\)\nb-graded \(\Cst\)\nb-algebras with copies of~\(\Hilm_t\) as
  the subspaces of the grading.  The subspaces of a grading over an
  inverse semigroup that is not a group have non-trivial intersection
  and so are not linearly independent.
\end{remark}

\begin{notation}
  Let \(\tilde{A}\subseteq \bid{A}\)
  be the \(\Cst\)\nb-subalgebra
  generated by \(E(A\rtimes_\red S) \subseteq \bid A\).
  We have \(\tilde A\supseteq A\)
  because \(E|_A=\Id_A\).
  For a representation \(\pi\colon A\to\Bound(\Hils)\),
  let~\(\tilde{\pi}\)
  be the restriction of \(\bid{\pi}\colon \bid{A}\to\Bound(\Hils)\)
  to~\(\tilde{A}\).
\end{notation}

\begin{definition}
  \label{def:faithful_for_E}
  A family of representations~\((\pi_i)_{i\in I}\) of~\(A\)
  is \emph{\(E\)\nb-faithful} if the representation \(\bigoplus_{i\in I} \tilde{\pi}_i\)
  of~\(\tilde{A}\) is faithful.
\end{definition}

\begin{proposition}
  \label{pro:E-faithful}
  If the family of representations~\((\pi_i)_{i\in I}\)
  is \(E\)\nb-faithful,
  then the representation \(\bigoplus_{i\in I} \Ind \pi_i\)
  of~\(A\rtimes_\red S\) is faithful.
\end{proposition}

\begin{proof}
  We may replace the family~\((\pi_i)_{i\in I}\)
  by the single representation \(\pi =\bigoplus \pi_i\).
  We may use the \(\tilde{A}\)\nb-valued
  expectation on~\(A\rtimes S\)
  to construct a Hilbert \(\tilde{A}\)\nb-module
  \(\ell^2(S,\tilde{A})\).
  Then
  \(\ell^2(S,\tilde{A}) \otimes_{\tilde{A}} \bid{A} \cong
  \ell^2(S,\bid{A})\)
  as correspondences from~\(A\rtimes_\red S\)
  to~\(\bid{A}\).  The resulting map
  \[
  \Bound(\ell^2(S,\tilde{A})) \to \Bound(\ell^2(S,\bid{A})),\qquad
  T\mapsto T\otimes 1_{\bid{A}},
  \]
  is injective because the map \(\tilde{A} \to \bid{A}\)
  is injective.  Its image contains the image
  of~\(A\rtimes_\alg S\)
  and hence of~\(A\rtimes_\red S\).
  Hence we may as well define~\(A\rtimes_\red S\)
  as the \(\Cst\)\nb-subalgebra
  of \(\Bound(\ell^2(S,\tilde{A}))\)
  generated by \(A\rtimes_\alg S\).  Moreover,
  \[
  \ell^2(S,\bid{A}) \otimes_{\bid{A}} \bid{\pi}
  \cong \ell^2(S,\tilde{A}) \otimes_{\tilde{A}} \bid{A}
  \otimes_{\bid{A}} \bid{\pi}
  \cong \ell^2(S,\tilde{A}) \otimes_{\tilde{A}} \bid{\pi}|_{\tilde{A}}
  = \ell^2(S,\tilde{A}) \otimes_{\tilde{A}} \tilde{\pi}.
  \]
  If~\(\tilde{\pi}\)
  is faithful, then the induced representation of
  \(\Bound(\ell^2(S,\tilde{A}))\)
  on \(\ell^2(S,\tilde{A}) \otimes_{\tilde{A}} \tilde{\pi}\)
  is also faithful.  Hence~\(\Ind \pi\)
  is a faithful representation of~\(A\rtimes_\red S\).
\end{proof}

\begin{remark}
  \label{rem:induction_accidentally_faithful}
  It can easily happen that the induced representation~\(\Ind\pi\)
  is faithful although~\(\pi\)
  is not faithful, even without the difficulty of extending from~\(A\)
  to~\(\tilde{A}\).
  Consider a finite group~\(\Gamma\)
  with \(n>1\)
  elements and let it act on \(A=\Cont(\Gamma)\).
  Then \(A\rtimes\Gamma\cong \Mat_n\C\)
  has only faithful nonzero representations.  If \(\pi\colon A\to\C\)
  is a character, then~\(\pi\)
  is not faithful, but \(\Ind\pi\)
  is an irreducible faithful representation of~\(A\rtimes\Gamma\).
\end{remark}

To turn Proposition~\ref{pro:E-faithful} into a useful criterion, we
need to understand~\(\tilde{A}\)
better.  We are particularly interested in when \(\tilde{A}=A\), that is,
when \(E(A\rtimes S)\subseteq A\) so that~\(E\) becomes an
ordinary conditional expectation \(A\rtimes S\to A\).

\section{The commutative case}
\label{sec:commutative_E}

First we consider the special case where~\(A\) is commutative.  Our
constructions are a minor generalisation of those by Khoshkam and
Skandalis in~\cite{Khoshkam-Skandalis:Regular}; the only difference
is that we also allow groupoid \(\Cst\)\nb-algebras twisted by Fell
line bundles.  Example~2.5 in~\cite{Khoshkam-Skandalis:Regular}
shows that \(\Ind\pi\) need not be faithful if~\(\pi\) is a faithful
representation of~\(A\).  So our problem is non-trivial.

Let \(A=\Cont_0(X)\)
be a commutative \(\Cst\)\nb-algebra.  Let~\(\Hilm\)
be a Hilbert \(A\)\nb-bimodule.
The ideals \(\s(\Hilm)\)
and \(\rg(\Hilm)\)
correspond to open subsets \(U\)
and~\(V\)
in~\(X\),
respectively, and~\(\Hilm\)
is an imprimitivity bimodule between \(\Cont_0(U)\)
and~\(\Cont_0(V)\).
A Hilbert module over~\(\Cont_0(V)\)
is equivalent to a continuous field of Hilbert spaces over~\(V\).
Since we want the compact operators on this field of Hilbert spaces to
be isomorphic to~\(\Cont_0(U)\),
hence commutative, this continuous field must be a complex line
bundle; equivalently, each fibre has dimension~\(1\).
The left action must map \(f\in\Cont_0(U)\)
to the operator that multiplies pointwise with the function
\(f\circ\alpha^{-1}\)
for some homeomorphism \(\alpha\colon U\to V\).
The Hilbert bimodule~\(\Hilm\)
and the data \((U,V,L,\alpha)\)
determine each other uniquely up to isomorphism.  The triple
\((U,V,\alpha)\)
is a partial homeomorphism of~\(X\).
Thus a Hilbert bimodule over~\(\Cont_0(X)\)
is equivalent to a partial homeomorphism of~\(X\)
together with a line bundle on its (co)domain.

Now we generalise the above discussion and consider an action of a
unital inverse semigroup~\(S\) on~\(A=\Cont_0(X)\).  This is
equivalent to a saturated Fell bundle over~\(S\) with unit
fibre~\(A\).
Following~\cite{BussExel:Fell.Bundle.and.Twisted.Groupoids}, we now
describe these in terms of twisted étale groupoids, that is, Fell
line bundles over étale groupoids.  An action of~\(S\) on~\(A\) is
equivalent to partial homeomorphisms \(\psi_t\colon D_{t^*}\to D_t\)
for certain open subsets \(D_t\subseteq X\) with line
bundles~\(L_t\) over~\(D_t\) for all \(t\in S\), together with
suitable multiplication isomorphisms for \(t,u\in S\).  These
multiplication isomorphisms can only exist if the partial
homeomorphisms~\((\psi_t)_{t\in S}\) form an inverse semigroup
action on~\(X\), that is, \(\psi_t\circ \psi_u = \psi_{t u}\) for
all \(t,u\in S\): this is the action on the primitive ideal space
of~\(A\) induced by an action on~\(A\) given by
\cite{Buss-Meyer:Actions_groupoids}*{Lemma 6.12}.  Hence we may form
the transformation groupoid \(X\rtimes S\), which is an \'etale,
possibly non-Hausdorff, groupoid with object space~\(X\).

The complex line bundles~\(L_t\) with their multiplication maps are
equivalent to a Fell line bundle~\(\mathcal{L}\) over the
groupoid~\(X\rtimes S\), that is, a Fell bundle over~\(X\rtimes S\)
with only \(1\)\nb-dimensional fibres.  This follows from
\cite{Buss-Meyer:Actions_groupoids}*{Theorem 6.13}, which shows that
the action of~\(S\) on~\(A\) comes from a Fell bundle over the
groupoid~\(X\rtimes S\) with unit fibre~\(A\).  For every Hausdorff
open subset \(U\subseteq (X\rtimes S)^1\), the
\(\Cont_0\)\nb-sections of the Fell bundle over~\(U\) form a
\(\Cont_0(U)\)-linear imprimitivity bimodule between \(\Cont_0(U)\)
and itself.  As above, this must be the space of sections of a line
bundle over~\(U\).  Since these Hausdorff open subsets
cover~\((X\rtimes S)^1\), we get a line bundle~\(\mathcal{L}\) over
all of~\((X\rtimes S)^1\).  The multiplication on quasi-continuous
sections of the Fell bundle induces the appropriate multiplication
between the fibres of~\(\mathcal{L}\).

An element of~\(\Hilm_t\) is a \(\Cont_0\)\nb-section of the line
bundle~\(L_t\) over~\(D_t\).  We may identify~\(\Hilm_t\) with the
space of all continuous \(\Cont_0\)\nb-sections of~\(\mathcal{L}\) on
the bisection~\((X\rtimes S)_t\) of the arrow space of~\(X\rtimes S\)
corresponding to \(t\in S\).

A section of the line bundle~\(\mathcal{L}\) over~\((X\rtimes S)^1\)
is called \emph{quasi-continuous} if it is a finite linear
combination of \(\Cont_0\)\nb-sections on Hausdorff, open
subsets~\(U\) of~\(X\rtimes S\), extended by~\(0\) outside~\(U\).
Let \(\Sect(X\rtimes S,\mathcal{L})\) be the space of
quasi-continuous sections.  The section \(\Cst\)\nb-algebra of the
Fell line bundle~\(\mathcal{L}\) is defined as the completion of
\(\Sect(X\rtimes S,\mathcal{L})\) in the maximal
\(\Cst\)\nb-seminorm.  \(\Cst\)\nb-algebras of this type may be
viewed as \emph{twisted} groupoid \(\Cst\)\nb-algebras (see
\cite{Renault:Cartan.Subalgebras}); they are the groupoid analogues
of \emph{twisted} group \(\Cst\)\nb-algebras.

Elements of~\(\Hilm_t\) extended by~\(0\) outside~\((X\rtimes S)_t\)
give elements of \(\Sect(X\rtimes S,\mathcal{L})\).  The
intersection of \((X\rtimes S)_t\) and \((X\rtimes S)_u\) for
\(t,u\in S\) corresponds to the open subset \(D_{t, u} \defeq
\bigcup_{v\le t,u} D_v\), and \(\Cont_0(D_{t,u})\) is the
ideal~\(I_{t,u}\) in~\(\Cont_0(X)\).  Hence the maps \(\Hilm_t \to
\Sect(X\rtimes S,\mathcal{L})\) defined above induce a
\Star{}isomorphism \(A\rtimes_\alg S \to \Sect(X\rtimes
S,\mathcal{L})\) by \cite{Buss-Meyer:Actions_groupoids}*{Proposition
  B.2}.  Hence \(A\rtimes S = \Cst(X\rtimes S,\mathcal{L})\) (see
also \cite{Buss-Meyer:Actions_groupoids}*{Corollary~5.6} and
\cite{BussExel:Fell.Bundle.and.Twisted.Groupoids}*{Proposition~2.14}).

The bidual~\(\bid{A}\) of~\(A\) contains the \(\Cst\)\nb-algebra
\(\Borel(X)\) of bounded Borel functions on~\(X\) because any
representation of~\(\Cont_0(X)\) extends uniquely to a
representation of~\(\Borel(X)\).  Any ideal in~\(\Cont_0(X)\) is of
the form~\(\Cont_0(U)\) for an open subset \(U\subseteq X\).  Its
multiplier algebra is \(\Contb(U)\), and any function
in~\(\Contb(U)\), extended by~\(0\) outside~\(U\), is a Borel
function on~\(X\).  Now Lemma~\ref{lem:description-E} shows that the
image of the weak conditional expectation \(E\colon
\Cont_0(X)\rtimes S \to \bid{\Cont_0(X)}\) is contained
in~\(\Borel(X)\).  So we may as well work in the more concrete
subalgebra \(\Borel(X) \subseteq \bid{\Cont_0(X)}\).

When we identify \(A\rtimes_\alg S\cong \Sect(X\rtimes
S,\mathcal{L})\), then the map \(E\colon A\rtimes_\alg S\to
\bid{A}\) simply restricts sections of~\(\mathcal{L}\) to the unit
fibre \(X\subseteq (X\rtimes S)^1\); this gives scalar-valued Borel
functions on~\(X\) through a canonical isomorphism
\(\mathcal{L}|_X\cong \C\times X\) for any Fell line bundle.  We can
now describe~\(\tilde{A}\) and decide when~\(E\) takes values
in~\(A\), that is, when \(A=\tilde{A}\):

\begin{proposition}
  \label{pro:tilde_A_commutative}
  Let \(A=\Cont_0(X)\) equipped with an action of a unital inverse
  semigroup~\(S\).  The \(\Cst\)\nb-algebra~\(\tilde{A}\) is the
  \(\Cst\)\nb-subalgebra of \(\Borel(X)\) that is generated by
  functions of the form~\(f|_{X\cap U}\), extended by zero on
  \(X\setminus (X\cap U)\), for Hausdorff, open subsets \(U\subseteq
  (X\rtimes S)^1\) and \(f\in \Contc(U)\).

  The conditional expectation \(E\colon A\rtimes S\to \bid{A}\)
  takes values in~\(A\)
  if and only if the associated transformation groupoid~\(X\rtimes S\)
  is Hausdorff.
\end{proposition}

\begin{proof}
  The first statement about~\(\tilde{A}\)
  is clear from our description of the conditional expectation~\(E\).
  We have \(A=\tilde{A}\)
  if and only if all functions of the form~\(f|_{X\cap U}\)
  for Hausdorff, open subsets \(U\subseteq (X\rtimes S)^1\)
  and \(f\in \Contc(U)\)
  are still continuous on~\(X\).
  This is the case if and only if~\(X\)
  is a closed subset of~\((X\rtimes S)^1\).
  This is equivalent to~\((X\rtimes S)^1\)
  being Hausdorff, see Lemma~\ref{lem:G_Hausdorff_iff_G0_Closed}.
\end{proof}

\begin{lemma}
  \label{lem:G_Hausdorff_iff_G0_Closed}
  A topological groupoid~\(G\)
  is Hausdorff if and only if its unit space~\(G^0\)
  is Hausdorff and closed as a subset of~\(G^1\).
\end{lemma}

\begin{proof}
  If~\(G^1\)
  is Hausdorff, so is the subspace~\(G^0\).
  Since \(G^0 = \{g \in G^1 \mid g = 1_{\s(g)}\}\)
  and the map \(g\mapsto 1_{\s(g)}\)
  on~\(G^1\)
  is continuous, the units~\(G^0\)
  form a closed subset of~\(G^1\) if~\(G^1\) is Hausdorff.

  Conversely, assume that~\(G^1\) is not Hausdorff.  Then there is a
  net \((g_i)_{i\in I}\) in~\(G^1\) with two different limit points
  \(g,h\in G^1\).  Since the range, inversion and multiplication
  maps are continuous, \(\rg(g)=\lim_i\rg(g_i)=\rg(h)\) and the net
  \(g_i^{-1}\cdot g_i\), which lies in~\(G^0\), converges both to
  \(g^{-1}\cdot g \in G^0\) and to \(g^{-1}\cdot h \neq g^{-1}\cdot
  g\).  If \(g^{-1}\cdot h\in G^0\), then~\(G^0\) is not Hausdorff.
  If \(g^{-1}\cdot h\notin G^0\), then~\(G^0\) is not closed
  in~\(G^1\).
\end{proof}

The same \(\Cst\)\nb-algebra~\(\tilde{A}\)
is used in~\cite{Khoshkam-Skandalis:Regular} to define the regular
representation of \(\Cont_0(X)\rtimes S = \Cst(G)\)
on a certain Hilbert \(\tilde{A}\)\nb-module.
This coincides with our regular representation.  Since~\(\tilde{A}\)
is commutative, it is isomorphic to~\(\Cont_0(Y)\)
for a certain space~\(Y\).
The inclusion map \(A\to\tilde{A}\)
gives a continuous map \(Y\to X\).
Since Borel functions are, in particular, functions on~\(X\),
we also get a map \(X\to Y\)
using evaluation homomorphisms; but this map need not be continuous
(see~\cite{Khoshkam-Skandalis:Regular}).

Our description of~\(\tilde{A}\) allows us to characterise when a
representation~\(\pi\) of~\(\Cont_0(X)\) is \(E\)\nb-faithful: this
means that the resulting representation of~\(\Cont_0(Y)\) is
faithful, that is, its ``support'' is dense in~\(Y\).  This
criterion is already obtained
in~\cite{Khoshkam-Skandalis:Regular}*{Corollary~2.11}.
Proposition~\ref{pro:faithful_rep_abstract} below is a
noncommutative analogue of this result.

\section{Conditional expectation in the noncommutative case}
\label{sec:noncommutative_E}

Now let~\(A\) be an arbitrary
\(\Cst\)\nb-algebra, equipped with an action of a unital inverse
semigroup~\(S\) by Hilbert bimodules.  We are going to characterise
when the conditional expectation~\(E\) is \(A\)\nb-valued, that is,
\(A=\tilde{A}\).  First we show by an example that~\(\tilde{A}\) may
become so complicated that a complete description is not a promising
goal.  It is easier to describe when \(A=\tilde{A}\).

\begin{example}
  \label{exa:action_compact_plus_unitary}
  Let~\(G\)
  be a group and let~\(S\)
  be the inverse semigroup obtained by adding a zero element to~\(G\).
  Let~\((u_g)_{g\in G}\)
  be a group representation of~\(G\)
  on a Hilbert space~\(\Hils\).
  Actually, it is enough to have a group homomorphism to the unitary
  group in the Calkin algebra \(\Bound(\Hils)/\Comp(\Hils)\),
  but already ordinary representations lead to rather complicated
  situations.  Let \(A=\Comp(\Hils)^+\),
  the unitisation of the \(\Cst\)\nb-algebra
  of compact operators.  The action of~\(S\)
  on~\(A\)
  is defined by taking \(\Hilm_0 \defeq \Comp(\Hils)\)
  and
  \(\Hilm_g \defeq \Comp(\Hils)\oplus \C\cdot u_g \subseteq
  \Bound(\Hils)\)
  for \(g\in G\).
  We clearly have \(\Hilm_g \Hilm_h = \Hilm_{gh}\)
  for \(g,h\in G\),
  and \(\Hilm_0 \Hilm_g = \Hilm_0 = \Hilm_g \Hilm_0\)
  and \(\Hilm_g^* = \Hilm_{g^{-1}}\)
  for all \(g\in G\)
  as well.  Hence we have got an action of~\(S\)
  on~\(A\) by Hilbert bimodules.

  The bidual~\(\bid{A}\)
  is naturally isomorphic to \(\bid{A}\cong \Bound(\Hils) \oplus \C\)
  because \(\Bound(\Hils) = \bid{\Comp(\Hils)}\),
  and \(1\in A\)
  is mapped to \((1,1)\).
  For \(t,u\in S\)
  with \(t\neq u\),
  there is always a unique element \(v\le t,u\),
  namely, \(v=0\).
  Hence \(I_{t,u}=\Comp(\Hils)\)
  for all \(t,u\in S\)
  with \(t\neq u\).
  Thus \(\bar{I}_{t,u} = \Bound(\Hils)\oplus 0\),
  and the projection onto this ideal is
  \((1,0)\in \Bound(\Hils)\oplus\C\).
  Since \(E\colon \bid{\Hilm_g} \to \bid{\Hilm_1}\)
  is weakly continuous, it must map
  \(k+\lambda u_g\mapsto (k+\lambda u_g,0)\in \Bound(\Hils)\oplus\C\)
  for all \(k\in\Comp(\Hilm)\),
  \(\lambda\in\C\),
  and \(g\neq 1\).
  Thus~\(\tilde{A}\)
  is the \(\Cst\)\nb-subalgebra
  of \(\Bound(\Hils)\oplus\C\)
  generated by \(\Comp(\Hilm)\)
  and the unitaries \(u_g\)
  for \(g\in G\setminus\{1\}\)
  in the first summand and by~\((1,1)\).  This gives
  \[
  \tilde{A} = (\Comp(\Hils) + \Cst(u_g\mid g\in G)) \oplus \C.
  \]
  Since any \(\Cst\)\nb-algebra is generated by the unitaries it
  contains, we may get any \(\Cst\)\nb-algebra
  containing~\(\Comp(\Hils)\) in the first summand.
\end{example}

Let~\(\Prim(A)\)
be the primitive ideal space of~\(A\).
The lattice of ideals in~\(A\) is isomorphic to the lattice of open
subsets of~\(\Prim(A)\) by
\cite{Dixmier:Cstar-algebras}*{Proposition 3.2.2}.
The action of~\(S\)
on~\(A\)
by Hilbert bimodules induces an action on~\(\Prim(A)\)
by partial homeomorphisms by
\cite{Buss-Meyer:Actions_groupoids}*{Lemma 6.12}.  Let
\(\alpha_t\colon D_{t^*} \to D_t\)
for \(t\in S\)
be the partial homeomorphism of~\(\Prim(A)\)
associated to \(t\in S\).
The open subsets \(D_{t^*}\)
and~\(D_t\)
of~\(\Prim(A)\)
correspond to the ideals \(\s(\Hilm_t)\)
and~\(\rg(\Hilm_t)\) in~\(A\).

\begin{lemma}
  \label{lem:I_tu_explicit}
  The ideal~\(I_{t,u}\)
  for \(t,u\in S\)
  corresponds to the open subset \(\bigcup_{v\le t,u} D_v\)
  in~\(\Prim(A)\).
  A representation \(\pi\colon A\to\Bound(\Hils)\)
  maps the central projection \([I_{t,u}]\in\bid{A}\)
  to the orthogonal projection onto the subspace
  \(\pi(I_{t,u}) \cdot\Hils\).
\end{lemma}

\begin{proof}
  Since the bijection between ideals in~\(A\) and open subsets
  in~\(\Prim(A)\)
  is a lattice isomorphism, the ideal generated by~\(\s(\Hilm_v)\)
  for \(v\le t,u\)
  corresponds to the union
  \(\bigcup_{v\le t,u} D_v \subseteq \Prim(A)\).
  The second statement holds because \([I]\in\bid{A}\)
  for an ideal \(I\idealin A\)
  is the weak limit of an approximate unit~\((e_\alpha)\)
  for~\(I\)
  and~\(\pi(e_\alpha)\)
  converges strongly to the orthogonal projection onto
  \(\pi(I)\cdot\Hils\).
\end{proof}

\begin{proposition}
  \label{pro:A_tilde_A}
  We have \(A=\tilde{A}\)
  if and only if the ideal~\(I_{1,t}\)
  is complemented in the larger ideal~\(\s(\Hilm_t)\)
  for each \(t\in S\).
\end{proposition}

\begin{proof}
  We use the description of the weak conditional expectation
  \(E\colon A\rtimes S\to \bid A\) in Lemma~\ref{lem:description-E}
  and the following computation.  Let \(a_1,a_2\in I_{1,t}\),
  \(\xi_1,\xi_2\in \Hilm_t\).  Then
  \begin{multline*}
    \braket{\Mult(\theta_{1,t})(\hat{\xi}_1)a_1}
    {\Mult(\theta_{1,t})(\hat{\xi}_2)a_2}
    = \braket{\theta_{1,t}(\xi_1 a_1)}{\theta_{1,t}(\xi_2 a_2)}
    \\= \braket{\xi_1 a_1}{\xi_2 a_2}
    = a_1^* \braket{\xi_1}{\xi_2} a_2.
  \end{multline*}
  Hence \(\Mult(\theta_{1,t})(\hat{\xi}_1)^*
  \Mult(\theta_{1,t})(\hat{\xi}_2) = \braket{\xi_1}{\xi_2}\) holds
  in \(\Mult(I_{1,t}) \subseteq \bid{A}\) for all \(\xi_1,\xi_2\in
  \Hilm_t\).

  Assume first that for each \(t\in S\),
  \(I_{1,t}\)
  is a complemented ideal in~\(\s(\Hilm_t)\),
  that is, \(\s(\Hilm_t) = I_{1,t} \oplus I_{1,t}^\bot\)
  for some ideal
  \(I_{1,t}^\bot\idealin \s(\Hilm_t) \idealin A\).
  Since~\(\Hilm_t\)
  is a full right module over~\(\s(\Hilm_t)\),
  we may split \(\Hilm_t \cong \Hilm_t^1 \oplus \Hilm_t^\bot\),
  where \(\Hilm_t^1\)
  and \(\Hilm_t^\bot\)
  are Hilbert modules over \(I_{1,t}\)
  and~\(I_{1,t}^\bot\),
  respectively.  Hence \(\Hilm_t\cdot I_{1,t} = \Hilm_t^1\),
  and this is isomorphic to~\(I_{1,t}\)
  by~\(\theta_{1,t}\).
  The map to \(\Mult(I_{1,t})\)
  annihilates~\(\Hilm_t^\bot\)
  because \(I_{1,t}\cdot I_{1,t}^\bot = 0\).
  Thus the image of~\(\Hilm_t\)
  in~\(\Mult(I_{1,t})\)
  is simply~\(I_{1,t}\),
  which is contained in~\(A\).
  Since this holds for all \(t\in S\)
  by assumption, we get \(E(A\rtimes_\alg S)\subseteq A\)
  and thus \(\tilde{A} = A\) as asserted.

  Conversely, assume that~\(I_{1,t}\) is not complemented
  in~\(\s(\Hilm_t)\) for some \(t\in S\).  Then the image of the map
  \(\s(\Hilm_t)\to \Mult(I_{1,t})\) is not contained in~\(I_{1,t}\):
  otherwise, the kernel of this map would be a complementary ideal
  for~\(I_{1,t}\) in~\(\s(\Hilm_t)\).  Hence there is an element
  \(\xi\in \Hilm_t\) such that \(\braket{\xi}{\xi} \in\s(\Hilm_t)\)
  maps to an element of~\(\Mult(I_{1,t})\) that does not belong
  to~\(I_{1,t}\).  Hence \(E(\xi)^* E(\xi) =
  \Mult(\theta_{1,t})(\hat{\xi})^*\Mult(\theta_{1,t})(\hat{\xi}) =
  \braket{\xi}{\xi}\) does not belong to~\(I_{1,t}\).  Any normal
  representation of~\(\bid{A}\) that vanishes on~\(I_{1,t}\)
  annihilates \(E(\xi)^* E(\xi)\) because it belongs
  to~\(\bar{I}_{1,t}\).  In particular, the normal extension of a
  faithful representation of~\(A/I_{1,t}\) must annihilate
  \(E(\xi)^* E(\xi)\).  If \(E(\xi)^* E(\xi)\in A\), then this
  implies \(E(\xi)^* E(\xi)\in I_{1,t}\), which is false.  Thus
  \(\tilde{A}\neq A\).
\end{proof}

\begin{corollary}[\cite{Exel-Pardo:Tight_groupoid}*{Theorem~3.15}]
  \label{cor:Hausdorffness-Transf-Groupoid}
  The transformation groupoid~\(X\rtimes S\) of an action of a
  unital inverse semigroup~\(S\) on a locally compact Hausdorff
  space~\(X\) by partial homeomorphisms \(\alpha_t\colon D_{t^*}\to
  D_t\) is Hausdorff if and only if, for each \(t\in S\), the
  \textup{(}open\textup{)} set \(D_{1,t}\defeq \bigcup_{e\le 1,t}
  D_e\) is closed in~\(D_t\).
\end{corollary}

\begin{proof}
  As discussed in Section~\ref{sec:commutative_E}, in the present
  situation the ideal~\(I_{1,t}\) corresponds to the open
  set~\(D_{1,t}\) defined in the statement.  This open set is closed
  in~\(D_t\) if and only if the ideal \(I_{1,t}=\Cont_0(D_{1,t})\)
  is complemented in \(\s(\Hils)=\Cont_0(D_t)\).  The result follows
  from Propositions \ref{pro:A_tilde_A}
  and~\ref{pro:tilde_A_commutative}.
\end{proof}

Next we reformulate the condition in Proposition~\ref{pro:A_tilde_A}
using the transformation groupoid \(G \defeq \Prim(A)\rtimes S\).
We may build this as usual for an inverse semigroup action, even
if~\(\Prim(A)\)
is not Hausdorff.  Its object space is~\(\Prim(A)\),
and it is \'etale, that is, the range and source maps are local
homeomorphisms.  Arrows are equivalence classes of pairs~\((t,\prid)\)
for \(\prid\in D_{t^*}\subseteq \Prim(A)\),
where~\(D_{t^*}\)
corresponds to the ideal~\(\s(\Hilm_t)\)
as above.  Two pairs \((t,\prid)\)
and~\((t',\prid')\)
are equivalent if \(\prid=\prid'\)
and there is \(v\in S\)
with \(v\le t, t'\)
and \(\prid\in D_{v^*}\).
There is a unique topology on \((\Prim(A)\rtimes S)^1\)
for which \([t,\prid]\mapsto \prid\)
is a homeomorphism onto~\(D_{t^*}\)
for each \(t\in S\).
The subsets \(U_t \defeq \{[t,\prid]\mid \prid \in D_{t^*}\}\)
form an open covering of \(\Prim(A)\rtimes S\) by bisections.

\begin{theorem}
  \label{the:conditional_expectation}
  The conditional expectation~\(E\) maps \(A\rtimes_\red S\)
  onto~\(A\) if and only if the subset of units \(\Prim(A)\) is
  closed in the arrow space \(\Prim(A)\rtimes S\).
\end{theorem}

\begin{proof}
  Let \(\alpha_t\colon D_{t^*} \to D_t\)
  be the partial homeomorphisms on~\(\Prim(A)\)
  that describe the action of~\(S\).
  Let \(D_{1,t}\defeq \bigcup_{v\le 1,t} D_v\)
  as in Lemma~\ref{lem:I_tu_explicit}.  For \(\prid\in \Prim(A)\)
  and \(t\in S\),
  we have \([t,\prid]=[1,\prid]\) if and only if \(\prid\in D_{1,t^*}\).

  The ideal~\(I_{1,t}\)
  is complemented in~\(\s(\Hilm_t)\)
  if and only if \(D_t = D_{1,t} \sqcup D_{1,t}^\bot\)
  for some open subset~\(D_{1,t}^\bot\)
  in~\(\Prim(A)\),
  namely, the open subset corresponding to the complement
  of~\(I_{1,t}\).
  Of course, \(D_{1,t}^\bot = D_t\setminus D_{1,t}\),
  so such a decomposition exists if and only if
  \(D_t\setminus D_{1,t}\)
  is open; equivalently, \(D_{1,t}\)
  is relatively closed in~\(D_t\).
  Thus the criterion for \(E(A\rtimes_\red S)\subseteq A\)
  in Proposition~\ref{pro:A_tilde_A} is equivalent to~\(D_{1,t}\)
  being relatively closed in~\(D_t\)
  for each \(t\in S\).
  The open subsets \(D_t\subseteq (\Prim(A)\rtimes S)^1\)
  form an open covering, and \(D_{1,t} = D_t \cap \Prim(A)\).
  Hence~\(D_{1,t}\)
  is relatively closed in~\(D_t\)
  for each \(t\in S\)
  if and only if the subset of units~\(\Prim(A)\)
  is closed in~\((\Prim(A)\rtimes S)^1\).
\end{proof}

The theorem above is related to
\cite{Brown-Nagy-Reznikoff-Sims-Williams:Cartan_etale}*{Corollary
  4.4}.

The existence of a conditional
expectation \(A\rtimes_\red S\to A\) should be viewed as
an analogue for inverse semigroup crossed products of Hausdorffness
for groupoid crossed products.  By
Lemma~\ref{lem:G_Hausdorff_iff_G0_Closed}, a groupoid with Hausdorff
object space has Hausdorff arrow space if and only if the set of units
is closed.  Theorem~\ref{the:conditional_expectation} involves the
same condition for the groupoid \(\Prim(A)\rtimes S\),
which may have a non-Hausdorff object space.  Thus the condition of
``having a closed set of units'' captures the good features of
Hausdorff groupoids in the context of inverse semigroup actions.
The next result shows that this property behaves well with respect to equivariant maps.

\begin{lemma}
  \label{lem:closed_units_inherited}
  Let \(X\)
  and~\(Y\)
  be topological spaces with \(S\)\nb-actions
  by partial homeomorphisms and let \(f\colon X\to Y\)
  be an \(S\)\nb-equivariant
  continuous map.  If the units are closed in~\(Y\rtimes S\),
  then the same happens in~\(X\rtimes S\).
\end{lemma}

\begin{proof}
  The map~\(f\)
  induces a continuous functor \(f_*\colon X\rtimes S\to Y\rtimes S\),
  mapping the germ of \(t\in S\)
  at \(x\in X\)
  to the germ of~\(t\)
  at~\(f(x)\).
  Each \(t\in S\)
  gives bisections \(D^X_t\subseteq X\rtimes S\)
  and \(D^Y_t\subseteq Y\rtimes S\).
  By construction, \(D^X_t = f^{-1}(D^Y_t)\).
  These bisections give an open covering of the respective arrow
  spaces.  So the unit bisection~\(D_1\)
  is closed if and only if \(D_1\cap D_t\)
  is relatively closed in~\(D_t\)
  for each \(t\in S\).
  If this holds in~\(Y\),
  then \(D^X_t \cap D^X_1 = f^{-1}(D^Y_t \cap D^Y_1)\)
  is also relatively closed in \(f^{-1}(D^Y_t) = D^X_t\).
\end{proof}

\begin{example}
  \label{exa:Hausdorff_groupoid_action_closed_units}
  Let~\(G\)
  be a Hausdorff groupoid and let~\(S\)
  be a wide inverse semigroup of bisections of~\(G\),
  so that \(G\cong G^0\rtimes S\).
  Let~\(A\)
  be a \(\Cst\)\nb-algebra
  with an action of~\(G\)
  by \(\Cst\)\nb-correspondences;
  that is, \(A\)
  is the space of \(\Cont_0\)\nb-sections
  on~\(G^0\)
  of a Fell bundle over~\(G\).
  Turn this Fell bundle over~\(G\)
  into an action of~\(S\)
  with an
  \(S\)\nb-equivariant
  continuous map \(\Prim(A)\to G^0\)
  as in \cite{Buss-Meyer:Actions_groupoids}*{Theorem 6.13}.
  Lemma~\ref{lem:G_Hausdorff_iff_G0_Closed} shows that the units in
  \(G= G^0\rtimes S\)
  are closed.  Hence the units in \(\Prim(A)\rtimes S\)
  are closed by Lemma~\ref{lem:closed_units_inherited}.  Thus our
  conditional expectation~\(E\)
  maps \(A\rtimes G\cong A\rtimes S\)
  to~\(A\).
  We may also construct the conditional
  expectation \(A\rtimes G\to A\) directly.
\end{example}

\begin{example}
  \label{exa:complemented_action}
  Paterson~\cite{Paterson:Groupoids} associates a certain locally
  compact, totally disconnected, possibly non-Hausdorff
  groupoid~\(G_P(S)\) to any inverse semigroup~\(S\).  This is the
  transformation groupoid for the canonical action of~\(S\) on the
  spectrum of the semilattice \(E=E(S)\) endowed with the totally
  disconnected Hausdorff topology from the product space
  \(\{0,1\}^E\).  We denote this spectrum by~\(\widehat{E}_P\), so
  Paterson's groupoid is \(G_P(S) =\widehat{E}_P\rtimes S\).
  This groupoid has the universal property that
  there is a natural bijection between actions of~\(G_P(S)\) on a
  topological space~\(X\) and actions of~\(S\) on~\(X\) by partial
  homeomorphisms with clopen domains and codomains (compare
  \cite{Steinberg:Groupoid_approach}*{Proposition~5.5} and
  \cite{Paterson:Groupoids}*{Section~4.3}).  The universal property
  of~\(G_P(S)\) follows from that of the universal groupoid \(G(S)
  \defeq \widehat{E}\rtimes S\) constructed
  in~\cite{Buss-Exel-Meyer:InverseSemigroupActions}.
  Here~\(\widehat{E}\) is also the spectrum of~\(E(S)\), that is, it
  is equal to~\(\widehat{E}_P\) as a set, but it carries a different,
  non-Hausdorff, topology, which will be explained below in the proof
  of Proposition~\ref{pro:E-unitary}.
  The space~\(\widehat{E}\) has the universal property that continuous
  maps from a topological space~\(X\) to~\(\widehat{E}\) correspond
  bijectively to semilattice maps (preserving zero and unit)
  from~\(E\) to the lattice of open subsets of~\(X\).  The map \(X\to
  \widehat{E}\) is continuous as a map to~\(\widehat{E}_P\) if and
  only if~\(E\) maps into the sublattice of clopen subsets of~\(X\).

  Call a Hilbert bimodule~\(\Hilm\) over a \(\Cst\)\nb-algebra~\(A\)
  \emph{complemented} if the ideals \(\s(\Hilm)\) and~\(\rg(\Hilm)\)
  are complemented ideals, that is, \(A\cong \s(\Hilm)\oplus I_1\)
  and \(A\cong \rg(\Hilm)\oplus I_2\) for certain ideals \(I_1,
  I_2\idealin A\), which are automatically unique.
  Equivalently, the domain and codomain of the partial homeomorphism
  of~\(\Prim(A)\) associated to~\(\Hilm\) are clopen.

  Let~\(S\) be a unital inverse semigroup, \(A\) a
  \(\Cst\)\nb-algebra, and let \((\Hilm_t)_{t\in S}\) be an action
  by complemented Hilbert bimodules.  Then the induced action
  of~\(S\) on \(\Prim(A)\) has clopen domains and codomains by
  assumption.  Thus the complemented actions of~\(S\) are in
  bijection with Fell bundles over Paterson's groupoid~\(G_P(S)\).

  As a consequence, Paterson's groupoid~\(G_P(S)\) is Hausdorff if
  and only if the conditional expectation on~\(A\rtimes S\) takes
  values in~\(A\) for any \emph{complemented} action, see
  Example~\ref{exa:Hausdorff_groupoid_action_closed_units}.
  Steinberg~\cite{Steinberg:Groupoid_approach} and Paterson
  (see~\cite{Paterson:Groupoids}*{Corollary~4.3.1}) characterise
  when
  Paterson's groupoid is Hausdorff: this happens if and only if for
  all \(t,u\in S\), the set \(\{v\in S \mid v\le t,u\}\) is finitely
  generated as an ordered set, that is, there is a finite set
  \(F\subseteq S\) such that
  \[
  \{v\in S \mid v\le t,u\} = \{v\in S \mid v\le w
  \text{ for some }w\in F\}.
  \]
  Since \(G_P(S)\) is a transformation groupoid \(X\rtimes S\),
  Corollary~\ref{cor:Hausdorffness-Transf-Groupoid} characterises
  when~\(G_P(S)\) is Hausdorff; in this form, this appears in
  \cite{Paterson:Groupoids}*{Proposition~4.3.6}.
\end{example}

\begin{example}
  \label{exa:tight_groupoid_Hausdorff}
  The tight groupoid \(G_\textup{tight}(S)\) of an inverse
  semigroup~\(S\) is the restriction of Paterson's groupoid to a
  certain closed, invariant subset
  (see~\cite{Exel:Inverse_combinatorial}).  Call an inverse
  semigroup action on a \(\Cst\)\nb-algebra \emph{tight} if it comes
  from an action of the tight groupoid (see
  \cite{Buss-Meyer:Actions_groupoids}*{Theorem~6.13}).  Our results
  show that the tight groupoid of~\(S\) is Hausdorff if and only if
  the weak conditional expectation on~\(A\rtimes S\) takes values
  in~\(A\) for each tight action of~\(S\) on a \(\Cst\)\nb-algebra.
  Exel and Pardo
  characterise when the tight groupoid is Hausdorff
  in~\cite{Exel-Pardo:Tight_groupoid}*{Theorem~3.16}.
\end{example}

Paterson's groupoid and the tight groupoid of an inverse semigroup
can only account for complemented \(S\)\nb-actions on
\(\Cst\)\nb-algebras because they have Hausdorff object space.  The
universal \(S\)\nb-action constructed
in~\cite{Buss-Exel-Meyer:InverseSemigroupActions}, which takes place
on a certain non-Hausdorff space~\(\hat{E}\), allows us to get rid
of the assumption on complements; we recall its definition during
the proof of the following proposition.

\begin{proposition}
  \label{pro:E-unitary}
  Let~\(S\) be an inverse semigroup with zero and unit.  The
  following are equivalent:
  \begin{enumerate}
  \item \label{enum:E-unitary1} \(S\) is \(E^*\)\nb-unitary: if
    \(e,t\in S\) satisfy \(e^2=e\) and \(e\le t\), then
    \(e=0\) or \(t^2=t\);
  \item \label{enum:E-unitary2} if \(e,t\in S\) satisfy \(e \le
    1,t\), then \(e=0\) or \(t\le 1\);
  \item \label{enum:E-unitary3} the space of units is closed in the
    transformation groupoid \(\hat{E}\rtimes S\) for the universal
    \(S\)\nb-action;
  \item \label{enum:E-unitary4} the space of units is closed in the
    transformation groupoid for any zero-preserving \(S\)\nb-action
    by partial homeomorphisms;
  \item \label{enum:E-unitary5} the weak conditional
    expectation~\(E\) takes values in~\(A\) for any zero-preserving
    action of~\(S\) on a \(\Cst\)\nb-algebra by Hilbert bimodules.
  \end{enumerate}
  In particular, if~\(S\) is \(E^*\)\nb-unitary, then the
  transformation groupoid \(X\rtimes S\) is Hausdorff for any
  zero-preserving action of~\(S\) on a Hausdorff space~\(X\).
\end{proposition}

\begin{proof}
  \ref{enum:E-unitary1}\(\iff\)\ref{enum:E-unitary2} holds because
  an element \(e\in S\) of a unital inverse semigroup satisfies
  \(e^2 = e\) if and only \(e\le 1\).

  Our next goal is to prove
  \ref{enum:E-unitary2}\(\iff\)\ref{enum:E-unitary3}.  First we
  recall the definition of~\(\hat{E}\) and the \(S\)\nb-action on
  it, see~\cite{Buss-Exel-Meyer:InverseSemigroupActions}.  Elements
  of~\(\hat{E}\) are the characters of \(E= \{e\in S\mid e^2 =
  e\}\), that is, functions \(\varphi\colon E\to \{0,1\}\) with
  \(\varphi(e f) = \varphi(e) \varphi(f)\), \(\varphi(1)=1\) and
  \(\varphi(0)=0\).  The topology is generated by the open subsets
  \[
  U_e \defeq \{ \varphi\in \hat{E} \mid \varphi(e)=1\}
  \]
  for \(e\in E\).  If \(\varphi(e)=\varphi(f)=1\) for \(e,f\in E\),
  then \(\varphi(e f)=1\), and if \(\varphi(e)=0\) or
  \(\varphi(f)=0\), then \(\varphi(e f)=0\).  Thus \(U_e \cap U_f =
  U_{e f}\).  Hence the subsets~\(U_e\) even form a basis of the
  topology, and any open subset is the union of the subsets of the
  form~\(U_e\) that it contains.  The map sending an open
  subset~\(V\) of~\(\hat{E}\) to the set of all \(e\in E\) with
  \(U_e \subseteq V\) is an isomorphism from the lattice of open
  subsets in~\(\hat{E}\) to the lattice of \emph{ideals} in~\(E\) by
  \cite{Buss-Exel-Meyer:InverseSemigroupActions}*{Lemma 2.14}; an
  ideal in~\(E\) is a subset~\(I\) with \(0\in I\) and such that
  \(e\le f\) and \(f\in I\) implies \(e\in I\).

  The element \(t\in S\) acts on~\(\hat{E}\) by the homeomorphism
  \[
  c_t\colon U_{t^* t} \to U_{t t^*},\qquad
  c_t(\varphi)(e) = \varphi(t^* e t);
  \]
  this defines a zero-preserving action of~\(S\) on~\(\hat{E}\) by
  partial homeomorphisms, and it is the universal such action on a
  topological space by
  \cite{Buss-Exel-Meyer:InverseSemigroupActions}*{Theorem 2.22}.

  The arrows in \(\hat{E}\rtimes S\) are equivalence classes of
  pairs \((t,\varphi)\) with \(t\in S\), \(\varphi\in U_{t^* t}\),
  where \((t,\varphi)\sim (t',\varphi')\) if \(\varphi=\varphi'\)
  and there is \(e\in E\) with \(\varphi(e)=\varphi'(e)=1\) and \(t
  e= t' e\).  The topology is such that the projection map
  \([t,\varphi]\mapsto \varphi\) is a local homeomorphism.  The
  subsets \([t] \defeq \{[t,\varphi] \mid \varphi\in U_{t^* t}\}\)
  form an open cover~\(\hat{E}\rtimes S\), and \([t]\) is
  homeomorphic to~\(U_{t^* t}\).

  By definition, \([1]\) is the set of units, and \([t]\cap [1]\) is
  the subset of all \([t,\varphi]=[1,\varphi]\) with \(\varphi\in
  \bigcup_{e\le 1,t} U_e\).  Let \(L_t \defeq \{e\in E\mid e\le
  t\}\) and let
  \[
  L_t^\bot \defeq \{f\in E\mid f \le t^* t \text{ and } ef=0
  \text{ for all }e\in L_t\}.
  \]
  An open subset~\(U_f\) is contained in \([t]\setminus [1]\) if and
  only if \(f \le t^* t\) and \(U_e \cap U_f=\emptyset\) for all
  \(e\in L_t\).  Since
  \(U_e \cap U_f = U_{e f}\), which is only empty if \(e f=0\), the
  open subset~\(U_f\) is contained in \([t]\setminus [1]\) if and
  only if \(f\in L_t^\bot\).  The subset \([1]\subseteq
  \hat{E}\rtimes S\) is closed if and only if \([t]\cap [1]\) is
  relatively closed in~\([t]\) for each \(t\in S\), if and only if
  \([t]\setminus [1]\) is open for each \(t\in S\).  Since the
  subsets~\(U_f\) form a basis, this happens if and only if
  \([t]\setminus [1]\) is the union of the subsets~\(U_f\) it
  contains.  Thus the units are closed in \(\hat{E}\rtimes S\) if
  and only if \(\bigcup_{e\in L_t \cup L_t^\bot} U_e = U_{t t^*}\).
  By \cite{Buss-Exel-Meyer:InverseSemigroupActions}*{Lemma 2.14},
  this only happens if \(t t^* \in L_t \cup L_t^\bot\).  We
  have \(t t^* \in L_t\) if and only if \(t\le 1\), and \(t t^* \in
  L_t^\bot\) if and only if \(e\le 1,t\) only for \(e=0\).  Thus
  \ref{enum:E-unitary2}\(\iff\)\ref{enum:E-unitary3}.

  \ref{enum:E-unitary3}\(\Rightarrow\)\ref{enum:E-unitary4} follows
  from Lemma~\ref{lem:closed_units_inherited} and the universal
  property of the action of~\(S\) on~\(\hat{E}\), see
  \cite{Buss-Exel-Meyer:InverseSemigroupActions}*{Theorem 2.22}.
  \ref{enum:E-unitary4}\(\Rightarrow\)\ref{enum:E-unitary5} follows
  from Theorem~\ref{the:conditional_expectation}.
  Example~\ref{exa:graph_example} shows an action of~\(S\) on
  a \(\Cst\)\nb-algebra~\(A\) such that the induced action
  on~\(\Prim(A)\) is the universal action on~\(\hat{E}\).  Then
  Theorem~\ref{the:conditional_expectation} shows
  \ref{enum:E-unitary5}\(\Rightarrow\)\ref{enum:E-unitary3}, which
  finishes the proof of the proposition.
\end{proof}

\begin{example}
  \label{exa:graph_example}
  We construct an action of~\(S\) by Hilbert bimodules on a graph
  \(\Cst\)\nb-algebra, using their well-understood ideal structure,
  see~\cite{Bates-Hong-Raeburn-Szymanski:Ideal_structure}
  or~\cite{Raeburn:Graph_algebras}.  Our graph has vertex set \(E^*
  \defeq E\setminus\{0\}\).  If \(e,f\in E^*\) satisfy \(e\ge f\),
  then we put countably many edges \(e\to f\); otherwise there is no
  edge \(e\to f\).  Let~\(A(E)\) be the resulting graph
  \(\Cst\)\nb-algebra.  Since any vertex receives infinitely many
  edges \(e\to e\), any subset of the vertex set is ``saturated.''
  Thus the lattice of ideals in~\(A(E)\) is isomorphic to the
  lattice of ``hereditary'' subsets in~\(E^*\) by the main result
  of~\cite{Bates-Hong-Raeburn-Szymanski:Ideal_structure}, see
  also~\cite{Eilers-Ruiz-Sorensen:Amplified}.  A subset~\(U\)
  of~\(E^*\) is hereditary if \(e\ge f\) and \(e\in U\) implies
  \(f\in U\).  This means that \(\{0\} \cup U\) is an ideal
  in~\(E\), and these ideals in~\(E\) correspond to open subsets
  of~\(\hat{E}\) by
  \cite{Buss-Exel-Meyer:InverseSemigroupActions}*{Lemma 2.14}.  Thus
  \(\Prim(A(E))\) and~\(\hat{E}\) have isomorphic lattices of open
  subsets.  This implies that they are homeomorphic because both are
  sober spaces.  (Any ideal in~\(E\) contains~\(0\) by convention;
  this is why we left out \(0\in E\) to construct \(A(E)\).)

  We must still lift the action of~\(S\) on~\(\hat{E}\) to an action
  on~\(A(E)\).  If \(I\subseteq E\) is an ideal, then the
  corresponding ideal in \(A(E)\) is Morita--Rieffel equivalent to
  the graph \(\Cst\)\nb-algebra of the restriction of the graph
  above to the vertex set \(I^* = I\setminus\{0\}\subseteq E^*\).
  If \(t\in S\), then \(e\mapsto t e t^*\) maps the subsemilattice
  \(E_{\le t^* t}\subseteq E\) isomorphically onto~\(E_{\le t t^*}\)
  with inverse \(e\mapsto t^* e t\).  This is a semilattice isomorphism,
  that is, it preserves the order relation~\(\le\) and the zero
  elements.  Thus it induces a graph isomorphism between the
  restrictions of our graphs to \(E_{\le t^* t}\) and~\(E_{\le t
    t^*}\) and thus an isomorphism between the associated graph
  \(\Cst\)\nb-algebras.  These are canonically Morita--Rieffel
  equivalent to the ideals in the graph \(\Cst\)\nb-algebra~\(A(E)\)
  corresponding to the ideals \(E_{\le t t^*}\) and~\(E_{\le t^*
    t}\).  Hence~\(t\) induces a canonical Morita--Rieffel
  equivalence between these two ideals.  This gives an action
  of~\(S\) by Hilbert bimodules on the graph
  \(\Cst\)\nb-algebra~\(A(E)\) that induces the desired action on
  the open subsets of~\(\Prim(A)\) and hence on \(\Prim(A) \cong
  \hat{E}\).
\end{example}

An inverse semigroup~\(S\) is called \emph{\(E\)\nb-unitary} if for
all \(e,t\in S\), the condition \(e^2=e\le t\) implies \(t=t^2\).
The inverse semigroup~\(S\) is
\(E\)\nb-unitary if and only if \(S_0\defeq S\sqcup \{0\}\) (\(S\)~with
a formal zero added) is \(E^*\)\nb-unitary.  Actions of~\(S\) on a
\cstar{}algebra correspond bijectively to zero-preserving actions
of~\(S_0\), and this correspondence preserves crossed products and
weak conditional expectations.  So Proposition~\ref{pro:E-unitary}
gives the following for \(E\)\nb-unitary inverse semigroups:

\begin{corollary}
  Let~\(S\) be an inverse semigroup with unit, but possibly without
  zero.  The weak conditional expectation has values in~\(A\) for all
  actions of~\(S\) on \cstar{}algebras~\(A\) if and only if~\(S\) is
  \(E\)\nb-unitary, if and only if the transformation groupoid
  \(X\rtimes S\) has closed units for any action of~\(S\) on a
  topological space~\(X\).

  Therefore, if~\(S\) is \(E\)\nb-unitary, then~\(X\rtimes S\) is
  Hausdorff for any action of~\(S\) on a Hausdorff space~\(X\).
\end{corollary}

\section{Faithful representations of the reduced crossed product}
\label{sec:faithful_rep}

Any representation \(\pi\colon A\to\Bound(\Hils)\)
extends uniquely to a weakly continuous representation of~\(\bid{A}\),
which we may then restrict to a representation~\(\tilde{\pi}\)
of~\(\tilde{A}\).
When is~\(\tilde{\pi}\)
faithful?  In view of Example~\ref{exa:action_compact_plus_unitary},
we only aim for a sufficient condition.  Our starting point is
Lemma~\ref{lem:description-E}:
the image of the subspace \(\Hilm_t\subseteq A\rtimes S\) under the
weak conditional expectation \(E\colon A\rtimes S\to \bid A\)
is contained in the multiplier algebra~\(\Mult(I_{1,t})\);
here we embed \(\Mult(I_{1,t}) \subseteq \bid{I_{1,t}} \subseteq \bid{A}\)
as before.  Hence~\(\tilde{A}\)
is contained in the \(\Cst\)\nb-subalgebra
of~\(\bid{A}\)
generated by \(\Mult(I_{1,t})\subseteq \bid{A}\)
for all \(t\in S\).
Taking even more generators, we let~\(\mathcal{I}\)
be the lattice of ideals generated by~\(I_{1,t}\)
for \(t\in S\), that is, we add finite intersections and unions of
ideals

Recall that \([I]\in \bid A\) for an ideal \(I \idealin A\)
denotes the support projection of~\(I\).

  \begin{lemma}
    \label{lem:central_projection_algebra}
    Let \(I,J\) be ideals of~\(A\).  Then \([I]\cdot [J] = [I\cap
    J]\) and \([I]\vee [J] = [I+ J]\).  That is, \(I\mapsto [I]\) is
    a lattice map.  In particular, \([I]+[J] = [I+J] + [I\cap J]\).
  \end{lemma}

  \begin{proof}
    The supremum \([I]\vee [J]\)
    and \([I+J]\)
    act in any representation~\(\pi\colon A\to \Bound(\Hils)\)
    by the orthogonal projection onto
    \(\pi(I)\Hils+\pi(J)\Hils = \pi(I+J)\Hils\)
    (see also Lemma~\ref{lem:I_tu_explicit}), hence they are equal
    in~\(\bid{A}\).
    The assertion \([I]\cdot [J] = [I\cap J]\)
    is equivalent to
    \(\pi(I)\Hils\cap \pi(J)\Hils = \pi(I\cap J)\Hils\)
    for any representation~\(\pi\)
    of~\(A\).
    The inclusion~\(\supseteq\)
    is obvious, and~\(\subseteq\)
    follows because both \(\pi(I)\)
    and~\(\pi(J)\)
    act nondegenerately on \(\pi(I)\Hils\cap \pi(J)\Hils\),
    giving
    \(\pi(I)\Hils\cap \pi(J)\Hils \subseteq
    \pi(I)\pi(J)\Hils = \pi(I\cap J)\Hils\).
  \end{proof}

\begin{proposition}
  \label{pro:faithful_rep_abstract}
  Let~\(\mathcal{I}\)
  be a lattice of ideals in a \(\Cst\)\nb-algebra~\(A\)
  and let \(A_\mathcal{I}\subseteq \bid{A}\)
  be the \(\Cst\)\nb-subalgebra
  generated by~\(\Mult(I)\)
  for all \(I\in\mathcal{I}\).
  Let \(\pi\colon A\to\Bound(\Hils)\)
  be a representation.  The restriction of~\(\bid\pi\)
  to~\(A_\mathcal{I}\)
  is faithful if \(a\in I\) whenever \(I,J\in\mathcal{I}\) and
  \(a\in J\) satisfy \(I\subseteq J\)
  and \(\pi(a) \pi(J)\Hils\subseteq \pi(I)\Hils\).
\end{proposition}

\begin{proof}
  The \(\Cst\)\nb-algebra~\(A_\mathcal{I}\)
  is the inductive limit of the subalgebras~\(A_\mathcal{F}\)
  for finite sublattices \(\mathcal{F}\subseteq \mathcal{I}\).
  Hence a representation on~\(A_\mathcal{I}\)
  is faithful if and only if it is faithful on~\(A_\mathcal{F}\)
  for each finite sublattice~\(\mathcal{F}\).
  Hence we may assume without loss of generality that the
  lattice~\(\mathcal{I}\)
  is finite.  We do this from now on.

  For \(J\in\mathcal{I}\),
  let \(J^\circ = \sum_{I\in\mathcal{I}, I<J} I\).
  Call~\(J\)
  \emph{irreducible} if \(J\neq J^\circ\), that is, \(J\) is not
  a sum of strictly smaller ideals in~\(\mathcal{I}\).
  We claim that the summands~\(\Mult(J)\)
  for irreducible~\(J\)
  already generate~\(A_\mathcal{I}\).
  If \(I\le J\)
  in~\(\mathcal{I}\),
  then \(I\idealin \Mult(J)\),
  which gives a unital \Star{}homomorphism
  \(\rho_{IJ}\colon \Mult(J)\to \Mult(I)\)
  such that \(\rho_{IJ}(x) = x\cdot [I]\)
  in~\(\bid{A}\)
  for all \(x\in\Mult(J)\).
  If \(J=I_1+I_2\),
  then we rewrite \(x\in\Mult(J)\)
  using Lemma~\ref{lem:central_projection_algebra}:
  \[
  x = x[I_1] + x[I_2] - x[I_1\cap I_2]
  = \rho_{I_1 J}(x) + \rho_{I_2 J}(x)
  - \rho_{(I_1\cap I_2) J}(x).
  \]
  Since the right hand side lies in
  \(\Mult(I_1) + \Mult(I_2) + \Mult(I_1\cap I_2)\),
  the generators \(x\in \Mult(J)\)
  are redundant if \(J=I_1+I_2\)
  for \(I_1,I_2\in\mathcal{I}\)
  with \(I_1,I_2\neq J\).
  Since~\(\mathcal{I}\)
  is finite, any \(J\in\mathcal{I}\)
  is a finite sum of irreducible ideals in~\(\mathcal{I}\).
  Repeating the above decomposition, we see that~\(A_\mathcal{I}\)
  is generated by~\(\Mult(J)\) for irreducible \(J\in\mathcal{I}\).

  If~\(J\)
  is irreducible, then \(J^\circ< J\)
  is the maximal element of~\(\mathcal{I}\)
  below~\(J\),
  and we may decompose any \(x\in\Mult(J)\)
  as \(\rho_{J^\circ J}(x)[J^\circ] + x\cdot ([J]-[J^\circ])\).
  The first term in \(\Mult(J^\circ)\)
  may be further decomposed using irreducible elements
  of~\(\mathcal{I}\)
  contained in~\(J^\circ\).
  Thus we may replace the generators \(x\in\Mult(J)\)
  by \(x\cdot ([J]-[J^\circ])\) for irreducible \(J\in\mathcal{I}\).

  If \(I\cap J\neq I\),
  then \(I\cap J\le I^\circ\),
  so that \(([I]-[I^\circ])[J]=0\)
  and hence also \(([I]-[I^\circ])([J]-[J^\circ])=0\).
  By symmetry, the same happens if \(I\cap J\neq J\).
  Hence \(([I]-[I^\circ])([J]-[J^\circ])=0\)
  whenever \(I \neq J\).

  Thus~\(A_\mathcal{I}\)
  is the \emph{orthogonal} direct sum of \(\Mult(J)([J]-[J^\circ])\)
  for all irreducible \(J\in\mathcal{I}\).
  The representation~\(\bid{\pi}\)
  is faithful on~\(A_\mathcal{I}\)
  if and only if it is faithful on each summand
  \(\Mult(J)([J]-[J^\circ])\).
  Let \(x\in\Mult(J)\)
  satisfy \(\bid{\pi}(x[J]-x[J^\circ])=0\),
  that is, \(\bid{\pi}(x[J]) = \bid{\pi}(x[J^\circ])\).
  Then \(\bid{\pi}(x)\pi(J)\Hils \subseteq \pi(J^\circ)\Hils\)
  and hence \(\pi(xa)\pi(J)\Hils \subseteq \pi(J^\circ)\Hils\)
  for each \(a\in J\).
  Since \(J,J^\circ\in\mathcal{I}\),
  the assumption in our proposition gives \(xa\in J^\circ\)
  for all \(a\in J\).
  Then \(\rho(xa)\rho(J)\Hils_\rho \subseteq \rho(J^\circ)\Hils_\rho\)
  for each representation~\(\rho\)  of~\(A\) and each \(a\in J\),
  giving \(\bid{\rho}(x)\rho(J)\Hils_\rho \subseteq \rho(J^\circ)\Hils_\rho\)
  for each representation~\(\rho\)
  and hence \(x[J] = x[J^\circ]\).
  Thus~\(\bid{\pi}\) is faithful on the summands
  \(\Mult(J)([J]-[J^\circ])\) and hence on~\(A_{\mathcal{I}}\).
\end{proof}

\begin{theorem}
  \label{the:E-faithful}
  Let~\(A\)
  be a \(\Cst\)\nb-algebra
  and let~\(S\)
  be a unital inverse semigroup acting on~\(A\)
  by Hilbert bimodules~\((\Hilm_t)_{t\in S}\).  Let~\(\mathcal{I}\)
  be a lattice of ideals in~\(A\)
  that contains the ideals~\(I_{1,t}\)
  for all \(t\in S\).
  Let \(\pi\colon A\to\Bound(\Hils)\)
  be a representation of~\(A\)
  such that, for all \(I,J\in\mathcal{I}\)
  and \(a\in J\)
  with \(I\subseteq J\)
  and \(\pi(a) \pi(J)\Hils\subseteq \pi(I)\Hils\),
  already \(a\in I\).
  Then~\(\pi\)
  is \(E\)\nb-faithful,
  so \(\Ind\pi\) is a faithful representation of~\(A\rtimes_\red S\).
\end{theorem}

\begin{proof}
  Lemma~\ref{lem:description-E} shows that \(E(\Hilm_t)\) is contained
  in~\(\Mult(I_{1,t})\) for all \(t\in S\).  Thus~\(\tilde A\) is
  contained in~\(A_{\mathcal{I}}\).
  Proposition~\ref{pro:faithful_rep_abstract} shows that~\(\pi\) is
  \(E\)\nb-faithful.  Hence \(\Ind\pi\) is faithful by
  Proposition~\ref{pro:E-faithful}.
\end{proof}

\begin{theorem}
  \label{the:pure_not_matter}
  The family of all irreducible representations of~\(A\)
  is \(E\)\nb-faithful.
\end{theorem}

\begin{proof}
  We apply Theorem~\ref{the:E-faithful} to the direct sum~\(\Pi\)
  of all irreducible representations and the lattice of all ideals
  of~\(A\).  Let
  \(I\idealin J\idealin A\)
  be ideals and \(a\in J\).
  If \(a\notin I\),
  then there is an irreducible representation
  \(\pi\colon J/I\to\Bound(\Hils_\pi)\)
  of~\(J/I\)
  with \(\pi(a)\neq0\).
  This irreducible representation extends to an irreducible
  representation of~\(A\),
  again denoted~\(\pi\).
  Since \(\pi|_J\)
  is irreducible and \(\pi|_I=0\),
  we get \(\pi(J)\Hils_\pi=\Hils_\pi\)
  and \(\pi(I)\Hils_\pi=0\),
  so \(\pi(a)\pi(J)\Hils_\pi\)
  is not contained in~\(\pi(I)\Hils_\pi\)
  because \(\pi(a)\neq0\).
  Since~\(\pi\)
  is a direct summand in~\(\Pi\),
  the same happens for~\(\Pi\).
  Hence~\(\Pi\)
  verifies the assumption in Theorem~\ref{the:E-faithful}.
\end{proof}

Theorem~\ref{the:pure_not_matter} shows that the irreducible
representations of~\(A\) already give a faithful representation
of~\(\tilde{A}\) and hence of~\(A\rtimes_\red S\) by induction.
In~\cite{Exel:noncomm.cartan} Exel defines~\(A\rtimes_\red S\) by
inducing only irreducible representations of~\(A\).  More precisely,
Exel constructs a positive linear functional~\(\tilde\varphi\)
on~\(A\rtimes S\) from every pure state~\(\varphi\) of~\(A\) and
defines \(A\rtimes_\red S\) as the image of~\(A\rtimes S\) in the
direct sum of the GNS-representations of~\(\tilde\varphi\) for all
pure states~\(\varphi\).  The induced functional
\(\tilde\varphi\colon A\rtimes S\to\C\) is computed
in~\cite{Exel:noncomm.cartan}*{Proposition~7.4} as follows.  Let
\(t\in S\) and
\(\xi\in\Hils_t\).  First assume that~\(\varphi\) is supported
on~\(\s(\Hils_e)\) for some \(e\le 1,t\) (a linear functional
\(\varphi\in A'\) is \emph{supported on an ideal} \(I\idealin
A\) if \(\varphi(a) = \lim_i \varphi(a u_i)\) for all \(a\in A\) and
any approximate unit~\((u_i)\) of~\(I\),
see~\cite{Exel:noncomm.cartan}).  This is equivalent to the
existence of a linear functional \(\psi\in I'\) and \(b\in I\) with
\(\varphi(a)=\psi(ab)\) for all \(a\in A\) by
\cite{Exel:noncomm.cartan}*{Proposition~5.1}.  If~\(\varphi\) is
supported on~\(\s(\Hils_e)\) for some \(e\le 1,t\), then
\(\tilde\varphi(\xi\delta_t) \defeq \lim_i \varphi(\theta_{1,t}^e(\xi
u_{e,i}))\), where~\((u_{e,i})_{i\in I}\) is an approximate unit of
\(\s(\Hils_e)=I_{1,e}\) and~\(\theta_{1,t}^e\) is the
isomorphism \(\Hils_t \s(\Hils_e) \congto
\Hils_1\s(\Hils_e)=I_{1,e}\) in~\eqref{eq:Def-thetas}.
If~\(\varphi\) is not supported on~\(\s(\Hils_e)\) for
any \(e\le 1,t\), then \(\tilde\varphi(\xi\delta_t) \defeq 0\).  The
above definition of~\(\tilde\varphi\) is only reasonable
if~\(\varphi\) is pure, and~\cite{Exel:noncomm.cartan} only
considers pure states in the definition of the reduced crossed
product.

The following proposition shows that this induced
state~\(\tilde\varphi\) coincides with \(\bid\varphi\circ E\), where
\(\bid\varphi\colon \bid A\to \C\) denotes the unique normal
extension of~\(\varphi\).  Hence the GNS-representation
of~\(\tilde\varphi\) is the induced representation of the
GNS-representation of~\(\varphi\) for any pure state~\(\varphi\),
and so our reduced crossed product is the same one as
in~\cite{Exel:noncomm.cartan} by Theorem~\ref{the:pure_not_matter}.

\begin{proposition}
  Let \((\Hils_t)_{t\in S}\) be an action of~\(S\) on a
  \(\Cst\)\nb-algebra \(A\) by Hilbert bimodules.  Let \(\varphi\in
  A'\) be a bounded linear functional on~\(A\) and let \(t\in S\),
  \(\xi\in\Hilm_t\).  Let~\((u_{t,i})\) be an approximate unit for
  the ideal~\(I_{1,t}\).  Then
  \begin{equation}
    \label{eq:formula-Ind-functional}
    \bid{\varphi}\circ E(\xi\delta_t)
    = \lim_i\varphi(\theta_{1,t}(\xi u_{t,i})),
  \end{equation}
  If~\(\varphi\) is supported on \(\s(\Hils_e)=I_{1,e}\) for some
  \(e\le 1,t\) and~\((u_{e,i})\) is an approximate unit
  for~\(I_{1,e}\), then \(\bid{\varphi}\circ E(\xi\delta_t) = \lim_i
  \varphi(\theta_{1,t}^e(\xi u_{e,i}))\).  If~\(\varphi\) is pure, then
  \(\bid\varphi\circ E=\tilde\varphi\).
\end{proposition}

\begin{proof}
  Formula~\eqref{eq:formula-Ind-functional} follows
  from~\eqref{eq:formula-cond.exp} because bounded linear
  functionals on \(\Cst\)\nb-algebras are strictly continuous.
  Assume that~\(\varphi\) is supported on one of the ideals
  \(I_{1,e}=\s(\Hils_e)\) for \(e\le 1,t\) that generate~\(I_{1,t}\).
  So there is \(\psi\in I_{1,e}'\) and \(b\in I_{1,e}\) with
  \(\phi(a)=\psi(ab)\) for all \(a\in A\).
  Then~\eqref{eq:formula-Ind-functional} implies
  \[
  \bid\varphi\circ E(\xi \delta_t)
  = \psi(\theta_{1,t}(\xi b))
  = \psi(\theta_{1,t}^e(\xi b))
  = \lim_i\varphi(\theta_{1,t}^e(\xi u_{e,i})).
  \]
  If~\(\varphi\) is pure, then either it is supported on an ideal or
  it vanishes on this ideal (see
  \cite{Exel:noncomm.cartan}*{Proposition~5.5}).  Therefore,
  if~\(\varphi\) is not supported on any of the ideals~\(I_{1,e}\)
  with \(e\le 1,t\), then it vanishes on \(I_{1,t} = \overline{\sum
    I_{1,e}}\).  Then \(\bid\varphi\circ E(\xi \delta_t)=0\).
\end{proof}

\section{Iterated crossed products}
\label{sec:iterated_crossed}

We now study reduced crossed products associated to actions of inverse
semigroups on groupoids.  Let~\(G\)
be a locally compact, locally Hausdorff groupoid with Haar system.
Let~\(S\)
be a unital inverse semigroup acting on~\(G\)
by partial equivalences.  Let~\(G\rtimes S\)
be the transformation groupoid as defined
in~\cite{Buss-Meyer:Actions_groupoids}.  This comes with a canonical
\(S\)\nb-grading,
that is, with open subsets \(G_t\subseteq G\rtimes S\)
for \(t\in S\) that satisfy
\[
G_t\cdot G_u = G_{tu},\quad
G_t^{-1} = G_{t^*},\quad
G_t\cap G_u = \bigcup_{v\le t,u} G_v,\quad
(G\rtimes S)^1=\bigcup_{t\in S} G_t.
\]
We have \(G\cong G_1\),
so~\(G\) is an open subgroupoid of~\(G\rtimes S\).

For instance, if~\(G\) is a space, then an action of~\(S\) on~\(G\) by
partial equivalences is essentially the same as an action on the
space~\(G\) by partial homeomorphisms, and the transformation groupoid
is the usual one.  If~\(S\) is a group, then an action of~\(G\) by
topological groupoid automorphisms is an action by equivalences, and
\(G\rtimes S\) is the semidirect product groupoid.  More generally,
any group extension \(H\into L\onto G\) comes from an action by
equivalences with \(L = H\rtimes G\).  These and several other
examples are explained in~\cite{Buss-Meyer:Actions_groupoids}.
Recently, we have developed a similar iterated crossed product theorem
for actions of possibly non-étale groupoids instead of inverse
semigroup actions in~\cite{Buss-Meyer:Groupoid_fibrations}.  In this
context, we also give several more examples that come from inverse
semigroup actions.  In particular, we observe that actions of a
discrete group~\(G\) on groupoids by equivalences are equivalent to
``strongly surjective'' \(G\)\nb-valued cocycles.

Let~\(\B\)
be a Fell bundle over~\(G\rtimes S\).
Since~\(G\)
is a subgroupoid of~\(G\rtimes S\),
we may restrict~\(\B\)
to a Fell bundle over~\(G\),
which we still call~\(\B\).
Let \(\Cst(G\rtimes S,\B)\)
and \(\Cst(G,\B)\)
be the full section \(\Cst\)\nb-algebras
of these Fell bundles.  Let \(\Cred(G\rtimes S,\B)\)
and \(\Cred(G,\B)\)
be the reduced section \(\Cst\)\nb-algebras;
these are
defined as follows.  There is a canonical (``left regular'')
representation~\(\Lambda_x\)
of~\(\Cst(G,\B)\)
on the Hilbert \(\B_x\)\nb-module
\(L^2(G_x,\B)\)
for each \(x\in G^0\),
and \(\Cred(G,\B)\)
is the image of~\(\Cst(G,\B)\)
in the representation
\(\Lambda\defeq \bigoplus_{x\in G^0} \Lambda_x\).

The \(S\)\nb-action
on~\(G\)
induces an \(S\)\nb-action
on~\(\Cst(G,\B)\)
by Hilbert bimodules~\((\Hilm_t)_{t\in S}\),
that is, a Fell bundle over~\(S\)
with unit fibre~\(\Cst(G,\B)\),
see~\cite{Buss-Meyer:Actions_groupoids}.  The Hilbert bimodule
\(\Hilm_t\defeq \Cst(G_t,\B)\)
is a completion of~\(\Sect(G_t,\B)\),
the space of quasi-continuous sections of~\(\B\)
over~\(G_t\)
with compact support, where quasi-continuity means finite linear
combinations of compactly supported functions on Hausdorff, open
subsets, extended by~\(0\)
outside their domain.  The space~\(\Sect(G_t,\B)\)
is a full pre-Hilbert
\(\Sect(G_{tt^*},\B)\)-\(\Sect(G_{t^*t},\B)\)-bimodule,
and \(\Cst(G_t,\B)\)
is its completion to a full Hilbert
\(\Cst(G_{tt^*},\B)\)-\(\Cst(G_{t^*t},\B)\)-bimodule.
Here we assume the expected results for Morita--Rieffel equivalence of
Fell bundle crossed products to hold in the case at hand; this is so
far only proved for Fell bundles over Hausdorff groupoids and for
special Fell bundles (such as actions by automorphisms twisted by
scalar-valued cocycles)
over non-Hausdorff groupoids.
The two issues are the positivity of the inner product and the
existence of the left action, which at first is only defined on a
dense subspace.  The Disintegration Theorem, if it holds in the case
at hand, implies both claims as in \cite{Renault:Representations} or
\cite{Holkar:Thesis}*{Sections 2.2--2.3}.

The reduced version of the Fell bundle is much easier to construct,
using the following observation:

\begin{lemma}
  \label{lem:reduced_embed}
  The inclusion map \(\Sect(G,\B)\to\Sect(G\rtimes S,\B)\)
  extends to a faithful, nondegenerate \Star{}homomorphism
  \(\Cred(G,\B) \to \Cred(G\rtimes S,\B)\).
\end{lemma}

The analogue of this lemma for full section algebras is true, see
\cite{Buss-Meyer:Actions_groupoids}*{Corollary 5.7}, but the only
proof we know uses the theorem about iterated crossed products.

\begin{proof}
  Fix \(x\in G^0 = (G\rtimes S)^0\)
  and let \(S_x = \{t\in S\mid x\in \s(G_t)\}\).
  Equivalently, \(t\in S_x\)
  if and only if the action of~\(t\)
  on~\(G^0/G\)
  is defined on the orbit of~\(x\).
  Say that~\(t\sim_x u\)
  if there is \(v\in S_x\)
  with \(v\le t,u\);
  this is an equivalence relation on~\(S_x\).
  We have \(t\sim_x u\)
  if and only if \([t,x]=[u,x]\)
  in \((G\rtimes S)^1\).
  For each equivalence class~\(h\)
  in~\(S_x/{\sim_x}\),
  pick a representative~\(t_h\)
  and pick \(\eta_h\in G_{t_h}\)
  with \(\s(g_h)=x\).
  Since~\(G_{t_h}\)
  is a partial equivalence of~\(G\),
  the source map induces an injective, open, continuous map
  \(G\backslash G_{t_h} \to G^0\).
  Thus the left \(G\)\nb-action
  gives a \(G\)\nb-equivariant
  homeomorphism \(G_{\rg(\eta_h)} \cong (G_{t_h})_x\),
  \(g\mapsto g\eta_h\).
  These maps piece together to a \(G\)\nb-equivariant
  homeomorphism
  \((G\rtimes S)_x \cong \bigsqcup_{h\in S_x/{\sim_x}} G_{\s(\eta_h)}\).
  Thus the restriction of the left regular representation
  \(\Lambda^{G\rtimes S}_x\)
  to~\(\Cst(G,\B)\)
  is unitarily equivalent to the sum of the representations
  \(\bigoplus_{h\in S_x/{\sim_x}} \Lambda^G_{\s(\eta_h)}\).
  Since \(1\in S_x\),
  this direct sum contains the regular representation~\(\Lambda^G_x\).

  By definition, \(\Cred(G\rtimes S,\B)\)
  is the completion of~\(\Sect(G\rtimes S,\B)\)
  in the \(\Cst\)\nb-seminorm
  coming from the regular representations at all \(x\in G^0\).
  As we just saw, the restriction of a regular representation
  of~\(\Sect(G\rtimes S,\B)\)
  to~\(\Sect(G,\B)\)
  is a sum of regular representations of~\(\Sect(G,\B)\),
  and every regular representation of~\(\Sect(G,\B)\)
  occurs in these sums.  Hence the reduced \(\Cst\)\nb-seminorm
  on~\(\Sect(G\rtimes S,\B)\)
  restricts to the reduced \(\Cst\)\nb-seminorm
  on~\(\Sect(G,\B)\),
  and the \Star{}homomorphism
  \(\Cred(G,\B) \to \Cred(G\rtimes S,\B)\)
  is faithful.
  It is nondegenerate because
  \(\Sect(G,\B)\cdot \Sect(G\rtimes S,\B) = \Sect(G\rtimes S,\B)\).
\end{proof}

Let~\(\Cred(G_t,\B)\)
be the closure of \(\Sect(G_t,\B)\)
in~\(\Cred(G\rtimes S,\B)\).
The space~\(\Sect(G_t,\B)\)
is a pre-Hilbert \(\Sect(G,\B)\)-\(\Sect(G,\B)\)-bimodule
using the convolution and involution in~\(\Sect(G\rtimes S,\B)\).
Thus the closure \(\Cred(G_t,\B)\)
becomes a Hilbert bimodule over the closure of~\(\Sect(G,\B)\)
in \(\Cred(G\rtimes S,\B)\),
which is~\(\Cred(G,\B)\)
by Lemma~\ref{lem:reduced_embed}.  Moreover,
\[
\Cred(G_t,\B) \otimes_{\Cred(G,\B)} \Cred(G_u,\B) \cong
\Cred(G_{t u},\B)
\]
because the convolution map
\[
\Sect(G_t,\B) \otimes_{\Sect(G,\B)} \Sect(G_u,\B) \to
\Sect(G_{t u},\B)
\]
in~\(\Sect(G\rtimes S,\B)\)
has dense image for all \(t,u\in S\).
Thus \(\Cred(G_t,\B)\)
for \(t\in S\)
is a saturated Fell bundle over~\(S\)
with unit fibre \(\Cred(G,\B)\).
Let \(\Cred(G,\B)\rtimes_\red S\)
be its reduced section \(\Cst\)\nb-algebra.
When is this \(\Cst\)\nb-algebra
isomorphic to \(\Cred(G\rtimes S,\B)\)?

We first attempted to prove this, until we found counterexamples.  The
problem is related to a possible failure of the following weak
exactness property:

\begin{definition}[see
  \cite{AnantharamanDelaroch:Weak_containment}*{Definition 2.6}]
  \label{def:inner_exact}
  Let~\(G\)
  be a locally compact, locally Hausdorff groupoid with Haar system.
  Let~\(\B\)
  be a Fell bundle over~\(G\).
  We say that~\(\B\)
  is \emph{inner exact}
  if, for each \(G\)\nb-invariant
  open subset \(U\subseteq G^0\)
  and \(F\defeq G^0\backslash G\), the sequence
  \begin{equation}
    \label{eq:Sequence-Definition-B-exact}
    \Cred(G_U,\B)\into \Cred(G,\B)\onto \Cred(G_{F},\B)
  \end{equation}
  is exact.  We call~\(G\)
  \emph{inner exact} if the Fell bundle~\(\B\)
  describing the canonical (or ``trivial'') action of~\(G\)
  on~\(G^0\) is inner exact, that is, if the sequence
  \[
  \Cred(G_U)\into \Cred(G)\onto \Cred(G_{F})
  \]
  is exact for each \(G\)\nb-invariant
  open subset \(U\subseteq G^0\).
\end{definition}

Minimal groupoids are inner exact for all Fell bundles because they
have no non-trivial \(G\)\nb-invariant open subsets \(U\subseteq
G^0\).

The analogue of the
sequence~\eqref{eq:Sequence-Definition-B-exact} with full
cross-section \cstar{}algebras of Fell bundles is always exact, at
least for Hausdorff locally compact groupoids,
see~\cite{Ionescu-Williams:Remarks_ideal_structure}.  Since amenable
groupoids always have isomorphic full and reduced Fell bundle
cross-section \cstar{}algebras
(see~\cite{Sims-Williams:Amenability_for_Fell_bundles_over_groupoids}),
all Fell bundles over amenable, locally compact, Hausdorff groupoids
(with Haar system) are inner exact.

It is easy to find examples of groupoids which are not inner exact and
non-Hausdorff (see Example~\ref{exa:Khoshkam-Skandalis}).  Hausdorff
examples are also known, but more subtle (see
Example~\ref{exa:Higson-Lafforgue-Skandalis}).  Inner exactness
of~\(G\)
does not imply inner exactness of all Fell bundles~\(\B\).
For instance, any trivial group bundle \(G=X\times \Gamma\)
is inner exact for the trivial Fell bundle, but not necessarily for
all Fell bundles if~\(\Gamma\)
is not exact (see Example~\ref{exa:KirchbergWasserman}).

\subsection{Counterexamples}
\label{sec:Examples}

Let~\(G\)
be a groupoid and let~\(\B\)
be a Fell bundle over~\(G\).
Assume that~\(\B\)
is \emph{not} inner exact.
Hence there is an open \(G\)\nb-invariant
subset \(U\subseteq G^0\)
such that the sequence~\eqref{eq:Sequence-Definition-B-exact}
is \emph{not} exact.  We are going to construct an action of an
inverse semigroup~\(S\)
on~\(G\)
with \(\Cred(G,\B)\rtimes_\red S\ncong \Cred(G\rtimes S,\B)\).

The inverse semigroup~\(S\)
and its action on~\(G\)
are embarrassingly trivial.  Let~\(S\)
be the inverse semigroup with three elements~\(0,1,-1\),
with usual number multiplication.  Thus~\(S\)
is the group \(\Z/2\cong \{1,-1\}\)
with a zero element added.
The same inverse semigroup is used
in~\cite{Buss-Meyer:Actions_groupoids} to give an example of an action
by partial Morita--Rieffel equivalences that is not equivalent to an
action by partial automorphisms.  We let both \(1\)
and~\(-1\)
in~\(S\)
act by the identity automorphism on~\(G\),
and we let~\(0\)
act by the identity on the open subgroupoid~\(G_U\),
where~\(U\)
witnesses the lack of inner exactness.  We use the obvious
multiplication isomorphisms.  We describe the transformation groupoid
and its \(S\)\nb-grading.
The transformation groupoid is the quotient of the product groupoid
\(G\times \Z/2\)
by the equivalence relation that is generated by \((g,1)\sim (g,-1)\)
for \(g\in G_U\).  The \(S\)\nb-grading is
\[
G_1 = \{[g,1] \mid g\in G^0\},\quad
G_{-1} = \{[g,-1] \mid g\in G^0\},\quad
G_0 = G_1\cap G_{-1} \cong G_U.
\]

The resulting action of~\(S\)
on \(A\defeq \Cred(G,\B)\)
is trivial on~\(\Z/2\),
and the idempotent \(0\in S\)
acts by the ideal \(I\defeq \Cred(G_U,\B)\)
in~\(A\).  We are going to prove:

\begin{lemma}
  \label{lem:reduced_crossed_products_counterexamples}
  In the situation above,
  \(\Cred(G,\B)\rtimes_\red S\ncong \Cred(G\rtimes S,\B)\)
  whenever inner exactness fails.  Explicitly,
  \[
  \Cred(G,\B)\rtimes_\red S
  \cong \Cred(G,\B) \oplus \frac{\Cred(G,\B)}{\Cred(G_U,\B)},\qquad
  \Cred(G\rtimes S,\B) \cong \Cred(G,\B) \oplus \Cred(G_F,\B).
  \]
\end{lemma}

Although the proof only requires rather trivial actions of~\(S\),
we describe the reduced crossed product~\(A\rtimes_\red S\)
for a general \(S\)\nb-action
on a \(\Cst\)\nb-algebra~\(A\),
and we also discuss the weak conditional expectation and the
\(\Cst\)\nb-subalgebra
of~\(\bid{A}\)
that it generates in order to illustrate our theory.

Let~\(S\)
act on a \(\Cst\)\nb-algebra~\(A\).
The idempotents \(0,1\)
act by \(A_0=\Id_I\)
and \(A_1=\Id_A\)
for some ideal \(I\idealin A\).
For the counterexamples, it is important to choose \(I\neq 0\).
The element \(-1\in S\)
acts by a full Hilbert bimodule~\(\Hilm_{-1}\).
The action also contains isomorphisms of Hilbert bimodules
\[
\Hilm_{-1}\otimes_A \Hilm_{-1} \cong A,\qquad
I\otimes_A \Hilm_{-1} \cong I \cong \Hilm_{-1} \otimes_A I;
\]
the remaining isomorphisms are canonical.  The multiplication
isomorphisms above have to be associative as well.  Since this implies
compatibility with the involution, the two isomorphisms
\(I\otimes_A \Hilm_{-1} \cong I\)
and \(\Hilm_{-1} \otimes_A I\cong I\)
determine each other, so we have to specify only one of them.

The Hilbert bimodule~\(\Hilm_{-1}\)
and the isomorphism \(\Hilm_{-1}\otimes_A \Hilm_{-1} \cong A\)
give a (saturated) Fell bundle over the group~\(\Z/2\).
Any such Fell bundle may be turned into an ordinary action by
automorphisms by a stabilisation.  The quickest way to see this uses
Takesaki--Takai duality.  The section algebra of the Fell bundle
over~\(\Z/2\)
carries a dual action of the dual group (which is again~\(\Z/2\)),
and the crossed product for that dual action carries an action
of~\(\Z/2\)
by automorphisms and is \(\Z/2\)\nb-equivariantly
isomorphic to \(A\otimes \Mat_2(\C)\)
by Takesaki--Takai Duality.  Hence it is no serious loss of generality
to replace the full Hilbert bimodule~\(\Hilm_{-1}\)
and the isomorphism \(\Hilm_{-1}\otimes_A \Hilm_{-1} \cong A\)
by an automorphism~\(\alpha\)
of~\(A\)
with \(\alpha^2=\Id_A\).
Then \(\Hilm_{-1}=A\)
with the usual right Hilbert module structure and the left action
through~\(\alpha\).

Since \(\Hilm_{-1}\cong A\) both as a left and a right Hilbert
module, there are canonical isomorphisms \(\Hilm_{-1} \otimes_A I\cong I\)
as a right Hilbert module and \(I \otimes_A \Hilm_{-1}\cong I\) as a
left Hilbert module; the first is of the form \(a\otimes b\mapsto a
b\), the second of the form \(a\otimes b\mapsto \alpha(a) b\).  Any
other such isomorphisms are obtained from these ones by composing with
unitary multipliers of~\(I\).  Thus the isomorphisms \(\Hilm_{-1}
\otimes_A I\cong I\) and \(I \otimes_A \Hilm_{-1}\cong I\) are of
the form \(a\otimes b\mapsto u a b\) and \(a\otimes b\mapsto u'
\alpha(a) b\) for unitary multipliers \(u,u'\) of~\(I\);
compatibility with the bimodule structure gives \(\alpha|_I^{-1} =
\Ad(u)=\Ad(u')\).  In particular, the ideal~\(I\) is
\(\alpha\)\nb-invariant.  The associativity conditions for a Fell
bundle are equivalent to the conditions \(u=u'\) and \(u^2=1\),
which are reasonable requests because \(\Ad(u^2)= \alpha^2|_I =
\Id_A\).  Thus after some simplifications, an action of~\(S\)
on~\(A\) becomes a triple \((I,\alpha,u)\) with \(I\idealin
A\), \(\alpha\in\Aut(A)\), \(u\in\Mult(I)\), \(\alpha^2=\Id_A\),
\(\alpha|_I = \Ad(u)\), and \(u^2=1\).

\begin{proposition}
  \label{pro:crossed_by_01-1}
  The subset \(J\defeq \{\delta_1 a - \delta_{-1} ua \mid a\in I\}\)
  in~\(A\) is a closed ideal and there are canonical isomorphisms
  \[
  A\rtimes S
  \cong A\rtimes_\red S
  \cong A\rtimes_\alg S
  \cong (A\rtimes_\alpha\Z/2)/J.
  \]
\end{proposition}

\begin{proof}
  The algebraic crossed product~\(A\rtimes_\alg\Z/2\)
  consists of elements of the form \(\delta_1 a+ \delta_{-1} b\),
  \(a,b\in A\),
  with the usual relations of a crossed product, generated by
  \(\delta_{-1}^2=\delta_1\)
  and \(\delta_{-1} a\delta_{-1} = \alpha(a)\).
  It is already a \(\Cst\)\nb-algebra
  because~\(\Z/2\)
  is finite.  The algebraic crossed product~\(A\rtimes_\alg S\)
  consists of equivalence classes of sums of the form
  \(\delta_1 a+ \delta_{-1} b+ \delta_0 c\)
  with \(a,b\in A\),
  \(c\in I\),
  where we divide out the relations
  \(\delta_0 c \equiv \delta_1 c \equiv \delta_{-1} uc\)
  for all \(c\in I\)
  because \(j_{1,0}(c) = c\)
  and \(j_{-1,0}(c) = u c\) for all \(c\in I\).  In our case,
  \(j_{1,0} = \theta_{1,0}\) and \(j_{-1,0} = \theta_{-1,0}\).

  Thus the obvious \Star{}homomorphism \(A\rtimes_\alg \Z/2\to
  A\rtimes_\alg S\) is surjective, and \(A\rtimes_\alg S\) is the
  quotient of~\(A\rtimes_\alg\Z/2\) by~\(J\), forcing it to be a
  \Star{}ideal.  The map
  \[
  I \to A\rtimes \Z/2,\qquad c\mapsto \frac{1}{2}(\delta_1 c - \delta_{-1} u c),
  \]
  is a \Star{}homomorphism and~\(J\) is its image.  So~\(J\) is closed
  and \(A\rtimes_\alg S = (A\rtimes_\alg \Z/2)/J\) is a
  \(\Cst\)\nb-algebra.  Thus the full \(\Cst\)\nb-algebra is
  \(A\rtimes_\alg S\) with its quotient norm.  So is \(A\rtimes_\red
  S\) because the map \(A\rtimes_\alg S \to A\rtimes_\red S\) is
  injective by Proposition~\ref{pro:embed_alg_red}.
\end{proof}

We describe the weak conditional expectation
\(E\colon A\rtimes S\to \bid{A}\)
and \(\tilde{A}\subseteq \bid{A}\).
The ideal~\(I_{1,-1}\)
is~\(I\), and the proof of Proposition~\ref{pro:A_tilde_A} implies
\[
E(\delta_1 a+ \delta_{-1} b) = a + ub[I] = a + \alpha(b) u[I]
= a + \alpha(b) u
\]
for all \(a,b\in A\), where the products take place in~\(\bid{A}\).
The last step uses \(u^* u = u u^* = [I]\), which holds because
\(u\in\Mult(I)\) is unitary.  Notice that \(E|_J=0\).  The element
\(u\in\bid{A}\) is a self-adjoint partial isometry with \(u^2=[I]\).
It satisfies \(\alpha(b) u c = E(\delta_{-1} b)c = E(\delta_{-1} b
c) = \alpha(b c) u\) for all \(b,c\in A\), which gives \(u c =
\alpha(c) u\) in~\(\bid{A}\) for \(c\in A\).  So elements of the
form~\(b[I]\) also belong to~\(\tilde{A}\).  The subset
\(\{a(1-[I])+b u + c[I]\mid a,b,c\in A\}\) is a \Star{}subalgebra
of~\(\bid{A}\).  We will show below that it is closed; hence this
is~\(\tilde{A}\).

The elements of the form \(a(1-[I])\) and \(b u + c[I]\) are
orthogonal, and \(a(1-[I])=0\) if and only if \(a=a[I]\), that is,
\(a\in I\).  Thus \(\{a (1-[I]) \mid a\in A\}\) gives a summand
isomorphic to~\(A/I\) in~\(\tilde{A}\).  If \(b\in I^\bot = \{a\in
A\mid a\cdot b=0\) for all \(b\in I\}\), then \(b u=0\) and \(b
[I]=0\), and vice versa.  Hence there is a well-defined map
\[
\varphi\colon A/I^\bot\rtimes_{\alpha} \Z/2 \to \bid{A},\qquad
[b]\delta_{-1}+[c]\delta_1\mapsto b u+c[I].
\]
This is a \Star{}homomorphism because \(u^2=[I]\) and \(u b =
\alpha(b) u\) for all \(b\in A\).  Hence its image is closed.  This
image together with the other summand~\(A/I\) is~\(\tilde{A}\).  The
\Star{}homomorphism~\(\varphi\) may fail to be injective, for
instance, if \(u=[I]\).  We do not describe the kernel
of~\(\varphi\).

After this illustration of our general theory, we return to the
proof of
Lemma~\ref{lem:reduced_crossed_products_counterexamples}.  Here
\(u=1\)
and \(\alpha=\Id\),
that is, the only non-trivial aspect of the action is the ideal~\(I\).
The Fourier isomorphism \(\Cst(\Z/2)\cong \C^2\)
induces an isomorphism
\[
A\rtimes_\alpha \Z/2 = A\otimes \Cst(\Z/2)\cong A\oplus A,\qquad
\delta_1 a+\delta_{-1} b\mapsto (a+b,a-b).
\]
This isomorphism maps the ideal \(J\subseteq A\rtimes_\alpha\Z/2\)
to \(0\oplus I\).  Therefore,
\begin{equation}
  \label{eq:IsoTrivialActionWithIdealI}
  A\rtimes S = A\rtimes_\red S = A\rtimes_\alg S \cong A\oplus (A/I)
\end{equation}
in this ``trivial'' special case.  This proves the assertion about
\(\Cred(G,\B)\rtimes_\red S\)
in Lemma~\ref{lem:reduced_crossed_products_counterexamples}.

Now we turn to the reduced groupoid \(\Cst\)\nb-algebra
of~\(G\rtimes S\).  We use the surjection
\[
\Sect(G,\B)\oplus \Sect(G,\B) \cong
\Sect(G,\B)\otimes \C[\Z/2] \cong
\Sect(G\rtimes \Z/2,\B) \onto
\Sect(G\rtimes S,\B),
\]
where the first isomorphism is the inverse Fourier isomorphism
for~\(\Z/2\)
and the last map is the obvious one,
\(a\delta_1 + b\delta_{-1}\mapsto a\delta_1 + b\delta_{-1}\).
Thus~\(\Cred(G\rtimes S,\B)\)
is isomorphic to the completion of \(\Sect(G,\B)\oplus \Sect(G,\B)\)
for the family of regular representations of~\(G\rtimes S\),
viewed as representations of \(\Sect(G,\B)\oplus \Sect(G,\B)\)
through the surjection above.  If \(x\in U\),
then \((G\rtimes S)_x = G_x\).
Hence the resulting regular representation is the standard regular
representation for~\(x\)
on one summand~\(\Sect(G,\B)\)
(corresponding to the trivial character on~\(\Z/2\))
and kills the other summand.  If \(x\in F\),
then \((G\rtimes S)_x = G_x\times\Z/2\),
and \(\Z/2\)
acts freely on the second factor.  Hence the resulting regular
representation is the standard regular representation for~\(x\)
on both summands~\(\Sect(G,\B)\).
Thus
\(\Cred(G\rtimes S,\B) \cong \Cred(G,\B) \oplus \Cred(G_F,\B)\).
This finishes the proof of
Lemma~\ref{lem:reduced_crossed_products_counterexamples}.

We still have to exhibit examples where inner exactness fails.  Here
we may use counterexamples produced for other purposes in the
literature.  The first example comes from~\cite{Khoshkam-Skandalis:Regular}.

\begin{example}
  \label{exa:Khoshkam-Skandalis}
  For a discrete group~\(\Gamma\),
  let~\(G^\Gamma\)
  be the (étale) group bundle over~\([0,1]\)
  with fibres \(G^\Gamma(1)=\Gamma\)
  and the trivial fibre at \(x\neq1\);
  this is the quotient of the group bundle \([0,1]\times\Gamma\)
  by the relation \((x,\gamma)\sim (x,\gamma')\)
  for all \(x\in [0,1)\), \(\gamma,\gamma'\in \Gamma\).

  We assume that~\(\Gamma\) is not amenable, say, a free group.  Then
  \[
  \Cred(G^\Gamma)\cong \Cont[0,1]\oplus \Cred(\Gamma),
  \]
  see \cite{Khoshkam-Skandalis:Regular}*{p.~53}.  Let \(U=[0,1)\)
  and \(F=\{1\}\).
  Then \(G^\Gamma_U \cong [0,1)\)
  and \(G^\Gamma_F = \Gamma\),
  so \(\Cred(G^\Gamma_U) \cong \Cont_0([0,1))\)
  and \(\Cred(G^\Gamma_F) \cong \Cred(\Gamma)\).
  Hence
  \(\Cred(G^\Gamma)/\Cred(G^\Gamma_U) \cong \C\oplus \Cred(\Gamma)
  \ncong \Cred(G^\Gamma_F)\),
  so inner exactness fails here.  Hence the ``trivial'' action of
  \(S=\{0,1,-1\}\)
  given by~\(U\)
  satisfies
  \(\Cred(G^\Gamma\rtimes S)\ncong \Cred(G^\Gamma)\rtimes_\red S\)
  by Lemma~\ref{lem:reduced_crossed_products_counterexamples}.
\end{example}

In the above example, the groupoid~\(G\)
on which~\(S\)
acts is non-Hausdorff.  Now we give counterexamples where~\(G\)
is a Hausdorff (étale) groupoid, but we also put a rather trivial Fell
bundle~\(\B\) on~\(G\).

\begin{example}
  \label{exa:KirchbergWasserman}
  Let~\(\Gamma\) be a non-exact group.  Such groups exist by an
  argument of Gromov (see~\cites{Gromov:Spaces_questions,
    Gromov:Random}).  The \(\Cst\)\nb-algebra \(\Cred(\Gamma)\) is
  not exact, that is, there is a \(\Cst\)\nb-algebra~\(A\) and an
  ideal \(I\subseteq A\) such that the sequence
  \begin{equation}
    \label{eq:NonExactKirchbergWasserman}
    I\otimes \Cred(\Gamma)\into
    A\otimes \Cred(\Gamma)\onto
    (A/I)\otimes \Cred(\Gamma)
  \end{equation}
  is not exact (see
  \cite{Kirchberg-Wassermann:Exact_groups}*{Theorem~5.2}).  Here and
  throughout, \(\otimes\) denotes the minimal \(\Cst\)\nb-algebra
  tensor product.  Kirchberg and
  Wassermann~\cite{Kirchberg-Wassermann:Operations} show that~\(A\)
  may be chosen to be the section \(\Cst\)\nb-algebra of a
  continuous field~\((A_x)_{x\in X}\) over the one-point
  compactification \(X\defeq \N^+ = \N\cup\{\infty\}\) and \(I=
  \Cont_0(\N)\cdot A\).  Then \(A/I=A_\infty\) is the fibre at
  \(\infty\in\N\).  Let \(G=X\times\Gamma\) be the trivial group
  bundle over~\(X\) with fibre~\(\Gamma\) everywhere.
  Let~\(\Gamma\) act trivially on~\(A\) and form the Fell
  bundle~\(\B\) over~\(G\) with fibres \((A_x)_{x\in X}\) from this
  trivial action.  Then inner exactness of~\(\B\) fails for the
  open subset \(\N\subseteq X\) by construction.  Since commutative
  \(\Cst\)\nb-algebras are nuclear, the sequence
  \[
  \Cont_0(V)\otimes \Cred(\Gamma)
  \into \Cont(X)\otimes \Cred(\Gamma)
  \onto \Cont(X\setminus V)\otimes \Cred(\Gamma)
  \]
  is exact for any open subset \(V\subseteq X\).  So our
  groupoid~\(G\) is inner exact but a certain Fell bundle~\(\B\) is
  not inner exact.
\end{example}

\begin{example}
  \label{exa:Higson-Lafforgue-Skandalis}
  This example is used
  in~\cite{Higson-Lafforgue-Skandalis:Counterexamples} as a
  counterexample for the Baum--Connes conjecture for groupoids.  It
  has also been used recently by Willett~\cite{Willett:Non-amenable}
  as examples of Hausdorff, étale, non-amenable groupoids~\(G\) with
  \(\Cst(G)\cong \Cred(G)\).

  Let~\(\Gamma\)
  be a countable discrete group and let \(\Gamma_n\subseteq \Gamma\)
  be a sequence of normal subgroups of finite index.  Let
  \(F_n\defeq \Gamma/\Gamma_n\)
  be the quotient groups and let \(\pi_n\colon \Gamma\to F_n\)
  be the quotient homomorphisms.  Let \(\bar \N=\N\sqcup \{\infty\}\)
  be the one-point compactification of~\(\N\)
  and let~\(G\)
  be the quotient of \(\bar\N\times \Gamma\)
  by the equivalence relation:
  \[
  (n,g)\sim (m,h)\Longleftrightarrow m=n\in \N
  \quad\text{and}\quad
  \pi_n(g)=\pi_n(h).
  \]
  We write~\([n,g]\)
  for the equivalence class of~\((n,g)\).
  The resulting étale group bundle~\(G\)
  has fibres \(G_n\cong F_n\)
  for \(n\in \N\)
  and \(G_\infty\cong \Gamma\).
  It is Hausdorff if and only if \(\bigcap \Gamma_n = \{1\}\).

  If \(\Gamma_n=\Gamma\) for all~\(n\), then~\(G\) is a
  (non-Hausdorff) group bundle of the same type as in
  Example~\ref{exa:Khoshkam-Skandalis}.  However, we want the
  groupoid~\(G\) to be Hausdorff.  So we assume that~\(\Gamma\) is
  residually finite and choose the subgroups above with \(\bigcap
  \Gamma_n=\{1\}\); for instance, \(\Gamma\) may be a free group or
  \(\textup{SL}_n(\Z)\).  If~\(\Gamma\) is not amenable, then~\(G\)
  is not inner exact because the sequence
  \[
  \Cred(G_\N)\into \Cred(G)\onto \Cred(\Gamma)
  \]
  is not exact in the middle, not even at the level of K\nb-theory,
  see~\cite{Higson-Lafforgue-Skandalis:Counterexamples}.
\end{example}

\subsection{Isomorphism of reduced iterated crossed products}
\label{sec:iso_reduced}

Our counterexamples show that there is no isomorphism
\(\Cred(G,\B)\rtimes_\red S\cong \Cred(G\rtimes S,\B)\)
in general.  There are, however, many cases where this isomorphism
holds.  A first result of this nature is proved
in~\cite{Deaconu-Kumjian-Ramazan:Fell_groupoid_morphism} for groupoid
fibrations of \emph{Hausdorff} \'etale groupoids with amenable kernel.

\begin{theorem}
  \label{the:iterated_crossed_1}
  Let~\(G\)
  be a locally Hausdorff, locally compact groupoid with Haar
  system.  Let~\(S\)
  be an inverse semigroup acting on~\(G\)
  by partial equivalences.  Let~\(\B\)
  be a Fell bundle over the transformation
  groupoid~\(G\rtimes S\) that is
  inner exact over~\(G\).
  Then \(\Cred(G,\B)\rtimes_\red S \cong \Cred(G\rtimes S,\B)\).
\end{theorem}

\begin{proof}
  First we show that the regular representation
  \(\Lambda = \bigoplus_{x\in G^0}\Lambda_x\)
  of~\(\Cred(G,\B)\)
  on~\(L^2(G^x)\)
  is \(E\)\nb-faithful.

  Let~\(\mathcal{I}\)
  be the lattice of ideals of \(\Cred(G,\B)\)
  of the form \(\Cred(G_U,\B)\)
  for open, \(G\)\nb-invariant
  subsets \(U\subseteq G^0\);
  this is a lattice because
  \(\Cred(G_U,\B)\cap \Cred(G_V,\B) = \Cred(G_{U\cap V},\B)\)
  and \(\Cred(G_U,\B)+ \Cred(G_V,\B) = \Cred(G_{U\cup V},\B)\).
  The ideals~\(I_{1,t}\)
  in the construction of the conditional expectation~\(E\)
  belong to~\(\mathcal{I}\).
  Thus we may use the criterion in Theorem~\ref{the:E-faithful}.  If
  \(U\subseteq G^0\) is open and \(G\)\nb-invariant, then
  \[
  \Cred(G_U,\B) \cdot \Lambda = \bigoplus_{x\in U} \Lambda_x.
  \]
  Let \(U\subseteq V\subseteq G^0\)
  be open, \(G\)\nb-invariant
  subsets.  Then \(a\in\Cred(G_V,\B)\)
  satisfies
  \(\Lambda(a)\Lambda(\Cred(G_V,\B)) \subseteq
  \Lambda(\Cred(G_U,\B))\)
  if and only if \(\Lambda_x(a)=0\)
  for all \(x\in V\setminus U\).
  This means that~\(a\)
  is mapped to~\(0\)
  in \(\Cred(G_{V\setminus U},\B)\).
  Since~\(\B\)
  is inner exact,
  this implies \(a\in\Cred(G_U,\B)\).
  Thus~\(\Lambda\)
  verifies the criterion in Theorem~\ref{the:E-faithful} and is
  \(E\)\nb-faithful.  Now the following lemma finishes the proof.
\end{proof}

\begin{lemma}
  \label{lem:regular_E-faithful}
  If the family of regular representations of~\(\Cred(G,\B)\)
  is \(E\)\nb-faithful,
  then \(\Cred(G,\B)\rtimes_\red S \cong \Cred(G\rtimes S,\B)\).
\end{lemma}

\begin{proof}
  The isomorphism \(\Cst(G,\B)\rtimes S \cong \Cst(G\rtimes S,\B)\)
  is proved in~\cite{Buss-Meyer:Actions_groupoids}.  The sum of the
  induced representations \(\Ind\Lambda_x\)
  for \(x\in G^0\)
  is a faithful representation of the quotient
  \(\Cred(G,\B)\rtimes_\red S\)
  of \(\Cst(G,\B)\rtimes S\)
  by Proposition~\ref{pro:E-faithful} and our assumption.
  As in the proof
  of Lemma~\ref{lem:reduced_embed}, the induced
  representation \(\Ind \Lambda_x\)
  is the regular representation of
  \(\Cst(G,\B)\rtimes S\cong \Cst(G\rtimes S,\B)\)
  on \(L^2((G\rtimes S)^x,\mu^x,\B)\),
  where~\(\mu\)
  is the unique Haar system on~\(G\rtimes S\)
  extending the given Haar system on~\(G\)
  (see \cite{Buss-Meyer:Actions_groupoids}*{Proposition~5.1}).  Hence
  the reduced norm that gives \(\Cred(G\rtimes S,\B)\)
  is defined by the same family of representations that gives the norm
  on \(\Cred(G,\B)\rtimes_\red S\).
\end{proof}

\begin{corollary}
  \label{cor:exact=>iterated-crossed-product}
  If~\(G\)
  is inner exact, then
  \(\Cred(G)\rtimes_\red S\cong \Cred(G\rtimes S)\).
\end{corollary}

The criteria above are not optimal.  Of course, it suffices to require
inner exactness only for the lattice generated by the open
\(G\)\nb-invariant
subsets of~\(G^0\)
corresponding to the ideals~\(I_{1,t}\).
A more serious limitation of our proof is that we only use
\(E(\Hilm_t)\subseteq \Mult(I_{1,t})\).
For instance, it does not use that \(E(\Hilm_t)\subseteq A\)
if~\(t\)
is idempotent.  This is why the following theorem is not a special
case of Theorem~\ref{the:iterated_crossed_1}:

\begin{theorem}
  \label{the:conditional_expect_groupoid_crossed}
  Let~\(G\)
  be a locally compact, locally Hausdorff groupoid with an action of a
  unital inverse semigroup~\(S\),
  and let~\(\B\)
  be a Fell bundle over the transformation groupoid~\(G\rtimes S\).
  If~\(G\)
  is closed in~\(G\rtimes S\),
  then the canonical conditional expectation
  \(\Cred(G,\B)\rtimes_\red S \to \bid{\Cred(G,\B)}\)
  takes values in~\(\Cred(G,\B)\)
  and \(\Cred(G,\B)\rtimes_\red S \cong \Cred(G\rtimes S,\B)\).
\end{theorem}

\begin{proof}
  The conditional expectation~\(E\)
  for the \(S\)\nb-action
  on~\(\Cst(G,\B)\)
  on the dense subalgebra
  \(\Sect(G\rtimes S,\B) \subseteq \Cst(G,\B)\rtimes_\alg S\)
  simply restricts a function on~\(G\rtimes S\)
  to~\(G\).
  This is a map to~\(\Sect(G,\B)\)
  if~\(G\)
  is closed in~\(G\rtimes S\).  (This works also for non-saturated
  Fell bundles, as considered
  in~\cite{Deaconu-Kumjian-Ramazan:Fell_groupoid_morphism}.)
  If~\(G\) is closed in~\(G\rtimes S\),
  then \(E(\Cred(G,\B)\rtimes_\red S)\subseteq \Cred(G,\B)\)
  and therefore any faithful representation of~\(\Cred(G,\B)\)
  is \(E\)\nb-faithful.
  Now Lemma~\ref{lem:regular_E-faithful} finishes the proof.
\end{proof}

\begin{remark}
  Conversely, if~\(G\)
  is not closed in~\(G\rtimes S\),
  then \(\tilde{A}\neq A\)
  because we may find \(\xi\in\Sect(G\rtimes S)\)
  for which \(E(\xi)\)
  lives on a single bisection of~\(G\)
  and is not continuous.  And such a function cannot belong
  to~\(\Cred(G,\B)\).
\end{remark}

\begin{remark}
  \label{rem:G_not_closed_but_Hausdorff}
  It may happen that~\(G\) is not closed in~\(G\rtimes S\) although
  \(G\rtimes S\) is Hausdorff.  For instance, let~\(S\) be the
  inverse semigroup of bisections of a non-Hausdorff \'etale
  groupoid~\(H\) with Hausdorff, locally compact unit space~\(H^0\),
  and let~\(G\) be the \v{C}ech groupoid of a Hausdorff open cover
  of the arrow space of~\(H\) (see
  \cite{Buss-Meyer:Actions_groupoids}*{Example~A.9}).  A canonical
  action of~\(S\) on~\(G\) that corresponds to the left translation
  action of a groupoid on its arrow space is described
  in~\cite{Buss-Meyer:Actions_groupoids}*{Corollary~3.21}, and it is
  shown that the transformation groupoid~\(G\rtimes S\) is Morita
  equivalent to the space~\(H^0\), which is the orbit space of the
  action of~\(H\) on~\(H^1\) by left multiplication.  Thus
  \(G\rtimes S\) is a free and proper groupoid, forcing it to have
  Hausdorff arrow space (see
  \cite{Buss-Meyer:Actions_groupoids}*{Proposition~A.7}).
  There is a canonical open and continuous functor \(G\rtimes S\to
  H\), such that the arrow space of~\(G\) is the preimage of the
  unit subspace \(H^0\subseteq H^1\).  Since \(H^0\subseteq H^1\) is
  open, but not closed, the preimage \(G^1\subseteq (G\rtimes S)^1\)
  is open, but not closed.

  Theorem~\ref{the:conditional_expect_groupoid_crossed} does not
  apply to this example.  But Theorem~\ref{the:iterated_crossed_1}
  does.  Indeed, the \v{C}ech groupoid~\(G\) is amenable, say,
  because it is étale and its \cstar{}algebra is nuclear (see
  \cite{Brown-Ozawa:Approximations}*{Theorem~5.6.18}).  Hence it is
  inner exact.  The transformation groupoid \(G\rtimes S\) is also a
  \v{C}ech groupoid and hence amenable; even more, it is equivalent
  to the space~\(H^0\).  Hence Theorem~\ref{the:iterated_crossed_1}
  (or Corollary~\ref{cor:exact=>iterated-crossed-product}) implies
  \(\Cred(G)\rtimes_\red S\cong \Cred(G\rtimes S) \cong
  \Cst(G\rtimes S)\sim \Cont_0(H^0)\).  The isomorphism
  \(\Cst(G)\rtimes S \cong \Cst(G\rtimes S)\) is shown
  in~\cite{Buss-Meyer:Actions_groupoids}.  Hence we get
  \(\Cst(G)\rtimes S \cong \Cred(G)\rtimes_\red S\) in this example.
\end{remark}

\end{document}